\documentclass[aos]{imsart}

\RequirePackage{amsthm,amsmath,amsfonts,amssymb}
\RequirePackage[authoryear]{natbib}
\RequirePackage[colorlinks,citecolor=blue,urlcolor=blue]{hyperref}
\RequirePackage{graphicx}

\usepackage{booktabs}
\usepackage[cmyk]{xcolor}
\usepackage{threeparttable}
\usepackage{float}
\usepackage{amsmath}
\usepackage{cleveref}
\usepackage{bbm}
\usepackage{hyperref}
\usepackage{mathabx}
\usepackage{enumitem}
\usepackage{chapterbib}

\newcommand{\argmin}{\ensuremath{\operatornamewithlimits{argmin}}}

\newcommand{\bSigma}{\boldsymbol \Sigma}

\newcommand{\bPhi}{\boldsymbol \Phi}

\newcommand{\bxi}{\boldsymbol \xi}

\newcommand{\bfeta}{\boldsymbol \eta}

\newcommand{\bvarepsilon}{\boldsymbol \varepsilon}

\newcommand{\be}{{\mathbf e}}
\newcommand{\bbf}{{\mathbf f}}
\newcommand{\bx}{{\mathbf x}}
\newcommand{\by}{{\mathbf y}}

\newcommand{\bv}{{\mathbf v}}

\newcommand{\bA}{{\bf A}}
\newcommand{\bB}{{\bf B}}

\newcommand{\bI}{{\bf I}}

\newcommand{\bS}{{\bf S}}

\newcommand{\bX}{{\bf X}}
\newcommand{\bY}{{\bf Y}}
\newcommand{\bZ}{{\bf Z}}
\newcommand{\bR}{{\bf R}}
\newcommand{\bU}{{\bf U}}
\newcommand{\bV}{{\bf V}}
\newcommand{\bQ}{{\bf Q}}

\newcommand{\bM}{{\bf M}}

\newcommand{\bG}{{\bf G}}

\newcommand{\cC}{{\cal C}}

\newcommand{\cM}{{\cal M}}
\newcommand{\cG}{{\cal G}}

\newcommand{\cT}{{\cal T}}

\newcommand{\cS}{{\cal S}}

\newcommand{\cH}{{\cal H}}
\newcommand{\cN}{{\cal N}}
\newcommand{\cF}{{\cal F}}

\newcommand{\cP}{{\cal P}}

\newcommand{\eZ}{\mathbb{Z}}
\newcommand{\eR}{\mathbb{R}}

\newcommand{\eS}{\mathbb{S}}

\newcommand{\eN}{\mathbb{N}}

\newcommand{\cov}{\text{Cov}}
\newcommand{\var}{\text{Var}}

\newcommand{\du}{{\rm d}u}
\newcommand{\dv}{{\rm d}v}

\newcommand{\pth}[1]{\left( #1 \right)}
\newcommand{\qth}[1]{\left[ #1 \right]}
\newcommand{\sth}[1]{\left\{ #1 \right\}}
\newcommand{\nth}[1]{\left\| #1 \right\|}

\def\T{{ \mathrm{\scriptscriptstyle T} }}

\startlocaldefs
\theoremstyle{plain}

\newtheorem{theorem}{Theorem}[section]
\newtheorem{lemma}[theorem]{Lemma}
\newtheorem{proposition}[theorem]{Proposition}
\theoremstyle{definition}
\newtheorem{definition}[theorem]{Definition}
\newtheorem{example}[theorem]{Example}

\newtheorem{condition}[theorem]{Condition}
\newtheorem{remark}[theorem]{Remark}

\endlocaldefs

\begin{document}

\begin{frontmatter}
\title{Convergence of covariance and spectral density estimates for high-dimensional functional time series}

\begin{aug}
\author[A]{\fnms{Bufan}~\snm{Li}\ead[label=e1]{lbf23@mails.tsinghua.edu.cn}},
\author[B]{\fnms{Xinghao}~\snm{Qiao}\ead[label=e2]{xinghaoq@hku.hk}},
\author[A]{\fnms{Weichi}~\snm{Wu}\ead[label=e3]{wuweichi@mail.tsinghua.edu.cn}}
\and
\author[C]{\fnms{Holger}~\snm{Dette}\ead[label=e4]{holger.dette@rub.de}}
\address[A]{Department of Statistics and Data Science, Tsinghua University, Beijing, P.R. China, \\\printead[presep={\ }]{e1,e3}}

\address[B]{Faculty of Business and Economics, The University of Hong Kong, Hong Kong SAR, \printead[presep={\ }]{e2}}

\address[C]{Fakultät für Mathematik, Ruhr-Universität Bochum, Bochum, Germany, \printead[presep={\ }]{e4}}
\end{aug}

\begin{abstract}
Second-order characteristics including covariance and spectral density functions are fundamentally important for both statistical applications and theoretical analysis in functional time series. In the high-dimensional setting where the number of functional variables is large relative to the length of functional time series, non-asymptotic theory for covariance function estimation has been developed for Gaussian and sub-Gaussian functional linear processes. However, corresponding non-asymptotic results for high-dimensional non-Gaussian and nonlinear functional time series, as well as for spectral density function estimation, are largely unexplored. In this paper, we introduce novel functional dependence measures, based on which we establish systematic non-asymptotic concentration bounds for estimates of (auto)covariance and spectral density functions in high-dimensional and non-Gaussian settings. We then illustrate the usefulness of our convergence results through two applications to dynamic functional principal component analysis and sparse spectral density function estimation. To handle the practical scenario where curves are discretely observed with errors, we further develop convergence rates of the corresponding estimates obtained via a nonparametric smoothing method. Finally, extensive simulation studies are conducted to corroborate our theoretical findings.
\end{abstract}

\begin{keyword}[class=MSC]
\kwd[Primary ]{62M10}
\kwd{62R10}
\kwd[; secondary ]{62M15}
\end{keyword}

\begin{keyword}
\kwd{Discretely observed functional data}
\kwd{Dynamic FPCA}
\kwd{Functional dependence measure}
\kwd{High-dimensional time series}
\kwd{Nagaev-type concentration inequality}
\kwd{Second-order statistics}
\end{keyword}

\end{frontmatter}

\section{Introduction}
\renewcommand{\theequation}{\thesection.\arabic{equation}}

The analysis of functional time series (i.e., time series of random functions defined on a compact interval) has attracted considerable attention in both time series and functional data analysis. Recent advances in data collection technology have led to the increasing prevalence of multivariate and high-dimensional functional time series across various applications.  Examples include cumulative intraday return trajectories \cite[]{horvath2014testing} for a large number of stocks, age-specific mortality rates \cite[]{tang2022}, yield curves \cite[]{hays2012functional} across multiple countries, hourly readings of PM~2.5 concentrations from different monitoring locations \cite[]{tan2024}, daily energy consumption curves from a collection of households \citep{chang2024b}, or human movement data of  individuals \cite[]{bastian2024multiplechangepointdetection}, to list a few. These data can be represented by  $p$-dimensional vectors 
$\bX_1 , \dots  ,\bX_n $ of the form 
$$
\bX_t=\sth{ (X_{t1}(u),\dots,X_{tp}(u))^\T,~~u \in [0,1] },
$$
corresponding to a (stationary) functional time series $ (\bX_t )_{t \in \mathbb{Z}  } $. In the high-dimensional setting, the dimension $p$ is large compared  to the length of functional time series $n$, and may even exceed it.

Estimating second-order characteristics of such processes is fundamentally important in time series and functional data analysis, and for theoretical analysis within the high-dimensional learning framework, it is essential to perform non-asymptotic analysis by deriving relevant concentration inequalities for Hilbert space-valued random elements with temporal dependence. 
\cite{fang2022} and \cite{guo2023consistency} made the first attempts to develop such bounds for estimators of  (auto)covariance functions of Gaussian and sub-Gaussian functional linear processes. The effects of temporal dependence on their results are quantified through a functional stability measure proposed in these references, which, however, lacks an explicit representation  within  the Hilbert space. Moreover, non-asymptotic analysis of other second-order statistics, such as estimates of spectral density functions and the spectral-domain applications, remains largely unexplored. 
Therefore, it is of particular interest to ask: 
\begin{itemize}
\item Is it possible to define functional dependence measures that can on the one hand be easily controlled and employed to establish concentration results for estimates of second-order characteristics, including (auto)covariance and spectral density functions, and can on the other  hand be flexible enough to accommodate non-Gaussian and nonlinear functional processes?

\item How can spectral concentration results be effectively applied to spectral-based methods for high-dimensional functional time series and be adapted to the practical scenario where  random functions are discretely observed with errors?
\end{itemize}

In this paper, we provide affirmative answers to these questions by addressing key theoretical gaps. Our main contributions are as follows.

First, we introduce novel functional dependence measures that offer new insights into how  temporal dependence affects non-asymptotic behaviors for estimators of second-order characteristics in high-dimensional functional time series. Although our work is inspired by the physical dependence adjusted norms for vector-valued scalar time series recently introduced in \cite{Zhang2021ConvergenceOC}, developing the corresponding dependence measures under the functional domain is far from incremental, as the infinite-dimensionality of functional times series introduces significant complexities for characterizing temporal dependence. Instead of discretizing functional data and operating on the maximum difference between two discretized objects  using the approach of \cite{Zhang2021ConvergenceOC} followed by aggregation, we rely on the $L_2$ norm of the difference between coupled curves. As a consequence, the proposed measures effectively capture the intrinsic functional nature of the data. 
Unlike the stability measure in \cite{guo2023consistency}, our dependence measures can be explicitly bounded for a general class of stationary functional processes.

Second, we conduct a systematic non-asymptotic analysis of the second-order statistics by developing concentration inequalities for the estimates of (auto)covariance and spectral density functions. These non-asymptotic results are not only of independent interest but also yield corresponding elementwise maximum rates of convergence, thereby providing foundational theoretical tools for downstream covariance-based and spectral-based high-dimensional learning tasks. Our concentration inequalities are of Nagaev-type, and relax the commonly imposed assumptions of Gaussianity and sub-Gaussianity for functional linear processes in the existing literature. To overcome the complexities introduced by the infinite-dimensionality of functional time series, we employ martingale inequality within the general Banach space in the proofs of our main results.

Third, we demonstrate the value and impact of  our non-asymptotic results in the context of spectral-based estimation for high-dimensional functional time series through two concrete applications. 
Given the infinite-dimensionality of functional time series, it is standard to reduce each function  to a finite set of scalars by principal component analysis (PCA). The common concepts are functional PCA (FPCA) which is applied to the estimated covariance functions, and dynamic FPCA which accounts for the temporal dependence in the data and is applied to the estimated spectral density functions \cite[see, e.g.,][]{hormann2015}.  These frameworks are then used in subsequent regularized estimation to tackle high-dimensionality. In our first application, we investigate the convergence properties of the estimated quantities within this dynamic FPCA framework. The second application involves the thresholded estimation of the matrix of spectral density functions without dimension reduction, under a functional sparsity assumption. This approach can be used to identify pairs of functional time series that are uncorrelated across all lags.

Finally, we address the practical scenario of discretely observed functional time series by employing the local linear smoothing method to obtain (auto)covariance and spectral density function estimators. To the best of our knowledge, such problems have only been studied in an asymptotic framework for univariate functional time series \cite[]{rubin2020}, and in this paper we develop for the first time  non-asymptotic theory for both marginal- and cross-estimators in high-dimensional functional times series, which is particularly relevant to practical applications.

Our work lies at the intersection of high-dimensional and functional time series, both of which have been extensively studied. We focus our review on the literature most relevant to the present context.
Alongside the aforementioned theoretical advancements for high-dimensional functional time series, recent years have seen a surge in various estimation and inference approaches. Notable developments include functional clustering \cite[]{tang2022}, functional vector autoregressions \cite[]{chang2024a}, functional factor models \cite[]{hallin2023,tavakoli2023,guo2025,li2025arxiv}, statistical inference for mean functions \cite[]{ZhouDette2023}, detection and estimation of structural breaks \cite[]{li2024}, graphical principal component analysis \cite[]{tan2024} and functional decorrelation and prediction \cite[]{chang2024b}. 
Additionally, there is a wealth of literature on non-asymptotic theory and different regularized estimators of the (auto)covariance matrix, spectral density matrix and its inverse in high-dimensional time series, see, e.g., \cite{chen2013,basu2015a,chang2018,fiecas2019spectral,Zhang2021ConvergenceOC} and \cite{barigozzi2024algebraic}. 
Finally, different dependence measures have been proposed in the functional time series literature to control the temporal dependence and establish the asymptotic theory for the estimated (auto)covariance and spectral density functions, see, e.g., strong mixing conditions \cite[]{bathia2010,chen2022}, cumulative mixing conditions \cite[]{panaretos2013} and $L^q$-$m$-approximability \cite[]{hormann2010,hormann2015}.

The remainder of this paper is organized as follows. In \Cref{sec.cov}, we propose novel functional dependence measures, and we use them to  establish non-asymptotic concentration bounds for the estimators of  (auto)covariance and spectral-density functions  in \Cref{sec:spectral}. In  \Cref{Sec.Appli}, we demonstrate the impact of our convergence results through applications to two concrete examples: estimation within the dynamic FPCA framework and  spectral density function estimation under the sparsity assumption. \Cref{sec.partial} develops the corresponding convergence rates for the practical scenario of discretely observed functional time series. 
In \Cref{sec:sim}, numerical studies are carried out to validate the established theoretical results. All technical proofs are relegated to the supplementary material.

\textbf{Notation.} For any positive integer \(n\), we write \([n] = \{1,\cdots,n\}\). For \(x,y\in\mathbb{R}\), we write \(x\vee y = \max(x,y)\) and \(x\wedge y = \min(x,y)\). We use $I(\cdot)$ to denote the indicator function. For two positive sequences \(\{a_n\},\{b_n\}\), we write \(a_n\lesssim b_n\) or $b_n\gtrsim a_n$ if there exists a positive constant \(C\) such that \(a_n\le Cb_n\). We write \(a_n\asymp b_n\) if and only if \(a_n\lesssim b_n\) and \(b_n\lesssim a_n\) hold simultaneously. For any vector \(\bv \in\eR^p\), we let \(|\bv|_\infty = \max_{i}|v_i|\), \(|\bv|_1 = \sum_{i}|v_i|\) and \(|\bv|_2 = \left(\sum_{i} v_i^2\right)^{1/2}\). For any matrix $\bA = (A_{jk})_{j,k\in[p]} \in\eR^{p\times p}$, we let $\|\bA\|_{\max} = \max_{j,k}|A_{jk}|$, $\|\bA\|_{1} = \max_{k}\sum_{j}|A_{jk}|$, $\|\bA\|_{\infty} = \max_{j}\sum_{k}|A_{jk}|$ and $\|\bA\|_2=\rho_{\min}(\bA\bA^\T)$. For any random variable \(X\), we denote its \(L_q\) norm (\(q \ge 1\)) as \(\|X\|_q = E(|X|^q)^{1/q}\). Let \(L_2([0,1])\) be the Hilbert space of square integrable (complex-valued) functions defined on \([0,1]\) equipped with the inner product \(\langle f,g\rangle = \int_{[0,1]}f(u)\overline{g(u)}\mathrm{d}u\) for \(f,g\in L_2([0,1])\) and the induced \(L_2\) norm \(\| f\|_{\cH} = \langle f,f \rangle^{1/2}\).
For vector functions $\mathbf{f} = (f_1,\dots,f_p), \mathbf{g} = (g_1,\dots,g_p) \in \otimes^p L_2([0,1])$, the inner product is defined as $
\langle \mathbf{f}, \mathbf{g} \rangle = \sum_{j=1}^p \int_{[0,1]} f_j(u) \overline{g_j(u)} \, \mathrm{d}u$, 
and the corresponding norm  by $\|\mathbf{f} \|_\mathcal{H} = \langle \mathbf{f}, \mathbf{f} \rangle^{1/2} $.
For any $K\in\eS=L_2([0,1] \times [0,1])$, we also use $K$ to denote the linear operator induced from the kernel function $K,$ i.e., 
for any $f\in L_2([0,1])$, $K(f)(\cdot) = \int_{[0,1]}K(\cdot,v)f(v)\mathrm{d}v\in L_2([0,1])$ 
and denote its Hilbert--Schmidt norm by $\|K\|_{\cS} = \big\{\iint_{[0,1]^2} |K(u,v)|^2\mathrm{d}u\mathrm{d}v\big\}^{1/2}$. For any two  $\bX, \bY \in \otimes^p L_2([0,1]) $, we define $(\bX\otimes\bY^\T)(u,v) = \bX(u)\bY^\T(v),$ where $\otimes$ is the Kronecker product. We use $C, C'$ to denote absolute constants whose values may change from line to line. Constants with a symbolic subscript, such as $C_{\star},C'_{\star},$ are used to indicate that their values depend only on the subscript, and may also vary from line to line.

\section{A dependence measure for functional time series}\label{sec.cov}
\renewcommand{\theequation}{\thesection.\arabic{equation}}

For $t\in[n]:=\{1, \ldots, n\} $ let \(\bX_t(\cdot) = \{X_{t1}(\cdot),\dots,X_{tp}(\cdot)\}^\T\)  
be a vector from a $p$-dimensional stationary functional time series $(\bX_t)_{t \in \mathbb{Z}}$ with mean zero and $(p\times p)$-matrix of  (auto)covariance functions \(\bSigma^{(h)}(u,v) = \{\Sigma^{(h)}_{jk}(u,v)\}_{j,k\in[p]}\) at lag $h \in {\mathbb Z}$ 
and $u,v \in [0,1],$ where \(\Sigma^{(h)}_{jk}(u,v) = \mathrm{cov}\{X_{tj}(u),X_{(t+h)k}(v)\}\). We assume that the functional time series $( \bX_t) _{t \in \mathbb{Z}}$ is defined by the  model \begin{equation}\label{Sec1-1}
\bX_t(u) = \bG(u,\cF_t),\quad t\in[n],~u\in[0,1],
\end{equation}
where \(\mathcal{F}_t = (\dots,\varepsilon_{t-1},\varepsilon_t)\) is a sequence of innovations and \((\varepsilon_t)_{t\in\mathbb{Z}}\) are i.i.d. random elements. Here \(\bG(\cdot,\cdot) = \{G_1(\cdot,\cdot),\dots,G_p(\cdot,\cdot)\}^\T\) is a \(p\)-dimensional measurable functional such that  for given $\cF_t$, the vector function $\bG(\cdot,\cF_t)$ takes values in the Hilbert space $(\otimes^p L_2([0,1]), \otimes^pB)$, where $B$ is the Borel sigma field generated by the norm $\|\cdot\|_\cH$ on $\cH = \otimes^p L_2([0,1])$.

Model  \eqref{Sec1-1} defines a physical representation for  the functional time series. Such modeling approaches have  been frequently used in non-linear time series analysis; see \cite{wu2005nonlinear} for a pioneering work.

To derive non-asymptotic results for (auto)covariance and spectral density estimators in high-dimensional  functional time series, we need to introduce appropriate dependence measures. For this purpose, for $l \le t$, we define \(\cF_{t,\{l\}} = (\cF_{l-1},\varepsilon'_{l},\varepsilon_{l+1}\dots,\varepsilon_t)\) as a coupled version of \(\cF_{t} = (\cF_{l-1},\varepsilon_{l},\varepsilon_{l+1}\dots,\varepsilon_t)\), where \(\varepsilon_l\) in \(\cF_{t}\) is replaced by an independent  copy \(\varepsilon_l'\). 
For a norm $\|\cdot\|_{N_1}$ on $L_2([0,1])$ and a norm $|\cdot|_{N_2}$ on $\eR^p$,  we define a composite norm for an element $\bX = (X_1,\dots,X_p)^\T  \in \otimes^p L_2([0,1])$ by 
         $$
        \big  \|  \bX \big \| _{N_1,N_2}  = \big | \big ( \| X_1 \|_{N_1} , \ldots , \| X_p \|_{N_1} \big )^\T \big |_{N_2},
         $$
which means that we first calculate the norm of each component with respect to $\|\cdot \|_{N_1}$ and then compute the  $| \cdot |_{N_2}$-norm of the resulting vector in $\mathbb{R}^p$. 

\begin{definition}[Functional dependence measures]\label{def1}
         For any \(p\)-dimensional functional stationary process of the form \eqref{Sec1-1}, we define 
         $$\omega_{t,q} = \omega_{t,q} (\bX_t ) = \big \| \|\bG(\cdot ,\mathcal{F}_t) - \bG(\cdot,\mathcal{F}_{t,\{0\}}) \|_{\cH,\infty} \big \|_q
         $$
         if $t \ge 0$, and $\omega_{t,q} = 0$ if $t<0$. The dependence adjusted norm of $\bX_t$ is defined as 
         \begin{equation}
         \label{det1}
            \|\|\bX_1\|_{\cH,\infty}\|_{q,\alpha} = \sup\limits_{m\ge 0} (m+1)^\alpha \Omega_{m,q},
        \end{equation}
        where $ \Omega_{m,q} = \sum_{t = m}^{\infty}\omega_{t,q}$ and $\alpha > 0.$ 
        The dependence adjusted norm of the  \(j\)-th entry of $\bX_t$ is defined by  \begin{equation}\label{det2}
            \|\|X_{1j}\|_{\cH}\|_{q,\alpha} = \sup\limits_{m\ge 0} (m+1)^\alpha \Delta_{m,q,j}
        \end{equation}
        where, $\Delta_{m,q,j} = \sum_{t = m}^{\infty}\delta_{t,q,j}$, \(\delta_{t,q,j} =  \delta_{t,q,j}(\bX_t) = \left\|\left\|G_j(u,\mathcal{F}_t) - G_j(u,\mathcal{F}_{t,\{0\}})\right\|_{\cH}\right\|_q\) if $t \geq 0$ and  $\delta_{t,q,j} =0$ if $t<0$. We define \begin{equation*}
            \bPhi_{q,\alpha}^\bX = \max_{j\in[p]}\|\|X_{1j}\|_{\cH}\|_{q,\alpha}^2,\quad \mathcal{M}_{q,\alpha}^\bX = \|\|\bX_1\|_{\cH,\infty}\|_{q,\alpha}^2
        \end{equation*}
        as the uniform and joint functional dependence measures, respectively.
    \end{definition}
    It is easy to verify that $\bPhi_{q,\alpha}^\bX \leq  \mathcal{M}_{q,\alpha}^\bX,$ and both measures can diverge as $p$ increases. We impose the following condition on finite upper bounds for our functional dependence measures.

    \begin{condition} \label{cond.FDM}
         There exist some constants $q>4$ and $\alpha > 0$ such that $\mathcal{M}_{q,\alpha}^\bX < \infty$.
    \end{condition}

    The dependence adjusted norm in \eqref{det1} and \eqref{det2} can be interpreted as \(q\)-norm ($q$-th moment condition) which additionally takes the temporal dependence of the time series into account. The parameter  \(\alpha\) is used to quantify the strength of temporal dependence. A larger value of \(\alpha\) implies faster decay of tail dependence measures and thus weaker temporal dependence. We emphasize that both functional dependence measures are increasing functions with respect to the parameters  $\alpha$ and  $q$. For the functional dependence measures, $\bPhi_{q,\alpha}^\bX$ evaluates the dependence-adjusted norm for each component  and subsequently takes the maximum value. In contrast, $\cM_{q,\alpha}^\bX$ first computes the maximum and then adjusts for dependence. This indicates that $\cM_{q,\alpha}^\bX$ is influenced by the cross-sectional dependence within $\bX_t$, whereas $\bPhi_{q,\alpha}^\bX$ focuses on the maximum temporal dependence strength of each $X_{tj}$ across $j.$ 
 \smallskip

    Recently \cite{Zhang2021ConvergenceOC} considered a similar concept of dependence for high-dimensional locally stationary scalar time series. The key difference of our approach to this work lies in the fact that \eqref{det1} and \eqref{det2} define a dependence concept for  functional times series. For functional objects, there are multiple ways to define the distance between two observed curves. In \eqref{det1} and \eqref{det2} we  first take the $L_2$ norm between the curve and its coupled version, then plug it in the scalar version of the dependence adjusted norm.  An alternative measure of functional dependence is the functional stability measure considered in \cite{guo2023consistency}, among others, to develop non-asymptotic theory for estimators of covariance functions. This measure is proportional to the functional Rayleigh quotient of the matrix of spectral density functions relative to that of covariance functions evaluated over the interval of frequencies. Explicitly computing bounds for this measure can be very challenging within the infinite-dimensional Hilbert space.

    We conclude with  two  examples illustrating that  nice bounds can be  derived for \eqref{det1} and \eqref{det2}, as well as for $\bPhi_{q,\alpha}^\bX$ and $\cM_{q,\alpha}^\bX$ in Definition \ref{def1}, in  general functional time series models. To make our notation clear, we first define the following functional matrix norm. For norm $\|\cdot\|_{N_1}$ on $\eS = L_2([0,1])\otimes L_2([0,1])$ and norm $\|\cdot\|_{N_2}$ on $\eR^{p\times p}$, we define $\|\bA\|_{N_1,N_2},\bA\in\eS^{p\times p}$ to be \begin{align*}
    \|\bA\|_{N_1,N_2} = \|\widetilde{\bA}\|_{N_2},\quad \mbox{where} \quad \widetilde{A}_{jk} = \|A_{jk}\|_{N_1}.
    \end{align*}

    \begin{example}[Vector functional linear process model]\label{example1}
{\rm 
        We consider the $p$-dimensional functional moving average model of infinite order: 
        \begin{equation}\label{Ex1eq1}
            \bX_t(u) = \sum_{m = 0}^{\infty} \int_0^1 \bA_m(u,v)\bvarepsilon_{t-m}(v)\mathrm{d}v,\quad u\in[0,1],
        \end{equation}
        where \(\bvarepsilon_t(\cdot) = \big (\varepsilon_{t1}(\cdot),\dots,\varepsilon_{tp}(\cdot)\big ) ^\T\), and $\{\varepsilon_{tj}(\cdot):t\in[n],j\in[p]\}$ are i.i.d. random curves with mean zero and finite \(q\)-th moment after taking $L_2$ norm, that is  \(\mu_{q} = E(\|\varepsilon_{tj}\|_\cH^q) < \infty\). Assume \(\bA_m(\cdot,\cdot)\) is a \(p\times p\) matrix function with real-valued functions as  entries. Further assume $\sum_{m=0}^\infty \nth{\bA_m}_{\cS,\infty} < \infty$ such that \eqref{Ex1eq1} converges almost surely, see Lemma 7.1 of \cite{Bbosq1}. Let \(\bA_{m\cdot k}(\cdot,\cdot)\) 
        and  \(\bA_{mj\cdot}(\cdot,\cdot)\)
        be the \(k\)-th column  and $j$-th row   of \(\bA_m(\cdot,\cdot)\), respectively. 
        After some derivations (see Section \ref{SectionC1} in the  Supplementary Material for details), we have there exist positive constants $C_q, C_q'$ such that
        \begin{align*}
             \omega_{t,q} & \le C_q\|\bA_t\|_{\cS,\infty}p^{1/q}\mu_q^{1/q}, 
             \\
             \delta_{t,q,j} & \le C_q'\|\bA_{tj\cdot}\|_{\cS,1}p^{1/q}\mu_q^{1/q}.
        \end{align*}
        Suppose that there exists constants  \(\gamma >1\)  and \( K_p > 0\) such that  \(\|\bA_t\|_{\cS,\infty} \le K_p/(t+1)^\gamma\) for all \(t \ge 0\). Then, with \(\alpha = \gamma -1\), there exists a positive constant $C_{\alpha,q}$ such that $$
       \max_{j\in[p]}\|\|X_{1j}\|_{\cH}\|_{q,\alpha} \le \|\|\bX_1\|_{\cH,\infty}\|_{q,\alpha} < C_{\alpha,q}K_pp^{1/q}\mu_{q}^{1/q}.$$
        }
        \end{example}

        \begin{example}[Vector functional autoregressive model]\label{example2}
        {\rm We focus on the $p$-dimensional functional autoregressive model of order $1$ (noting that models of higher order can be reformulated as an equivalent model of order $1$):
        \begin{equation}\label{armodel}
            \bX_t(u) = \int_0^1 \bA(u,v)\bX_{t-1}(v) \dv + \bvarepsilon_t(u),\quad u\in[0,1],
        \end{equation}
        where the error process $(\bvarepsilon_t)_{t \in \mathbb{Z}}$ and matrix function $\bA(\cdot,\cdot) = \big ( A_{jk}(\cdot,\cdot) \big )_{j,k\in[p]} $ are defined the same way as in Example \ref{example1}. Similar to Section 3.1 of \cite{Bbosq1}, we suppose there exists a positive integer $j$ such that $\|\widetilde{\bA}^j\|_{2} = c < 1$, where $\widetilde{\bA} = \big (\widetilde{A}_{jk} \big )_{j,k\in[p]} \in \mathbb{R}^{p \times p} $ is a matrix with entries  $\widetilde{A}_{jk}=\|A_{jk}\|_\cS$. Define $c' = \max_{k\in[j]}\|\widetilde{\bA}^{k}\|_{2}$ and  $f(\alpha) = \sup_{m\ge 0}(m+1)^\alpha c^{m/j-1}/(1-c^{1/j}) < \infty$. After some derivations (see Section \ref{SectionC2} in the  Supplementary Material for details), we conclude that there exists a positive constant $C_q$ such that \begin{align}\label{ex2eq2}
            \max_{j\in[p]}\|\|X_{1j}\|_{\cH}\|_{q,\alpha} \le \|\|\bX_1\|_{\cH,\infty}\|_{q,\alpha} \le C_qc'f(\alpha)p^{1/2}\mu_q^{1/q}.
        \end{align}
        }
        \end{example}
   
\section{Covariance and spectral density function estimation}
  \renewcommand{\theequation}{\thesection.\arabic{equation}}
  \setcounter{equation}{0}\label{sec:spectral}
   
        In this section we establish non-asymptotic results for the estimates of the (auto)covariance and the spectral density functions of high-dimensional functional time series.
        \subsection{Covariance function estimation}\label{Sec.covestimate}

        Based on the observed data, 
        we can estimate the (auto)covariance function at lag $h$ by its sample version:\begin{equation}\label{eq:3.1eq1}
        \widehat{\bSigma}^{(h)}(u,v) = \frac{1}{n-|h|}\sum_{t = 1}^{n}\bX_{t}(u)\bX_{t+h}(v)^\T,\ |h| = 0,1,\dots,\ u,v \in[0,1],
        \end{equation}
        where we set $\bX_{t+h} = 0$ if $t+h \le 0$ or $t + h > n$.

        \begin{theorem}\label{thm1}
        Assume that $|h| < n/2$  and Condition \ref{cond.FDM} holds. Then, there exists positive constants $C_{q,\alpha},C_\alpha$ and $C_\alpha'$ such that, for any  $x>0,$ \begin{align}
            P\left\{\|\widehat{\bSigma}^{(h)}-\bSigma^{(h)}\|_{\cS,\max}>\cM_{q,\alpha}^\bX x\right\} &\le C_{q,\alpha}x^{-q/2}(\log p)^qD_{n,h} + C_\alpha p^2\exp(-C_\alpha' n x^2 ) ,\label{thm1eq1}\\
            P\left\{\|\widehat{\bSigma}^{(h)}-\bSigma^{(h)}\|_{\cS,\max}> \bPhi_{q,\alpha}^\bX ~x\right\} &\le C_{q,\alpha}x^{-q/2}p^2D_{n,h} + C_\alpha p^2\exp(-C_\alpha' n x^2 ).\label{thm1eq2}
        \end{align}
        where 
    $$
    D_{n,h} = \begin{cases}
      n^{1-q/2}(1+|h|)^{q/4-1}&\text{ if } ~~\alpha > 1/2 - 2/q, \\ 
       n^{1-q/2}(1+|h|)^{q/4-1} + n^{-q/4-\alpha q/2} &    \text{ if } ~~\alpha \le 1/2 - 2/q.
    \end{cases}
    $$
    \end{theorem}

The concentration inequalities in Theorem~\ref{thm1} imply the elementwise maximum rate of convergence for the estimated (auto)covariance function:

\begin{equation}
\label{cov.rate}
\|\widehat{\bSigma}^{(h)}-\bSigma^{(h)}\|_{\cS,\max} = O_P\left[\bPhi_{q,\alpha}^\bX\left\{\Big(\frac{\log p}{n}\Big)^{1/2} + C_\bX p^{4/q}D_{n,h}^{2/q}\right\}\right],
\end{equation}
where 
\begin{equation}\label{eq:constCX}
    C_\bX = \min\sth{1, \frac{\cM_{q,\alpha}^\bX (\log p)^2}{\bPhi_{q,\alpha}^\bX p^{4/q}}}.
\end{equation}
This result plays a crucial role in further convergence analysis of downstream covariance-based learning tasks in high-dimensional settings.
\smallskip

We note that the difference between the estimates \eqref{thm1eq1} and \eqref{thm1eq2} is that, the estimate \eqref{thm1eq1} presents a concentration result that takes cross-sectional dependence into account, while the estimate  \eqref{thm1eq2} focuses on the concentration of each dimension individually and then aggregates these results using Bonferroni correction. Consequently, the rate in \eqref{cov.rate} includes the extra coefficient $C_\bX$ to represent the trade-off between these two approaches. 
We also emphasize that the rate in the estimate  (\ref{cov.rate}) consists of two terms. If there exists a sufficiently large $q$ satisfying Condition~\ref{cond.FDM} such that the first term dominates, this rate simplifies to $O_P\big\{\bPhi_{q,\alpha}^\bX(\log p/n)^{1/2}\big\}$ for  general nonlinear and non-Gaussian functional processes  and  aligns with the rate derived in \cite{guo2023consistency} and \cite{fang2022} for Gaussian or sub-Gaussian functional linear processes. Conversely, under a weaker condition where a relatively small $q$ satisfies Condition~\ref{cond.FDM}, the second term becomes dominant, resulting in a rate that grows polynomially with $p$.
We also note that our established rate is consistent with that specified in Theorem~6.4 of \cite{Zhang2021ConvergenceOC} for scalar time series.

\begin{remark}
{\rm Due to the infinite-dimensional nature of functional data, it is standard to perform FPCA by truncating each $X_{tj}(u)$ to the first $M_j$ terms, such that $X_{tj}(u) \approx \sum_{l=1}^{M_j} \xi_{tjl} \psi_{jl}(u)$, where the functional principal component (FPC) scores $\xi_{tjl}= \langle X_{tj} , \psi_{jl}\rangle$ for $l \in [M_j]$ are mean-zero random variables satisfying $\cov(\xi_{tjl}, \xi_{tjl'})=\vartheta_{jl} I(l \neq l')$ and $\vartheta_{j1} \geq \cdots \geq \vartheta_{jM_j} >0$ are eigenvalues of $\Sigma_{jj}^{(0)}(u,v)$ with the associated eigenfunctions $\psi_{j1}(u), \cdots, \psi_{jM_j}(u).$ By applying techniques similar to those in \cite{guo2023consistency} and leveraging the established rate in \eqref{cov.rate}, it is not difficult to derive the elementwise maximum rates of convergence for the estimated eigenvalue/eigenfunction pairs $\{(\hat\vartheta_{jl}, \hat \psi_{jl}(u))\}_{l \in [M_j]}$ of $\widehat \Sigma_{jj}^{(0)}(u,v)$, as well as the sample autocovariances among the estimated FPC scores $\hat \xi_{tjl}= \langle X_{tj} , \hat \psi_{jl}\rangle$ for $l \in [M_j].$ These results are essential for the subsequent convergence analysis of the FPCA-based regularized estimation in high-dimensional settings.}
\end{remark}

\subsection{Spectral density function  estimation}\label{sec2.3}

        Let \(\mathbbm{i}\) be the imaginary unit with \(\mathbbm{i}^2 = -1\). We define the spectral density function at frequency $\theta\in(0,2\pi]$ as \begin{align*}
            \bbf_\theta(\cdot,\cdot) = \frac{1}{2\pi}\sum_{h = -\infty}^{\infty} \bSigma^{(h)}(\cdot,\cdot)\exp(-\mathbbm{i} h\theta),
        \end{align*}
        where \Cref{lem1} below implies that the series converges if Condition \ref{cond.FDM} holds. 
             
        The spectral density function $\bbf_\theta(\cdot,\cdot)$ extends the concept of spectral density matrix \citep{chang2025statistical} to the functional domain and extends the univariate spectral density function \citep{panaretos2013} to the multivariate setting. Theoretical results on the spectral density function estimation in the  high-dimensional regime  have been scarce. In this section, we use a lag-window type statistic  to estimate the spectral density function and derive the elementwise maximum rate of convergence via non-asymptotic results.

        To be precise, we use the statistic \begin{equation}\label{eq2.3-1}
	   \widehat{\bbf}_\theta(\cdot,\cdot) = \frac{1}{2\pi}\sum_{h = -m_0}^{m_0}K(h/m_0)\widehat{\bSigma}^{(h)}(\cdot,\cdot)\exp(-\mathbbm{i} h\theta)
        \end{equation}
        to estimate \(\bbf_\theta(\cdot,\cdot)\), where $m_0$ is a truncation parameter, and $K(\cdot)$ is a symmetric kernel function supported on the interval  $[-1,1]$, $K(0) = 1$, $\sup_{x\in[0,1]}K(x) \le 1$ and $1-K(x) = O(|x|^\tau)$ for some $\tau >0$. Define \begin{align*}
            R(m_0) = \max\limits_{j,k\in[p]}\left[\sum_{|h|> m_0}\|\Sigma^{(h)}_{jk}\|_{\cS} + \sum_{|h|\le m_0}|1-K(h/m_0)|\|\Sigma^{(h)}_{jk}\|_{\cS}\right]
        \end{align*}
        to quantify the combined truncation and smoothing errors. The following lemma shows this quantity can be  nicely controlled  under our setting.

        \begin{lemma}\label{lem1}
            If Condition \ref{cond.FDM} holds, then
        $R(m_0) \le Cm_0^{-\alpha \tau/(\tau+\alpha)}\bPhi_{2,\alpha}^\bX.$
        \end{lemma}

        The rate in \Cref{lem1} can be improved to $R(m_0) \leq Cm_0^{-\alpha}\bPhi_{2,\alpha}^\bX$ by using   
        flap-top kernels as introduced by    \cite{POLITIS19991}. The main results of this section are concentration inequalities  for the maximum deviation between $\widehat{\bbf}_\theta $ and its expectation.

        \begin{theorem}\label{thm2} 
        Assume $m_0 < n/3$ and Condition \ref{cond.FDM} holds. Then there exist positive constants $C_{q,\alpha},C_\alpha$ and $C_{\alpha}'$ such that \begin{align*}
            P\left\{\sup\limits_{\theta\in[0,2\pi]}\|\widehat{\bbf}_\theta-E(\widehat{\bbf}_\theta)\|_{\cS,\max}> \cM_{q,\alpha}^\bX x\right\} \le & C_{q,\alpha}x^{-q/2}(\log p)^{5q/4}F_{n,m_0} \\
            &+ C_\alpha m_0p^2\exp\left(-\frac{C_\alpha' x^2n}{m_0}\right),\\
            P\left\{\sup\limits_{\theta\in[0,2\pi]}\|\widehat{\bbf}_\theta-E(\widehat{\bbf}_\theta)\|_{\cS,\max}> \bPhi_{q,\alpha}^\bX ~x\right\} \le & C_{q,\alpha}x^{-q/2}p^2F_{n,m_0} \\
            & + C_\alpha m_0p^2\exp\left(-\frac{C_\alpha' x^2n}{m_0}\right),
        \end{align*}
        where $$
        F_{n,m} = \begin{cases}
      n^{1-q/2}m^{q/2}&\text{ if } ~~\alpha > 1/2 - 2/q, \\ 
       n^{1-q/2}m^{q/2} + n^{-q/4-\alpha q/2}m^{q/4+1} &    \text{ if } ~~\alpha \le 1/2 - 2/q.
    \end{cases}
    $$
        \end{theorem}

      Combining   the  concentration inequalities in  \Cref{thm2} with \Cref{lem1} gives  the  convergence rate 
      of  the estimator \eqref{eq2.3-1} for the  spectral density  function, that is  
      \begin{equation*}\sup\limits_{\theta\in[0,2\pi]}\|\widehat{\bbf}_{\theta} - \bbf_{\theta}\|_{\cS,\max} = O_P ({\cal H}_1) 
      \end{equation*}
    where
    \begin{equation}
      {\cal H}_1 = R(m_0) + \bPhi_{q,\alpha}^\bX\left\{C_\bX' p^{4/q}(F_{n,m_0})^{2/q}  + \Big(\frac{m_0\log(p\vee m_0)}{n}\Big)^{1/2}\right\}\label{thm2eq1}
        \end{equation}
        {and 
    \begin{equation*}
        C_\bX' = \min\sth{1, \frac{\cM_{q,\alpha}^\bX (\log p)^{5/2}}{\bPhi_{q,\alpha}^\bX p^{4/q}}}.
        \end{equation*}
       
        The first term in \eqref{thm2eq1} represents the truncation and smoothing errors of the lag-window statistic \eqref{eq2.3-1}. The second and third terms in \eqref{thm2eq1} arise from the concentration inequalities in Theorem \ref{thm2}, and their behavior is determined by the moment and dependence conditions. While increasing $m_0$ decreases the first term, it enlarges the second and third terms. Therefore, the optimal choice of $m_0$ minimizes \eqref{thm2eq1}, achieving a balance between these terms. 

        Convergence results of $\sup_{\theta\in[0,2\pi]}\|\widehat{\bbf}_{\theta} - \bbf_{\theta}\|_{\cS,\max}$ have many applications, such as dynamic FPCA in Section~\ref{sec.fpca} and estimation of the spectral density functions under the sparsity assumption in Section~\ref{sec.sparse_est}. In Section \ref{sec.suppD} of the Supplementary Material, we also derive the convergence results of $\|\widehat{\bbf}_{\theta} - \bbf_{\theta}\|_{\cS,\max}$ at a fixed frequency $\theta\in[0,2\pi]$.

    \section{Applications}\label{Sec.Appli}
\renewcommand{\theequation}{\thesection.\arabic{equation}}
  \setcounter{equation}{0}
        This section presents two applications of the established theoretical results in Section \ref{sec2.3}.

        \subsection{Dynamic FPCA}\label{sec.fpca}

        In this section we consider dynamic FPCA. The \(j\)-th diagonal element of spectral density function $\widehat{\bbf}_\theta = ( f_{\theta ,jk } )_{j,k\in [p]} $  has the eigen-decomposition
        \begin{equation*}
            f_{\theta ,jj}\{x(\cdot)\} = \sum\limits_{m\ge 1}\lambda_{jm}(\theta)\langle x(\cdot),\varphi_{jm}(\cdot ; \theta)\rangle\varphi_{jm}(\cdot ; \theta),
            \end{equation*} 
            where $(\lambda_{jm})_{m\geq 1} $ and $(\varphi_{jm})_{m\geq 1} $ are the eigenvalues and eigenfunctions of $f_{\theta, jj}$. 
            Similarly, for the $j$-th diagonal element of  the  lag-window estimator \eqref{eq2.3-1}, we have             \begin{equation*}
            \hat{f}_{\theta ,jj}\{x(\cdot)\} = \sum\limits_{m\ge 1}\widehat{\lambda}_{jm}(\theta)\langle x(\cdot),\widehat{\varphi}_{jm}(\cdot ; \theta)\rangle\widehat{\varphi}_{jm}(\cdot ; \theta),
        \end{equation*}
        where $(\widehat{\lambda}_{jm})_{m\geq 1} $ and $(\widehat{\varphi}_{jm})_{m\geq 1} $ are the eigenvalues and eigenfunctions of $\hat{f}_{\theta, jj}$.   
        As proposed in \cite{hormann2015}, the \(m\)-th dynamic FPC score of \(X_{tj}\) and its estimate are respectively defined by  \begin{equation}\label{FPCADef3}
            \zeta_{tjm} = \sum\limits_{l\in\mathbb{Z}}\langle X_{(t-l)j},\phi_{jml}\rangle,\quad \widehat{\zeta}_{tjm} = \sum_{l = -L}^{L}\langle X_{(t-l)j}, \widehat{\phi}_{jml}\rangle,
        \end{equation}
        where \(L\) denotes the truncation parameter and \begin{equation*}
            \phi_{jml}(\cdot) = \frac{1}{2\pi}\int_{0}^{2\pi}\varphi_{jm}(\cdot ; \theta)\exp(-\mathbbm{i} l\theta)\mathrm{d}\theta, \quad \widehat{\phi}_{jml}(\cdot) = \frac{1}{2\pi}\int_{0}^{2\pi}\widehat{\varphi}_{jm}(\cdot ; \theta)\exp(-\mathbbm{i} l\theta)\mathrm{d}\theta
        \end{equation*} 
        are the $m$-th (estimated) dynamic FPC filter coefficients.
      
        Note that we should speak of a version of (estimated) eigenfunctions and dynamic FPC filter coefficients and scores. They are not uniquely defined since eigenfunctions are defined up to any multiplicative factor on the complex unit circle. In the following, for each computed $\widehat{\varphi}_{jm}(\cdot;\theta)$, we specify the particular version of $\varphi_{jm}(\cdot;\theta)$ that is being estimated. We first arbitrarily choose an eigenfunction $\varphi_{jm}^*(\cdot ; \theta)$, and then substitute it with \begin{equation}\label{eigenfunctionversion}
            \varphi_{jm}(\cdot;\theta) = \varphi_{jm}^*(\cdot ; \theta)\overline{\langle\varphi_{jm}^*(\cdot ; \theta),\widehat{\varphi}_{jm}(\cdot ; \theta)\rangle}/|\langle\varphi_{jm}^*(\cdot ; \theta),\widehat{\varphi}_{jm}(\cdot ; \theta)\rangle|
        \end{equation}
        if $\langle\varphi_{jm}^*(\cdot ; \theta),\widehat{\varphi}_{jm}(\cdot ; \theta)\rangle \ne 0$. Multiplying $\varphi_{jm}^*(\cdot ; \theta)$ with $\exp(\mathbbm{i} \alpha)$ does not change the right hand side, so $\varphi^*_{jm}$ is unique if $\langle\varphi_{jm}^*(\cdot ; \theta),\widehat{\varphi}_{jm}(\cdot ; \theta)\rangle \ne 0$. If $\langle\varphi_{jm}^*(\cdot ; \theta),\widehat{\varphi}_{jm}(\cdot ; \theta)\rangle = 0$, we take $\varphi_{jm}(\cdot;\theta) = \varphi_{jm}^*(\cdot;\theta)$.
        Note that the procedure in \eqref{eigenfunctionversion} is simpler than the approach in \cite{hormann2015}, which involves choosing $\varphi_{jl}$ and $\widehat{\varphi}_{jm}$ based on reference curves \(v_j(\cdot)\), while still successfully tackles the specification problem.

        A standard procedure to estimate models involving high-dimensional functional time series consists of three steps. Due to the infinite-dimensionality of functional time series, the first step performs dynamic FPCA that converts the problem of modeling $p$-dimensional functional time series to that of modeling vector time series of dynamic FPC scores. The second step implements the regularization methods under certain structural assumptions based on estimated dynamic FPC scores. The third step re-converts the vector estimates obtained in the second step to functional estimates via estimated dynamic FPC filter coefficients obtained in the first step. 
       
        Before presenting convergence rates of relevant estimated quantities to theoretically support the aforementioned three-step procedure, we impose the following eigengap condition.

        \begin{condition}\label{cond4}
        {\rm 
            Let \(\lambda_{j1}(\theta) > \lambda_{j2}(\theta) > \cdots\) be the eigenvalues of \(f_{\theta, jj}\). Denote \(\alpha_{j1}(\theta) = \lambda_{j1}(\theta) - \lambda_{j2}(\theta)\), and \(\alpha_{jm}(\theta) = \min\{\lambda_{jm}(\theta)-\lambda_{j(m-1)}(\theta),\lambda_{j(m+1)}(\theta) - \lambda_{jm}(\theta)\}\) for \(m\ge 2\). There exists an increasing positive sequence of \((\delta_m)_{m\in\mathbb{N}}\) diverging to $\infty$ such that \begin{equation*}
        \inf\limits_{\theta\in[0,2\pi]}\min\limits_{j\in[p]}|\alpha_{jm}(\theta)| \ge \delta_m^{-1}.
	\end{equation*}
    }
        \end{condition}

        In the following theorem, we establish the elementwise maximum rates of convergence for estimated eigenvalues, eigenfunctions, and dynamic FPC filter coefficients, which can be used to provide theoretical guarantees for the first step and the third step under high-dimensional settings.

    \begin{theorem}\label{thm3}
          Assume that the  conditions of  \Cref{thm2} and  Condition \ref{cond4} hold.
          Then for any $M, l\in\eN$, 
          $$
          \|\widehat{\phi}_{jml}-\phi_{jml}\|_\cH \le\sup_{\theta\in[0,2\pi]} \|\widehat{\varphi}_{jm}(\cdot ; \theta) - \varphi_{jm}(\cdot ; \theta)\|_\cH,
          $$
          and  
          $$\max\limits_{j\in[p],m\in[M]}\sup\limits_{\theta\in[0,2\pi]}\left\{|\widehat{\lambda}_{jm}(\theta) - \lambda_{jm}(\theta)| + \|\widehat{\varphi}_{jm}(\cdot ; \theta) - \varphi_{jm}(\cdot ; \theta)\|_\cH/\delta_m\right\} = O_P\left(\cH_1\right), $$
          where $\cH_1$ is defined in \eqref{thm2eq1}.
        \end{theorem}

       Before presenting the convergence properties of estimated dynamic FPC scores, we impose a differentiability condition on the eigenfunctions to control the truncation error in \eqref{FPCADef3}.

    \begin{condition}\label{cond5}
    {\rm 
    There exists an integer $\kappa > 2$ such that a version of $\varphi_{jm}^*(\cdot; \theta)$ is $\kappa$-times differentiable with respect to $\theta$ for all $j\in[p],m\in\eN$ with \begin{equation*}
    \max_{j \in [p],\, m \in \mathbb{N}} \int_{0}^{2\pi} \left\| \frac{\partial^\kappa}{\partial \theta^\kappa} \varphi_{jm}^*(\cdot;\theta) \right\|_{\mathcal{H}}^2 \, d\theta = O(1).
    \end{equation*}
    Furthermore, $\varphi_{jm}^*(\cdot; \theta)$ can be extended to a $2\pi$-periodic function in $\theta$ on $\mathbb{R}$.
     }
        \end{condition}
        
    Under Condition \ref{cond5}, the dynamic FPC score $\zeta_{tjm}$ is 
    $\zeta_{tjm} = \sum_{|l|\le L}\langle X_{(t-l)j},\phi_{jml}\rangle + \sum_{|l| > L}\langle X_{(t-l)j},\phi_{jml}^*\rangle,$ where dynamic FPC filter coefficients $\phi_{jml}$ and $ \phi_{jml}^*$ are respectively computed from the eigenfunctions $\varphi_{jm}(\cdot; \theta)$ in~\eqref{eigenfunctionversion} and $\varphi_{jm}^*(\cdot; \theta)$ in Condition~\ref{cond5}.
    Consider the (auto)covariances at lag $h$ between dynamic FPC scores $\sigma_{jkml}^{(h)} = \cov\sth{\zeta_{tjm},\zeta_{(t+h)kl}}$ and the corresponding estimates $$
    \widehat{\sigma}^{(h)}_{jkml} = { 1 \over n-2L-h} \sum_{t = L}^{n-L-h}\widehat{\zeta}_{tjm}\widehat{\zeta}_{(t+h)kl}.
    $$
    To provide theoretical support for the second step that relies on estimated dynamic FPC scores, we establish the elementwise maximum rate of $\{\widehat{\sigma}^{(h)}_{jkml}\}$ in the following theorem.

        \begin{theorem}\label{thm4}
           Assume that the conditions of Theorem \ref{thm3} and Condition \ref{cond5} hold, \(L < n/4\), and $h$ is fixed. Further assume that $\max_{j\in[p]}\int_0^1\Sigma^{(0)}_{jj}(u,u)\du = O(1),$ $\cH_1 = o(1),$ and $ \cH_1\delta_M = O(1),$ where $\cH_1$ is specified in \eqref{thm2eq1}. Then the estimates \(\{\widehat{\sigma}^{(h)}_{jkml}\}\) satisfy \begin{align*}
                \max\limits_{j,k\in[p]}\max\limits_{m,l\in[M]}\frac{|\widehat{\sigma}^{(h)}_{jkml} - \sigma^{(h)}_{jkml}|}{\delta_m \vee \delta_l} = O_P\left(L^{2-\kappa} + L^2\cH_1+\cH_2\right),
            \end{align*}
            where $\cH_2=\bPhi_{q,\alpha}^\bX C_\bX p^{4/q} L^{2+4/q}n^{-1/2}$  with  $C_\bX$ specified in \eqref{eq:constCX}.
    \end{theorem}
    The convergence rate in \Cref{thm4} comprises  three terms. The first term $L^{2-\kappa}$ depends on the smoothness parameter $\kappa$ with larger values yielding a faster rate, and arises from the truncation error in (\ref{FPCADef3}). The second term $L^2\cH_1$ is due to the dynamic FPC filter coefficients estimation errors, while the third term $\cH_2$ results from the errors in estimating the (auto)covariance functions.

    \subsection{Estimation of sparse spectral density function}\label{sec.sparse_est}

        In this section, we consider estimating the spectral density function $\bbf_{\theta}$ in the high-dimensional regime, where the estimator $\widehat{\bbf}_\theta$ in \eqref{eq2.3-1} is inconsistent. 
        However, the problem of the curse of dimensionality does not exist, if the  ``true''  spectral density function $\bbf_{\theta}$
        satisfies some lower-dimensional structural assumptions.      

        To define such structural assumptions, for an $\bbf_{\theta} = (f_{\theta, jk})_{j,k\in[p]}$ we write $\bbf_{\theta}\succeq 0$ if for each $\theta\in[0,2\pi]$ it is positive semi-definite, i.e., 
        $\sum_{j,k\in[p]}\iint_{[0,1]^2}f_{\theta,jk}(u,v)\overline{a_j(u)}a_k(v)\du\dv \ge 0
        $ for any $a_j, a_k \in L_2([0,1]).$ 

     \begin{definition}
     \label{defsparse}
        For  $0\le q^* < 1$ we define  \begin{equation*}
            \cC\big\{q^*,s_0(p)\big\} = \Big\{\bbf_{\theta}:\bbf_\theta\succeq 0,~\max\limits_{k\in[p]}\sum_{j = 1}^p\sup\limits_{\theta\in[0,2\pi]}\|f_{\theta, jk}\|^{q^*}_\mathcal{S} \le s_0(p)\Big\}
        \end{equation*}
        as the class of approximately sparse spectral density  functions (uniformly over all frequencies). 
    \end{definition}
    
        In the special case $q^* = 0$, under the convention $0^0 = 0$,  $\cC\{q^*,s_0(p)\}$ corresponds to the class of truly sparse spectral density functions. If $f_{\theta,jk} = 0$ for all $\theta\in[0,2\pi]$, it implies total linear uncorrelatedness between the $j$-th and $k$-th components of functional time series $\bX_t$ across all lags. 
        
        To estimate $\bbf_\theta$ in the sense of Definition \ref{defsparse}, we use the following threshold estimator 
        \begin{equation}\label{det12}
        \widehat{\bbf}_{\theta}^{\cT} = (\hat{f}_{\theta, jk}^{\cT})_{j,k\in[p]}\quad\mbox{with} \quad\hat{f}_{\theta, jk}^{\cT} =
        \hat{f}_{\theta, jk} \Big (1-{ \lambda \over  \sup_{\theta\in[0,2\pi]}\|\hat{f}_{\theta ,jk}\|_\cS} \Big )_+,
        \end{equation}
        where $\lambda > 0$ is the thresholding parameter and  
        $(x)_+:= \max(0,x)$ for any $x \in {\mathbb R}$.

        \begin{theorem}\label{thm5}
        Assume all conditions in Theorem \ref{thm2} hold and $\lambda^{-1}\cH_1 = o(1),$ where \(\cH_1 \) is specified in \eqref{thm2eq1}. Then uniformly on $\cC\{q^*,s_0(p)\}$, \begin{align*}
                &\sup_{\theta\in[0,2\pi]}\|\widehat{\bbf}_{\theta}^{\cT}-\bbf_{\theta}\|_{\cS,1}  = \sup_{\theta\in[0,2\pi]}\max\limits_{k\in[p]}\sum_{j = 1}^p \|\hat{f}_{\theta, jk}^{\cT}-f_{\theta, jk}\|_{\cS}\nonumber = O_P\Big\{s_0(p)\lambda^{1-q^*}\Big\}.
        \end{align*}
        \end{theorem}
        Theorem~\ref{thm5} established the uniform convergence rate, over all frequencies, for the threshold estimator of spectral density function  in the functional analogue of matrix $\ell_1$ norm.

        We finally turn to investigate the support recovery  properties of the estimator $\widehat{\bbf}_{\theta}^{\cT}$ over the class of truly sparse spectral density functions defined as      
           \begin{align*}
            \cC\{s_0(p)\} &= \Big \{\bbf_{\theta}:\bbf_{\theta}\succeq 0, \max\limits_{k\in[p]}\sum_{j = 1}^p I \Big (\sup\limits_{\theta\in[0,2\pi]}\|f_{\theta, jk}\|_\cS\ne 0 \Big )\le s_0(p)\Big \}.
        \end{align*}
        Define the support of $\bbf_\theta$ as 
        $$
        \mathrm{supp}(\bbf_\theta) = \Big\{(j,k):\sup_{\theta\in[0,2\pi]}\|f_{\theta ,jk}\|_\cS > 0\Big\}.
        $$

        \begin{theorem}\label{thm6}
            Assume all conditions in Theorem \ref{thm5} hold and $\sup_{\theta\in[0,2\pi]}\|f_{\theta ,jk}\|_\cS > \lambda$ for all $(j,k)\in\mathrm{supp}(\bbf_\theta)$. Then we have \begin{align*}
                \inf\limits_{\bbf_\theta\in\cC\{s_0(p)\}}P\left\{\mathrm{supp}(\widehat{\bbf}_{\theta}^{\cT})=\mathrm{supp}(\bbf_{\theta})\right\}\rightarrow 1~~\mathrm{as}~~n\rightarrow\infty.
            \end{align*}
        \end{theorem}

        Theorem \ref{thm6} implies that $\widehat{\bbf}_{\theta}^{\cT}$ can recover the support of truly sparse spectral density functions with probability approaching one, provided that the signal strength is sufficiently large.

        \begin{remark}
          {\rm 
          \Cref{thm5} and \ref{thm6} remain correct for a general class of threshold estimators, see Section~\ref{Sec.A.5} of the Supplementary Material for details.  
          See also similar results for the threshold estimation of the sparse spectral density function at a fixed frequency in Section \ref{Sec.Add_Sparse} of the Supplementary Material. 
            }
        \end{remark}

    \section{Discretely observed functional time series}\label{sec.partial}
\renewcommand{\theequation}{\thesection.\arabic{equation}}
  \setcounter{equation}{0}
  
        In this section, we consider the practical scenario where curves are discretely observed with errors. For each $t\in[n]$ and $j\in[p]$, suppose \(X_{tj}(u)\) is observed with errors at \(T_{tj}\) random time points \(U_{tj1},\dots,U_{tjT_{tj}}\in[0,1]\). Let \(Y_{tji}\) be the observed value of \(X_{tj}(U_{tji})\) satisfying \begin{align*}
            Y_{tji} = X_{tj}(U_{tji})+\varepsilon_{tji},\quad i=1,\dots,T_{tj},
        \end{align*}
        where the random errors \(\varepsilon_{tji}\)'s, independent of $X_{tj}$'s, are i.i.d. with \(E(\varepsilon_{tji}) = 0\) and \(\var(\varepsilon_{tji}) = \sigma_j^2<\infty\).
        
        For densely observed curves with $T_{tj}$'s larger than some order of $n,$ it is conventional to implement local linear smoothing to the observations from each curve, thereby producing reconstructed curves that can be used to compute the second-order statistics as in Section \ref{Sec.covestimate} and \ref{sec2.3}. 
        In what follows, denote  \(K_b = K(\cdot/b)/b\) for a univariate kernel with bandwidth $b>0.$ For each $t,j$ and $u$, the estimation of \(X_{tj}(u)\) is attained via \(\widehat{X}_{tj}(u) = \hat{a}_0(u)\), where \begin{equation*}
            \big\{\hat{a}_0(u),\hat{a}_1(u)\big\} = \argmin\limits_{a_0(u),a_1(u)}\frac{1}{T_{tj}}\sum_{i = 1}^{T_{tj}}\big\{Y_{tji} - a_0(u) - a_1(u)(U_{tji}-u)\big\}^2K_{b_{j}}(U_{tji}-u).
        \end{equation*} 
       
       For any $j \in [p],$ individual functions often exhibit similar smoothness properties and sometimes similar shapes. Therefore, we use the same bandwidth $b_{j}$ for all of them. Denote the reconstructed curves by $\widehat{\bX}_t = (\widehat{X}_{t1},\dots,\widehat{X}_{tp})^\T$. We then estimate the (auto)covariance function at lag $h$ by $$ \widetilde{\bSigma}^{(h)}(u,v) = \frac{1}{n-|h|}\sum\limits_{t =1}^{n}\widehat{\bX}_t(u)\widehat{\bX}_{t+h}^\T(v),\ h=0,1,\dots,\ (u,v)\in[0,1]^2,$$
        where we set $\widehat{\bX}_{t+h} = 0$ if $t+h \le 0$ or $t+h>n$.

    Before presenting the convergence results, we impose some regularity conditions.

        \begin{condition}\label{condpartial1}
~~ \begin{itemize}       \item[(\romannumeral 1)] The errors $\{\varepsilon_{tji}\}$ are sub-Gaussian random variables, i.e., there exists some positive constant $C$ such that for all $z \in {\mathbb R}$, $E\{\exp(\varepsilon_{tji}z)\}\le \exp(C^2\sigma_j^2z^2/2)$ and  $\max_j\sigma_j=O(1)$.                          \item[(\romannumeral 2)] The kernel $K(\cdot)$ is a symmetric probability density function on support \([-1,1]\) and is Lipschitz continuous, i.e., there exists some positive constant \(C\) such that \(|K(u)-K(v)| \le C|u-v|\) for any $u,v\in[-1,1].$

            \item[(\romannumeral 3)] The observational points \(\{U_{tji}:t\in[n],j\in[p],i\in[T_{tj}]\}\) are i.i.d. copies of a random variable $U$ defined on \([0,1]\) with density \(f_U(\cdot)\) satisfying \(0 < m_f \le \inf_{u\in[0,1]}f_U(u) \le \sup_{u\in[0,1]}f_U(u) \le M_f < \infty\). Moreover, $\{\bX_t\},\{U_{tji}\}$ and $\{\varepsilon_{tji}\}$ are mutually independent.          \item[(\romannumeral 4)] There exists a sufficiently large   positive constant $C$ such that $\max_{t,j} T_{tj}(\min_{t,j} T_{tj})^{-1} \le C$ and $\max_{j} b_{j}(\min_{j} b_j)^{-1} \le C$. For each $j\in [p],$ the average sampling frequency $\widebar{T}_j = (n^{-1} \sum_{t=1}^n T_{tj}^{-1})^{-1} \rightarrow \infty, b_{j} \rightarrow 0$ and $\widebar T_j b_{j} \rightarrow \infty$ as $n \rightarrow \infty.$
\item[    (\romannumeral 5)] For each $t\in[n],j, k\in[p]$, $X_{tj}(u)$ is twice continuously differentiable, and $\Sigma^{(h)}_{jk}(u,v)$ is twice continuously differentiable over $u,v\in[0,1]$.
            \item[(\romannumeral 6)] For each $t \in [n],$ 
            $E(|\bX_{t}^*|_\infty^2) \lesssim E(\|\bX_{t}\|_{\cH,\infty}^2)$ and $E(|\bX_{t}^{(2)*}|_\infty^2) \lesssim E(\|\bX_{t}\|_{\cH,\infty}^2),$ where $\bX_t^* = (X_{t1}^*,\dots,X_{tp}^*)^\T$ and $\bX_{t}^{(2)*} = (X_{t1}^{(2)*},\dots,X_{tp}^{(2)*})^\T,$ with 
            $X_{tj}^* = \sup_{u\in[0,1]}|X_{tj}(u)|$ and $X_{tj}^{(2)*} = \sup_{u\in[0,1]}|\partial^2_u X_{tj}(u)|$ for each $j \in [p]$.
            \end{itemize}
        \end{condition}

        Condition~\ref{condpartial1}(\romannumeral 1) imposes the sub-Gaussianity on the random errors. Conditions~\ref{condpartial1}(\romannumeral 2)--(\romannumeral 5) are standard in the literature of local linear smoothing for functional data \cite[]{yao2005functional,zhang2007statistical} adapted to the multivariate setting. Condition~\ref{condpartial1}(\romannumeral 6) requires that the supremum norm of the original curve and its twice differential is not significantly larger than the $L_2$ norm of the original curve. This condition rules out irregular cases in which the curves and their second derivatives exhibit extreme spikes. We give two examples that satisfy Condition~\ref{condpartial1}(\romannumeral 6).

        \begin{example}\label{example4}
            Assume that the curves $X_{tj}(u)$'s are three times  continuously differentiable and $E(\|\partial^\kappa_u\bX_t\|_{\cH,\infty}^2) \lesssim E(\|\bX_t\|_{\cH,\infty}^2)$ for $\kappa=1,2,3$. Then $\bX_t$ satisfies Condition \ref{condpartial1}(\romannumeral 5).
        \end{example}

        \begin{example}\label{example3}
            For each $t \in [n],j \in [p],$ consider the Karhunen–-Loève expansion $X_{tj}(\cdot) = \sum_{l=1}^\infty \xi_{tjl}\psi_{tjl}(\cdot)$ with $\xi_{tjl}=\langle X_{tj},\psi_{tjl}\rangle$ and  $\var(\xi_{tjl})=\vartheta_{jl}$ for $l=1,2,\dots.$ Assume that there exists a large constant $C$ such that $E(\max_{j\in[p]} \xi_{tjl}^2) \le C E(\max_{j\in[p]} \xi_{tj1}^2)$ for $l = 1,2,\dots$. Similar to Assumption 4 in \cite{zhou2022theory}, assume that $\psi_{tjl}$'s are twice continuously differentiable and $\sup_{u\in[0,1]}|\psi_{tjl}(u)| \asymp l^{\tilde{\delta}_1}, \sup_{u\in[0,1]}|\partial^2_u\psi_{tjl}(u)| \asymp l^{\tilde{\delta}_2}$ for some $\tilde{\delta}_1, \tilde{\delta}_2>0.$ 
            Then if $\vartheta_{tjl}\asymp \tilde{\delta}^{-l}$ for some $\tilde{\delta} > 1$, $\bX_t$ satisfies Condition \ref{condpartial1}(\romannumeral 5).
        \end{example}

        \begin{theorem}\label{thm7}
            Assume that all conditions in Theorem \ref{thm1} and Condition \ref{condpartial1} hold, and $\log (p\vee n)/\min_j\widebar{T}_j b_{j}\rightarrow 0.$ Then we have $$\|\widetilde{\bSigma}^{(h)}-\bSigma^{(h)}\|_{\cS,\max} = O_P(\cH_3+\cH_4),$$
        where \(\cH_3\) is specified in \eqref{cov.rate} for the fully observed case, and 
        $$
        \cH_4 = (\Omega_{0,2})^2\{\max_{j}b_{j}^2 + (\min_{j}\widebar{T}_j b_{j})^{-1/2}(\log p)^{1/2}\}. 
        $$
        \end{theorem}

    \begin{remark}\label{remark9} 
        {\rm 
        To facilitate further discussion, we consider a simplified scenario where $b_{j} \asymp b$ and $ T_{tj}\asymp T$ for all $t\in[n],j\in[p].$ To balance the variance and bias terms, we can choose the optimal bandwidth $b \asymp T^{-1/5}(\log p)^{1/5}$ and the convergence rate in Theorem \ref{thm7} is reduced to 
            $$\|\widetilde{\bSigma}^{(h)}-\bSigma^{(h)}\|_{\cS,\max} = O_P\qth{\bPhi_{q,\alpha}^\bX\sth{\Big(\frac{\log p}{n}\Big)^{1/2} + C_\bX p^{4/q} D_{n,h}^{2/q}}+ \Omega_{2,0}^2 \Big(\frac{\log p}{T}\Big)^{2/5}} .$$
            Compared to the fully observed case, the additional term $(\Omega_{2,0})^2 T^{-2/5}(\log p)^{2/5}$ in the rate arises from the local linear smoothing step, and is proportional to the optimal rate in \cite{zhang2007statistical} up to a factor of $(\log p)^{2/5}(\Omega_{2,0})^2$ due to the high-dimensional effect and temporal dependence. This rate exhibits an interesting phase transition phenomenon depending on the relative order of $T$ to $n.$ When $T$ grows very fast, the resulting rate $O_P\pth{\cH_3}$ coincides with that of the fully observed case, implying that the theory for very dense functional time series falls within the parametric paradigm. When $T$ grows moderately fast, it leads to a slower rate \(O_P\{(\Omega_{2,0})^2 T^{-2/5}(\log p)^{2/5}\}\).
        }
    \end{remark}

    We finally establish the convergence rate of spectral density function estimator based on discrete observations, denoted by \begin{equation}\label{part.spec}
        \widetilde{\bbf}_{\theta}(\cdot,\cdot) = \frac{1}{2\pi}\sum_{h = -m_0}^{m_0}K(h/m_0)\widetilde{\bSigma}^{(h)}(\cdot,\cdot)\exp(-\mathbbm{i} h\theta).
    \end{equation}

    \begin{theorem}\label{thm8}
        Assume that all conditions in Theorem \ref{thm2} and Condition \ref{condpartial1} hold, and $\log (p\vee n)/\min_j\widebar{T}_j b_{j}\rightarrow 0$. Then we have $$\sup\limits_{\theta\in[0,2\pi]}\|\widetilde{\bbf}_{\theta}-\bbf_\theta\|_{\cS,\max} = O_P(\cH_1+\cH_5),$$
        where \(\cH_1\) is specified in \eqref{thm2eq1} for the fully observed case, and 
        $$
        \cH_5 = m_0\Omega_{0,2}^2\{\max_{j}b_{j}^2 + (\min_{j}\widebar{T}_j b_{j})^{-1/2}(\log p)^{1/2}\}.
        $$
        \end{theorem}

        \section{Simulations}
    \label{sec:sim}
    \renewcommand{\theequation}{\thesection.\arabic{equation}}
  \setcounter{equation}{0}
  
    In this section, we carry out simulations to validate our established theoretical results in Section \ref{sec:spectral}, \ref{Sec.Appli} and \ref{sec.partial}. 
    We generate $p$-dimensional functional time series by  
    \begin{align}
        X_{tj}(u) = \sum_{l=1}^4 (3l/2)^{-1/2} \xi_{tjl} \psi_{l}(u), ~~t \in [n], ~j \in [p], ~u\in[0,1], \label{eq6.1}
    \end{align}
    where $\{\psi_{1}(u),\psi_{2}(u),\psi_{3}(u),\psi_{4}(u)\} = \sqrt{2}\{\sin(2\pi u), \cos(2\pi u),\sin(4\pi u),\cos(4\pi u)\}$ are the basis functions. For each $l\in[4]$, the basis coefficient vector $\bxi_{tl} = (\xi_{t1l},\dots,\xi_{tpl})$ follows a vector autoregressive model $\bxi_{tl} = \rho \bA \bxi_{(t-1)l} + \bfeta_{tl}$, where \(\rho \in (0,1)\) and \(\{\bfeta_{tl}\}\) are independently sampled from a \(p\)-dimensional random vector with each entry following an independent \( t \)-distribution with 6 degrees of freedom. The matrix $\bA = \bI_{p/50} \otimes\bA_0$ with $\bA_0 = \bv_1\bv_1^\T/|\bv_1|_2 + \bv_2\bv_2^\T/|\bv_2|_2$, where the $j$-th entries of $\bv_1$ and $\bv_2 \in {\mathbb R}^{50}$ are
    respectively $1$ and $\cos(2j)$ for $j\in[50]$. It can be easily verified that \(\cM_{q,\alpha}(\bX) \lesssim p^{1/q}\) with $q = 5$ for any \(\alpha > 0\).

    \subsection{Fully observed functional scenario}\label{sec6.1}

    \begin{figure}[thbp]
    \centering
    \includegraphics[width=1\linewidth]{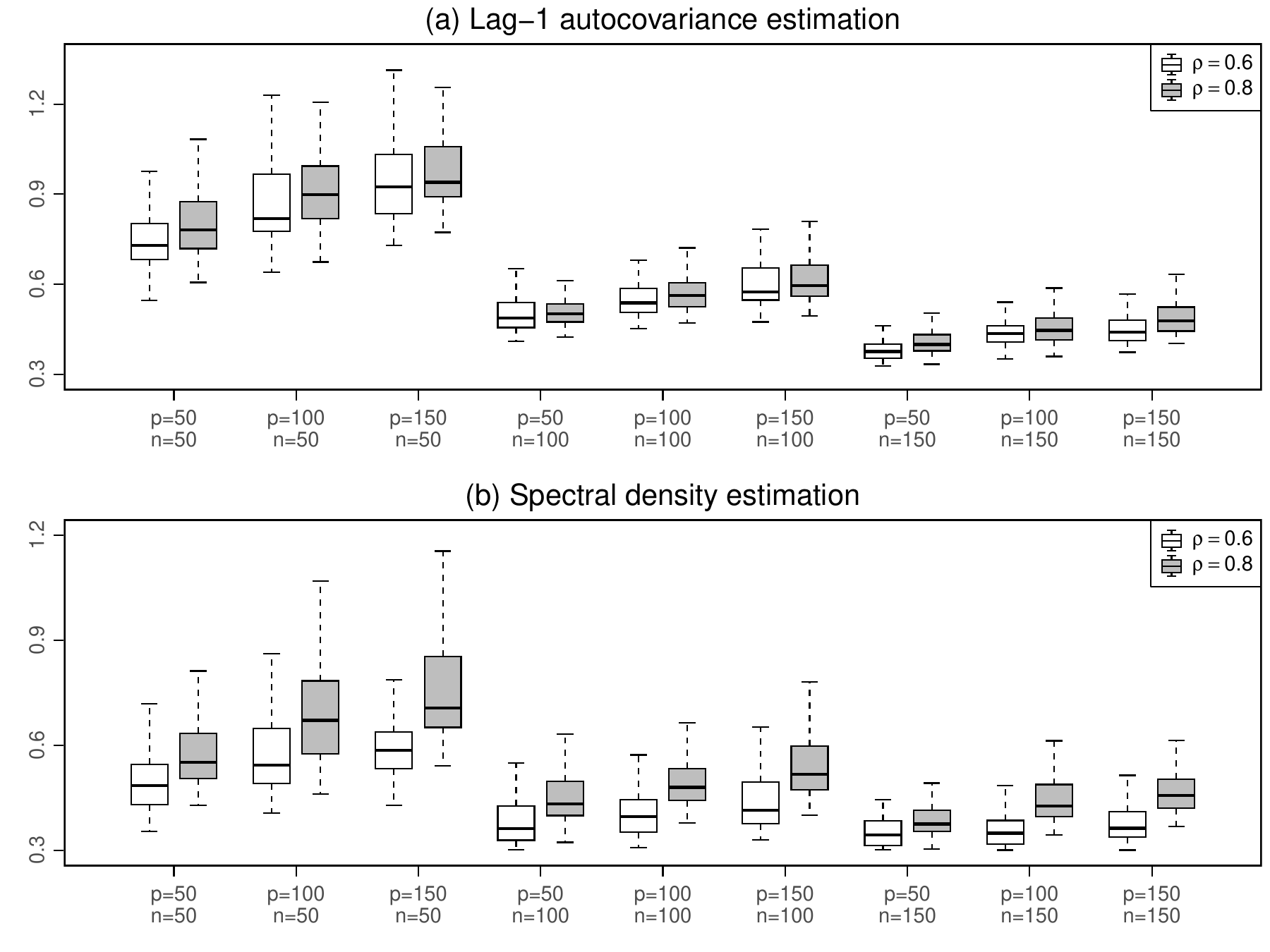}
    \caption{Boxplots of $\mathrm{MaxErr}(\widehat{\bSigma})$ and $\mathrm{MaxErr}(\widehat{\bbf_\theta})$.}
        \label{fig:1}
    \end{figure}

    We first focus on fully observed functional times series generated by \eqref{eq6.1} and evaluate the finite-sample performance of the estimated lag-1 autocovariance function $\widehat{\bSigma}^{(1)}$ in \eqref{eq:3.1eq1} and the estimated spectral density function $\widehat{\bbf}_\theta$ in \eqref{eq2.3-1}. We use the rectangular kernel and choose the truncation parameter \( m_0 = \lceil \log n \rceil \) when obtaining $\widehat{\bbf}_\theta$. We consider settings of $n \in \{50, 100, 150\}$, $p \in \{50, 100, 150\}$, and $\rho \in \{0.6, 0.8\}$. The estimation accuracy is evaluated in terms of the following elementwise maximum estimation errors, 
    \begin{align*}
    \mathrm{MaxErr}(\widehat{\bSigma}) &= \big\|\widehat{\bSigma}^{(1)} - \bSigma^{(1)}\big\|_{\cS,\max} ~~ \text{and}~~\mathrm{MaxErr}(\widehat{\bbf_\theta}) = \sup_{\theta \in [0, 2\pi]} \big\|\widehat{\bbf}_{\theta} - \bbf_{\theta}\big\|_{\cS,\max}.
    \end{align*}

    Figure~\ref{fig:1} displays boxplots of $\mathrm{MaxErr}(\widehat{\bSigma}^{(1)})$ and $\mathrm{MaxErr}(\widehat{\bbf_\theta})$ based on 100 simulation runs. Some patterns are observable. First, as $p$ increases, the functional dependence measure $\cM_{5,\alpha}(\bX) \lesssim p^{1/5}$ grows relatively slowly, and the estimation errors exhibit a modest upward trend, which is consistent with the convergence rates established in Theorems~\ref{thm1} and~\ref{thm2}. Second, as the strength of temporal dependence increases (i.e., as $\rho$ varies from $0.6$ to $0.8$), the estimation performance deteriorates.

    \subsection{Discretely observed functional scenario}

    \begin{figure}[htbp]
        \centering
        \includegraphics[width= 1\linewidth]{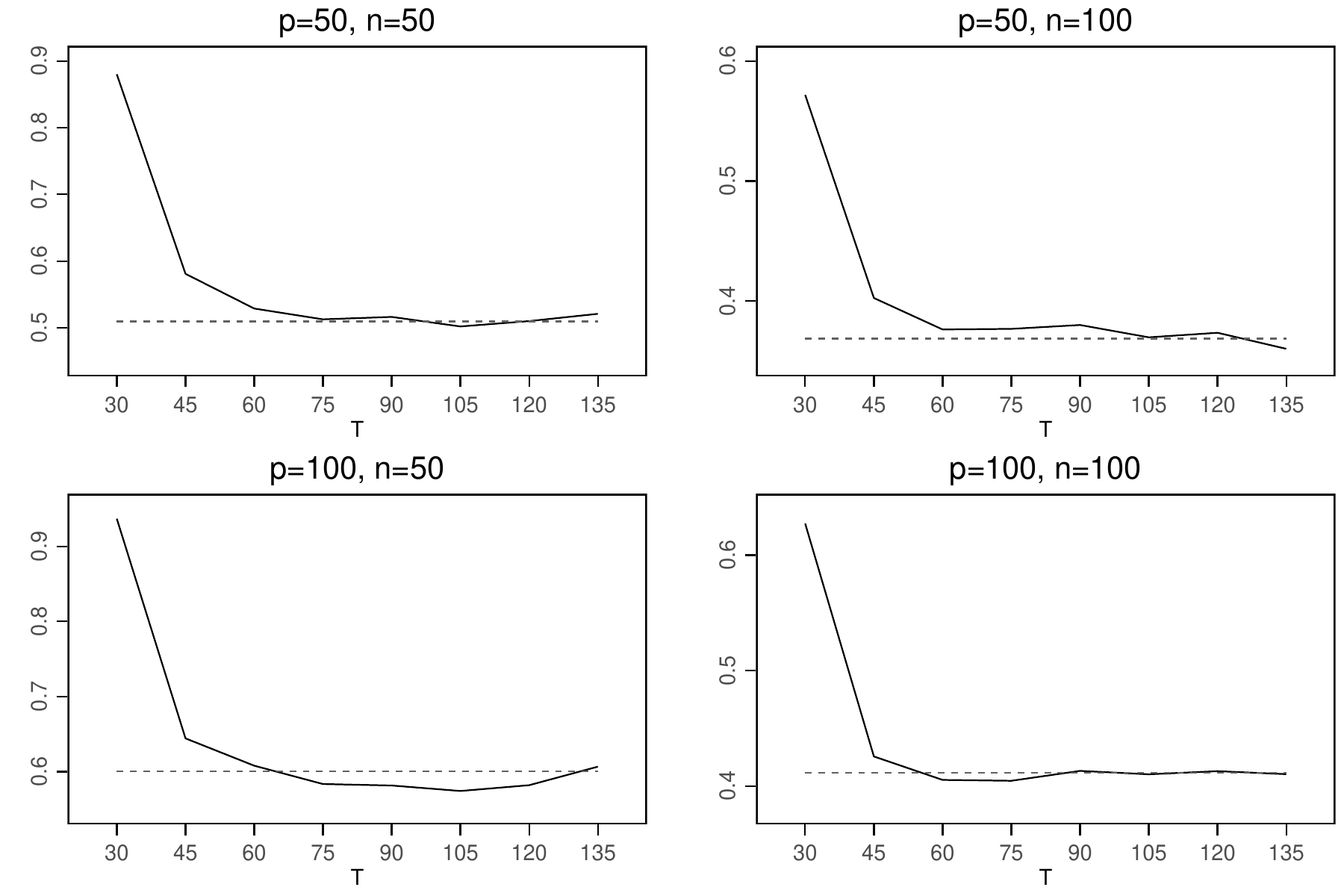}
        \caption{Plots of average $\mathrm{MaxErr}(\widetilde{\bbf}_\theta)$ against $T$ under discretely observed scenario (black solid) and under fully observed scenario (gray dashed).}
        \label{fig:2}
    \end{figure}

    This section considers the practical scenario where the functional time series $X_{tj}(\cdot)$ in \eqref{eq6.1} are discretely observed with errors. Specifically, we generate the observed values  $Y_{tji} = X_{tj}(U_{tji})+\varepsilon_{tji},$ for $t\in[n],j\in[p]$ and $i\in[T],$ where the time points $U_{tji}$'s and the errors \(\varepsilon_{tji}\)'s are sampled independently from $\mathrm{Unif}(0,1)$ and $\cN(0,4)$ respectively. We examine settings of \(n \in\sth{50,100}, p\in\sth{50,100}, \rho = 0.6,\) and \( T\in\{30, 45, \dots, 135\}\). For simplicity, we assess the performance of the spectral density function estimator $\widetilde{\bbf}_\theta$ in \eqref{part.spec} using the minimum elementwise maximum estimation error across a grid of candidate bandwidths in the prespecified set $\cH_b,$ whose elements are proportional to $(T ^{-1}\log p)^{1/5}:$ 
    \begin{align*}
    \mathrm{MaxErr}(\widetilde{\bbf}_\theta) = \min_{b\in\cH_b} \sup_{\theta \in [0, 2\pi]} \big\|\widetilde{\bbf}_{\theta} - \bbf_{\theta}\big\|_{\cS,\max}.
    \end{align*}
    It is noted that the bandwidth selection here is to corroborate the theoretical results in Section~\ref{sec.partial}. In practice when $\bbf_\theta$ is unknown, one may adopt the standard cross-validation method to select the optimal bandwidth. 

    Figure \ref{fig:2} plots the averages of $\mathrm{MaxErr}(\widetilde{\bbf}_\theta)$ over $100$ simulation runs. As $T$ increases, we observe a sharp decline in the averages of $\mathrm{MaxErr}(\widetilde{\bbf}_\theta)$ followed by a plateau, aligning well with the result under the fully observed case. This trend provides empirical evidence supporting the occurrence of a phase transition from the \textit{moderate dense} to the \textit{very dense} regime, as discussed in Remark \ref{remark9}.

    \subsection{Application to dynamic FPCA}

    \begin{figure}[htbp]
        \centering
        \includegraphics[width=1.0\linewidth]{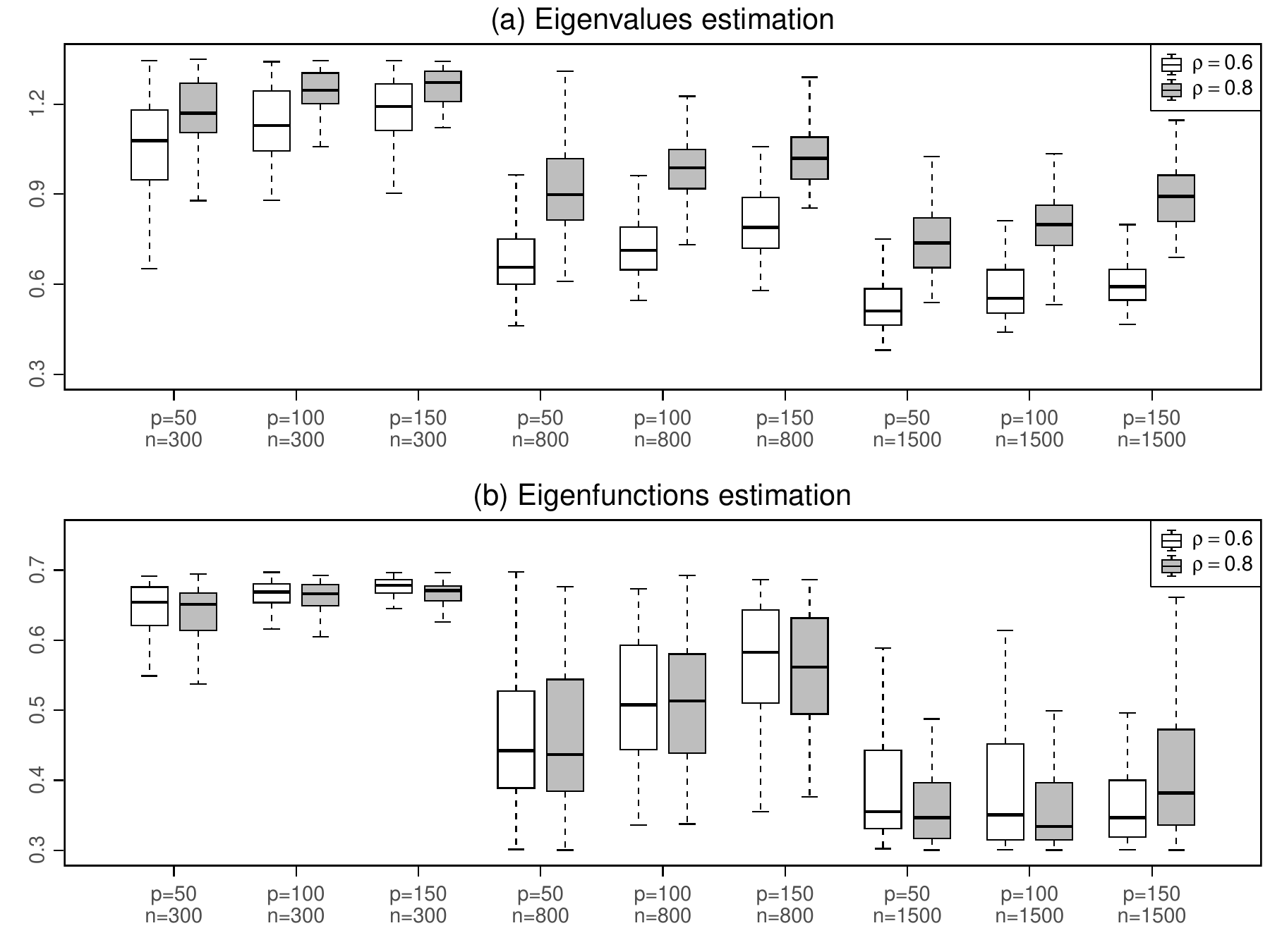}
        \caption{Boxplots of  $\mathrm{MaxErr}(\widehat{\lambda})$ and $\mathrm{MaxErr}(\widehat{\varphi})$.}
        \label{fig:3}
    \end{figure}

    We finally conduct simulations to validate the established theory for dynamic FPCA in Theorem \ref{thm3}. We focus on the estimated spectral density function $\widehat{\bbf}_\theta$ for fully observed functional time series as in \Cref{sec6.1}. We consider settings of $n\in\sth{300,800,1500}, p\in\sth{50,100,150},$ and $\rho\in\sth{0.6,0.8}$. 

    The accuracy of the estimated eigenvalues and eigenfunctions is quantified via the following elementwise maximum estimation errors: 
    \begin{align*}
    \mathrm{MaxErr}(\widehat{\lambda}) &= \max_{j\in[p],m \in [4]} \sup\limits_{\theta \in [0, 2\pi]}  |\widehat{\lambda}_{jm}(\theta) - \lambda_{jm}(\theta)|,\\
    \mathrm{MaxErr}(\widehat{\varphi}) &= \max_{j\in[p],m \in [4]} \sup\limits_{\theta \in [0, 2\pi]} \frac{\|\widehat{\varphi}_{jm}(\cdot ; \theta) - \varphi_{jm}(\cdot ; \theta)\|_\mathcal{H}}{\delta_m}.
    \end{align*}
    Similar to the findings in Section \ref{sec6.1}, we observe in Figure \ref{fig:3} that the errors decrease markedly as $n$ enlarges, while they exhibit a slight increase with larger values of $p.$ 


\begin{funding}
Li and Wu have been supported by NSFC 12271287. Qiao has been partially supported by the Seed Fund for Basic Research for New Staff at the University of Hong Kong. Dette has been supported by the Deutsche Forschungsgemeinschaft (DFG) TRR 391 {\it Spatio-temporal Statistics for the Transition of Energy and Transport}, project number 520388526 (DFG).
\end{funding}


\begin{supplement}
\textbf{Supplement to “Convergence of Covariance and Spectral Density Estimates for High-dimensional Functional Time Series”:} We present the proofs of all theorems and technical lemmas, additional derivations and results in the Supplementary Material.
\end{supplement}


\bibliographystyle{imsart-nameyear} 
\bibliography{ref}       

\clearpage

    \appendix
	\renewcommand{\bibname}{refa}
	\linespread{1.0}\selectfont
	\begin{center}
		{\noindent \bf \large Supplementary material  to ``Convergence of covariance and spectral density estimates for high-dimensional functional time series}\\
	\end{center}
	\begin{center}
		{\noindent Bufan Li, Xinghao Qiao, Weichi Wu and Holger Dette}
	\end{center}
	\bigskip
	
	\setcounter{page}{1}
	\setcounter{section}{0}
	\renewcommand\thesection{\Alph{section}}
	
	This supplementary material contains the proofs of main theorems in Section \ref{supsec.mainpf}, technical lemmas and their proofs in Section \ref{supsec.tech}, derivations in examples in Section \ref{supsec.exmcal} and some additional results in Section \ref{sec.suppD} and \ref{Sec.Add_Sparse}.

    For any random variable $X$, we write $E_0(X) = X - E(X)$.

        \section{Proofs of Main Theorems}\label{supsec.mainpf}

 \renewcommand{\theequation}{\thesection.\arabic{equation}}
  \setcounter{equation}{0}
  \subsection{Proof of Theorem \ref{thm1}}

        We prove the results for $h \ge 0$ while $h < 0$ can be similarly dealt with. For the first concentration inequality, using Lemma \ref{nonGaussianConcen2}, we have \begin{align*}
            P\left\{(n-h)\|\widehat{\bSigma}^{(h)}-E(\widehat{\bSigma}^{(h)})\|_{\cS,\max}> x\right\} \le & C_{q,\alpha}x^{-q/2}(\log p)^q\|\|\bX_1\|_{\cH,\infty}\|_{q,\alpha}^qD'_{n,h} \\
            &+ C_\alpha p^2\exp\left\{-\frac{x^2}{C_\alpha\left(\bPhi_{4,\alpha}^\bX\right)^2n}\right\}.
        \end{align*}
        Here $D'_{n,h}=n(1+|h|)^{q/4-1}$ (resp., $n(1+|h|)^{q/4-1}+n^{q/4-\alpha q/2}$) if $\alpha > 1/2-2/q$ (resp., $\alpha \le 1/2-2/q$). Notice that $n/2\le n-h\le n$ and $\cM_{q,\alpha}^\bX = \|\|\bX_1\|_{\cH,\infty}\|_{q,\alpha}^2 \ge \bPhi_{4,\alpha}^\bX$. This gives \begin{align*}
            P\left\{\|\widehat{\bSigma}^{(h)}-E(\widehat{\bSigma}^{(h)})\|_{\cS,\max} > \cM_{q,\alpha}^\bX x\right\} \le C_{q,\alpha}x^{-q/2}(\log p)^qD_{n,h} + C_\alpha p^2\exp\left\{-C_\alpha' nx^2\right\}.
        \end{align*}
        Here $D_{n,h}=n^{1-q/2}(1+|h|)^{q/4-1}$ (resp., $n^{1-q/2}(1+|h|)^{q/4-1}+n^{-q/4-\alpha q/2}$) if $\alpha > 1/2-2/q$ (resp., $\alpha \le 1/2-2/q$).

        For the second concentration inequality, using Lemma \ref{lm:elementwisecon2} and noticing that for any $j\in[p]$, $\|\|X_{1j}\|_{\cH,\infty}\|_{q,\alpha} \le(\bPhi_{q,\alpha}^\bX)^{1/2}$, we have \begin{align*}
           &P\left\{(n-h)\|\widehat{\Sigma}^{(h)}_{jk}-E(\widehat{\Sigma}^{(h)}_{jk})\|_{\cS,\max}> x\right\} \\
           &\le C_{q,\alpha}x^{-q/2}\|\|X_{1j}\|_{\cH,\infty}\|_{q,\alpha}^{q/2}\|\|X_{1k}\|_{\cH,\infty}\|_{q,\alpha}^{q/2}D'_{n,h} + C_\alpha \exp\left\{-\frac{x^2}{C_\alpha\left(\bPhi_{4,\alpha}^\bX\right)^2n}\right\}\\
            \le & C_{q,\alpha}x^{-q/2}(\bPhi_{q,\alpha}^\bX)^{q/2}D'_{n,h} + C_\alpha \exp\left\{-\frac{x^2}{C_\alpha\left(\bPhi_{4,\alpha}^\bX\right)^2n}\right\},
        \end{align*} 
        where $D_{n,h}'$ is defined in the same way. By $\bPhi_{q,\alpha}^\bX\ge\bPhi_{4,\alpha}^\bX$ and Bonferroni inequality, we have \begin{align*}
        P\left\{\|\widehat{\bSigma}^{(h)}-E(\widehat{\bSigma}^{(h)})\|_{\cS,\max} > \bPhi_{q,\alpha}^\bX x\right\} \le C_{q,\alpha}x^{-q/2}p^2D_{n,h} + C_\alpha p^2\exp\left\{-C_\alpha' nx^2\right\}.
        \end{align*}
        Here $D_{n,h}$ is defined in the same way.

        \subsection{Proof of Theorem \ref{thm2}}

        \subsubsection{Proof of the first concentration inequality}
        
        The estimator \(\widehat{\bbf}_\theta\) can be decomposed as $$2\pi n \widehat{\bbf}_\theta = \sum_{l = -m_0}^{m_0}\frac{n}{n-|l|}K(l/m_0)\exp(-\mathbbm{i} l\theta)\left\{\sum\limits_{t}\bX_t \otimes\bX_{t+l}^\T\right\}=\bQ_1(\theta) + \bQ_2(\theta),$$
        where we define \(\bQ_1(\theta) = \sum_{1\le s < t\le n}a_{st}(\theta)\bX_s\otimes\bX_{t}^\T , \bQ_2(\theta) = \sum_{1\le t\le s\le n}a_{st}(\theta)\bX_s\otimes\bX_{t}^\T \), and \(a_{st}(\theta) = K\{(t-s)/m_0\}\exp\{-\mathbbm{i} (t-s)\theta\}n/(n-|t-s|)\). Recall the notation that for any random variable $X$, $E_0(X) = X-E(X)$. Notice that \begin{equation*}
            \begin{aligned}
                \big\|E_0\big(\widehat{\bbf}_{\theta, jk}\big)\big\|_{\cS} & = \left\|\sum_{l = -m_0}^{m_0} \frac{n}{n-|l|}\exp(-\mathbbm{i} l\theta) E_0\left(\sum_{t} X_{tj}\otimes X_{(t+l)k}\right)\right\|_{\cS}\\
                & = \left[\sum_{l_1,l_2=-m_0}^{m_0} \exp\left\{-\mathbbm{i}(l_1-l_2)\theta\right\}\frac{n^2A_{l_1,l_2}}{(n-l_1)(n-l_2)}\right]^{1/2}
            \end{aligned}
        \end{equation*}
        with $A_{l_1,l_2} = \int_{[0,1]^2}E_0\{\sum_{t}X_{tj}(u)X_{(t+l_1)k}(v)\}E_0\{\sum_{t}X_{tj}(u)X_{(t+l_2)k}(v)\}\du\dv$. This shows $\|E_0(\widehat{\bbf}_{\theta, jk})\|_\cS^2$ is a (random and real) trigonometric polynomial of $\theta$ of order $2m_0$. Denote \(\theta_h = \pi h/(4m_0),0\le h\le 8m_0,h\in\mathbb{N}\). By Lemma D.1 in \cite{Zhang2021ConvergenceOC}, we have for any $j,k\in[p]$ \begin{equation*}
            \begin{aligned}
                \sup_{\theta\in[0,2\pi]}\big\|E_0\big(\widehat{\bbf}_{\theta, jk}\big)\big\|_{\cS} \le \sqrt{2}\max\limits_{h\in[8m_0]}\big\|E_0\big(\widehat{\bbf}_{\theta_h ,jk}\big)\big\|_{\cS}.
            \end{aligned}
        \end{equation*}
        Using Bonferroni inequality, we have \begin{align}\label{pfthm2-1}
            &P\left\{2\pi n\sup_{\theta \in[0,2\pi]} \big\|E_0\big(\widehat{\bbf}_{\theta}\big)\big\|_{\cS,\max} > x\right\} \\
            &\le (8m_0+1)\sum_{\omega= 1}^{2}\max\limits_{h\in[8m_0]}P\left[\left\|E_0\left\{\bQ_\omega\left(\theta_h\right)\right\}\right\|_{\cS,\max} >\frac{x}{2\sqrt{2}}\right].\notag
        \end{align}
        Notice that $|a_{st}(\theta)|\le 3$. It is easy to verify that Lemma \ref{nonGaussianConcen} also holds when coefficients of the quadratic form is upper bounded by $3$. We obtain \begin{equation*}
            \begin{aligned}
                P\left[\left\|E_0\left\{\bQ_\omega\left(\theta_h\right)\right\}\right\|_{\cS,\max} > x\right] \le  & C_{q,\alpha}x^{-q/2}(\log p)^{5q/4}\|\|\bX_1\|_{\cH,\infty}\|_{q,\alpha}^qF_{n,m_0}' \\
                &+ C_\alpha p^2\exp\left\{-\frac{x^2}{C_\alpha\left(\bPhi_{4,\alpha}^\bX\right)^2nm_0}\right\}.
            \end{aligned}
        \end{equation*}
        Here $F'_{n,m_0}=nm_0^{q/2-1}$ (resp., $nm_0^{q/2-1}+n^{q/4-\alpha q/2}m_0^{q/4}$) if $\alpha > 1/2-2/q$ (resp., $\alpha \le 1/2-2/q$). Noticing that $\bPhi_{4,\alpha}^\bX \le \cM_{q,\alpha}^\bX , \|\|\bX_1\|_{\cH,\infty}\|_{q,\alpha}^2 = \cM_{q,\alpha}^\bX $, elementary calculation gives \begin{equation*}
            \begin{aligned}
                P\left\{\sup_{\theta \in[0,2\pi]} \big\|E_0\big(\widehat{\bbf}_{\theta}\big)\big\|_{\cS,\max} > \cM_{q,\alpha}^\bX  x\right\} \le & C_{q,\alpha}x^{-q/2}(\log p)^{5q/4}F_{n,m_0} \\
                & + C_\alpha m_0p^2\exp\left(-\frac{C_\alpha' x^2n}{m_0}\right).
            \end{aligned}
        \end{equation*}
        where \(F_{n,m_0} = n^{1-q/2}m_0^{q/2}\) (resp., \(n^{1-q/2}m_0^{q/2} + n^{-q/4-\alpha q/2}m_0^{q/4+1}\)) if \(\alpha > 1/2 - 2/q\) (resp., \(\alpha \le 1/2-2/q\)).

        \subsubsection{Proof of the second concentration inequality}

        By Lemma \ref{lm:elementwisecon1}, for $w\in[2]$ and $j,k\in[p]$ we have \begin{align*}
            P\left[\left\|E_0\left\{Q_{\omega jk}\left(\theta_h\right)\right\}\right\|_{\cS} > x\right] \le  & C_{q,\alpha}x^{-q/2}\|\|\bX_{1j}\|_{\cH,\infty}\|_{q,\alpha}^{q/2}\|\|\bX_{1k}\|_{\cH,\infty}\|_{q,\alpha}^{q/2}F_{n,m_0}' \\
                &+ C_\alpha \exp\left\{-\frac{x^2}{C_\alpha\left(\bPhi_{4,\alpha}^\bX\right)^2nm_0}\right\}.
        \end{align*}
        Here $F'_{n,m_0}=nm_0^{q/2-1}$ is defined in the same way as the previous section. Notice that for any $j\in[p]$, $\|\bX_{1j}\|_{\cH,\infty}\|_{q,\alpha} \le (\bPhi_{q,\alpha}^\bX)^{1/2}$ and $\bPhi_{4,\alpha}^\bX \le \bPhi_{q,\alpha}^\bX$. Using Bonferroni inequality, we have \begin{align*}
            P\left[\left\|E_0\left\{\bQ_{\omega}\left(\theta_h\right)\right\}\right\|_{\cS,\max} > x\right] \le  & C_{q,\alpha}x^{-q/2}p^2(\bPhi_{q,\alpha}^\bX)^{q/2}F_{n,m_0}'\\
            & + C_\alpha p^2 \exp\left\{-\frac{x^2}{C_\alpha\left(\bPhi_{q,\alpha}^\bX\right)^2nm_0}\right\}.
        \end{align*}
        And elementary calculation gives \begin{equation*}
            \begin{aligned}
                P\left\{\sup_{\theta \in[0,2\pi]} \big\|E_0\big(\widehat{\bbf}_{\theta}\big)\big\|_{\cS,\max} > \bPhi_{q,\alpha}^\bX  x\right\} \le C_{q,\alpha}x^{-q/2}p^2F_{n,m_0} + C_\alpha m_0p^2\exp\left(-\frac{C_\alpha' x^2n}{m_0}\right).
            \end{aligned}
        \end{equation*}
        where \(F_{n,m_0}\) is defined in the same way as the previous section.

        \subsubsection{Control of the truncation error}

        We also need to control the truncation error induced by $m_0$, and the kernel smoothing error. We have \begin{equation*}
            \begin{aligned}
            &\sup\limits_{\theta\in[0,2\pi]}\big\|E\big(\widehat{\bbf}_{\theta}\big)-\bbf_{\theta}\big\|_{\cS,\max} = \sup\limits_{\theta\in[0,2\pi]}\max\limits_{j,k\in[p]} \big\|E\big(\hat{f}_{\theta, jk}\big)-f_{\theta, jk}\big\|_\cS\\
            & \le \sup\limits_{\theta\in[0,2\pi]}\max\limits_{j,k\in[p]}\left[\left\|\sum_{|h|>m_0}e^{\mathbbm{i} h\theta}\Sigma_{jk}^{(h)}\right\|_\cS\ + \left\|\sum_{|h|\le m_0}\left\{1-K\left(h/m_0\right)\right\}e^{\mathbbm{i} h\theta}\Sigma_{jk}^{(h)}\right\|_\cS\right]\\
            & \le \sup\limits_{\theta\in[0,2\pi]}\max\limits_{j,k\in[p]}\left[\sum_{|h|>m_0}\left\|e^{\mathbbm{i} h\theta}\Sigma_{jk}^{(h)}\right\|_\cS + \sum_{|h|\le m_0}\left\|\left\{1-K\left(h/m_0\right)\right\}e^{\mathbbm{i} h\theta}\Sigma_{jk}^{(h)}\right\|_\cS\right]\\
            & \le \sup\limits_{\theta\in[0,2\pi]}\max\limits_{j,k\in[p]}\left[\sum_{|h|>m_0}\|\Sigma_{jk}^{(h)}\|_\cS + \sum_{|h|\le m_0}\left\{1-K\left(h/m_0\right)\right\}\|\Sigma_{jk}^{(h)}\|_\cS\right]= R(m_0).
            \end{aligned}
        \end{equation*}
        
        \subsection{Proof of \Cref{thm3}}

        For $\big\|\phi_{jml} - \widehat{\phi}_{jml}\big\|$, we have \begin{align*}
            \big\|\phi_{jml} - \widehat{\phi}_{jml}\big\| = &\nth{\int_{0}^{2\pi}\varphi_{jm}(\cdot ; \theta)\exp(-\mathbbm{i} l\theta)\mathrm{d}\theta - \int_{0}^{2\pi}\widehat{\varphi}_{jm}(\cdot ; \theta)\exp(-\mathbbm{i} l\theta)\mathrm{d}\theta}\\
            \le & \int_{0}^{2\pi}\nth{\varphi_{jm}(\cdot ; \theta)\exp(-\mathbbm{i} l\theta) - \widehat{\varphi}_{jm}(\cdot ; \theta)\exp(-\mathbbm{i} l\theta)}\mathrm{d}\theta\\
            = & \int_{0}^{2\pi}\nth{\varphi_{jm}(\cdot ; \theta)- \widehat{\varphi}_{jm}(\cdot ; \theta)}\mathrm{d}\theta \le \sup_{\theta\in[0,2\pi]} \nth{\varphi_{jm}(\cdot ; \theta)- \widehat{\varphi}_{jm}(\cdot ; \theta)}.
        \end{align*}

        From Lemma 3.1 in \cite{hormann2010}, we have for any \(\theta\in[0,2\pi]\), \(|\widehat{\lambda}_{jl}(\theta) - \lambda_{jl}(\theta)| \le \|\hat{f}_{\theta, jj} - f_{\theta, jj}\|_{\cS}\). Notice that we identify $\varphi_{jl}$ in \eqref{eigenfunctionversion}.  From Lemma 3.2 in \cite{hormann2010}, we have for any \(\theta\in[0,2\pi]\), \(\|\widehat{\varphi}_{jl}(\cdot ; \theta) - \varphi_{jl}(\cdot ; \theta)\|/\delta_l \le 2\sqrt{2}\|\hat{f}_{\theta, jj} - f_{\theta, jj}\|_{\cS}\). This gives $$\max\limits_{j\in[p]}\max\limits_{l\in[M]}\sup\limits_{\theta\in[0,2\pi]}\left\{|\widehat{\lambda}_{jl}(\theta) - \lambda_{jl}(\theta)| + \|\widehat{\varphi}_{jl}(\cdot;\theta) - \varphi_{jl}(\cdot ; \theta)\|\delta_l^{-1}\right\} \le C\sup\limits_{\theta\in[0,2\pi]}\max_{j\in[p]}\|\hat{f}_{\theta, jj} - f_{\theta, jj}\|_{\cS}.$$
        And the right hand side is smaller than \(\sup\limits_{\theta\in[0,2\pi]}\|\widehat{\bbf}_\theta - \bbf_\theta\|_{\cS,\max}\). Using the result in \Cref{thm2} we finish our proof.

        \subsection{Proof of \Cref{thm4}}

        We organize our proof in three steps. First, we decompose the estimation error into two parts: precision error and truncation error. Second, we use our results in Section \ref{sec.cov} to control the precision error. Finally, we analyze the truncation error.

        \subsubsection{Decomposition and definition}

        Recall that \begin{align*}
            &\widehat{\sigma}^{(h)}_{jkml} = \left\{\sum_{|r_1|,|r_2|\le L}\sum_{t = L}^{n-L-h}\langle X_{(t-r_1)j},\widehat{\phi}_{jmr_1} \rangle\langle X_{(t+h-r_2)k},\widehat{\phi}_{klr_2}\rangle\right\}/(n-2L-h),\\
            &\zeta_{tjm} = \sum_{|r_1|\le L}\langle X_{(t-r_1)k},\phi_{jmr_1}\rangle + \sum_{|r_1|>L}\langle X_{(t-r_1)j},\phi^*_{jmr_1}\rangle,~~\mbox{and}\\
            &\zeta_{(t+h)kl} = \sum_{|r_2|\le L}\langle X_{(t+h-r_2)k},\phi_{klr_2}\rangle + \sum_{|r_2|> L}\langle X_{(t+h-r_2)k},\phi^*_{klr_2}\rangle.
        \end{align*}
        From Proposition 3 of \cite{hormann2015} we have $E(\zeta_{tjm})=E(\zeta_{(t+h)kl})=0$. The definition of covariance of dynamic FPC score gives \begin{align*}
            \sigma_{jkml}^{(h)} =& E(\zeta_{tjm}\zeta_{(t+h)kl})\\
            =&\sum_{|r_1|\le L,|r_2|\le L} \langle\phi_{jmr_1},\Sigma_{jk}^{(h+r_1-r_2)}(\phi_{klr_2})\rangle +\sum_{|r_1|>L,|r_2|\le L} \langle\phi^*_{jmr_1},\Sigma_{jk}^{(h+r_1-r_2)}(\phi_{klr_2})\rangle \\
            & +\sum_{|r_1|\le L,|r_2|> L} \langle\phi_{jmr_1},\Sigma_{jk}^{(h+r_1-r_2)}(\phi^*_{klr_2})\rangle+\sum_{|r_1|>L,|r_2|> L} \langle\phi^*_{jmr_1},\Sigma_{jk}^{(h+r_1-r_2)}(\phi^*_{klr_2})\rangle.
        \end{align*}
        Define $P^{(h)}_{jkml} = \widehat{\sigma}_{jkml}^{(h)} - \sum_{|r_1|,|r_2|\le L} \langle\phi_{jmr_1},\Sigma_{jk}^{(h+r_1-r_2)}(\phi_{klr_2})\rangle$. Also define \begin{align*}
            P^{(h)'}_{jkml} =& \sigma_{jkml}^{(h)} - \sum_{|r_1|\le L,|r_2|\le L} \langle\phi_{jmr_1},\Sigma_{jk}^{(h+r_1-r_2)}(\phi_{klr_2})\rangle = \sum_{|r_1|\le L,|r_2| > L}\langle\phi_{jmr_1},\Sigma_{jk}^{(h+r_1-r_2)}(\phi^*_{klr_2})\rangle \\
            &+ \sum_{|r_1|>L,|r_2|\le L }\langle\phi^*_{jmr_1},\Sigma_{jk}^{(h+r_1-r_2)}(\phi_{klr_2})\rangle + \sum_{|r_1|,|r_2| > L}\langle\phi^*_{jmr_1},\Sigma_{jk}^{(h+r_1-r_2)}(\phi^*_{klr_2})\rangle 
        \end{align*}
        as the truncation error that arises from (\ref{FPCADef3}). We decompose the estimation error into two parts, namely $$\widehat{\sigma}^{(h)}_{jkml} - \sigma^{(h)}_{jkml} = P^{(h)}_{jkml} - P^{(h)'}_{jkml}.$$
        
        To simplify our notation, let \(\widehat{\Delta}_{jk}^{(h)} = \widehat{\Sigma}_{jk}^{(h)} - \Sigma^{(h)}_{jk},\widehat{\omega}_{jmr} = \widehat{\phi}_{jmr} - \phi_{jmr}\).
        The precision error can be further decomposed as \(P^{(h)}_{jkml} = \sum_{|r_1|,|r_2|\le L}P^{(h)}_{jkmlr_1r_2}\), where \begin{align*}
	P^{(h)}_{jkmlr_1r_2} = & \frac{1}{n-2L-h}\left\{\sum_{t = L}^{T-L-h}\langle X_{(t-r_1)j},\widehat{\phi}_{jmr_1} \rangle\langle X_{(t+h-r_2)k},\widehat{\phi}_{klr_2}\rangle\right\} \\
    & - \langle\Sigma^{(h+r_1-r_2)}_{jk}(\phi_{jmr_1}),\phi_{klr_2}\rangle\\
	= & \langle\widehat{\Sigma}_{jk}^{(h+r_1-r_2)}(\widehat{\phi}_{jmr_1}),\widehat{\phi}_{klr_2}\rangle - \langle\Sigma^{(h+r_1-r_2)}_{jk}(\phi_{jmr_1}),\phi_{klr_2}\rangle \\
	= & P_{jkmlr_1r_2}^{(h1)} + P_{jkmlr_1r_2}^{(h2)}+P_{jkmlr_1r_2}^{(h3)}+P_{jkmlr_1r_2}^{(h4)}.
        \end{align*}
        Here \begin{align*}
	P_{jkmlr_1r_2}^{(h1)} & = \langle\widehat{\omega}_{jmr_1},\widehat{\Sigma}_{jk}^{(h+r_1-r_2)}(\widehat{\omega}_{klr_2})\rangle,\\
	P_{jkmlr_1r_2}^{(h2)} & = \langle\phi_{jmr_1},\widehat{\Delta}_{jk}^{(h+r_1-r_2)}(\widehat{\omega}_{klr_2})\rangle + \langle\widehat{\omega}_{jmr_1},\widehat{\Delta}_{jk}^{(h+r_1-r_2)}(\phi_{klr_2})\rangle,\\
	P_{jkmlr_1r_2}^{(h3)} & = \langle\phi_{jmr_1},\Sigma_{jk}^{(h+r_1-r_2)}(\widehat{\omega}_{klr_2})\rangle + \langle\widehat{\omega}_{jmr_1},\Sigma_{jk}^{(h+r_1-r_2)}(\phi_{klr_2})\rangle,\\
	P_{jkmlr_1r_2}^{(h4)} & = \langle\phi_{jmr_1},\widehat{\Delta}_{jk}^{(h+r_1-r_2)}(\phi_{klr_2})\rangle.
        \end{align*}

        \subsubsection{Control of $P^{(h)}_{jkml}$}

        The analysis of $P^{(h)}_{jkml}$ relies on the control of $\widehat{\omega}_{jmr}$ and $\widehat{\Delta}_{jk}^{(h+r_1-r_2)}$. For $\widehat{\omega}_{jmr}$, by Lemma 3.2 in \cite{hormann2010}, we have \begin{align*}
	2\pi \|\widehat{\omega}_{jmr}\|_\cH &= 2\pi\|\widehat{\phi}_{jmr} - \phi_{jmr}\|_\cH = \int_{0}^{2\pi}\|\{\varphi_{jm}(\cdot ; \theta) - \widehat{\varphi}_{jm}(\cdot ; \theta)\}\exp(\mathbbm{i} r\theta)\|_\cH\mathrm{d}\theta \\
        &\le \int_{0}^{2\pi}\|\varphi_{jm}(\cdot ; \theta) - \widehat{\varphi}_{jm}(\cdot ; \theta)\|_\cH\mathrm{d}\theta \le 2\sqrt{2}\delta_m\int_{0}^{2\pi}\|f_{\theta, jj} - \hat{f}_{\theta, jj}\|_\mathcal{S}\mathrm{d}\theta.
        \end{align*}
        Also we have \begin{align*}
            &P\left\{\sum_{|r_1|,|r_2|\le L}\sup\limits_{j,k\in[p]}\|\widehat{\Delta}^{(h+r_1-r_2)}_{jk}\|_\cS>\cM_{q,\alpha}^\bX L^2x\right\} \le \sum_{|r_1|,|r_2|\le L}P\left\{\sup\limits_{j,k\in[p]}\|\widehat{\Delta}^{(h+r_1-r_2)}_{jk}\|_\cS>\cM_{q,\alpha}^\bX x\right\}\\
            &\le C_{q,\alpha}L^2x^{-q/2}D_{n,n/2}(\log p)^q + Cp^2L^2\exp(-C_\alpha x^2 n),\\
            &P\left\{\sum_{|r_1|,|r_2|\le L}\sup\limits_{j,k\in[p]}\|\widehat{\Delta}^{(h+r_1-r_2)}_{jk}\|_\cS>\bPhi_{q,\alpha}^\bX L^2x\right\} \le \sum_{|r_1|,|r_2|\le L}P\left\{\sup\limits_{j,k\in[p]}\|\widehat{\Delta}^{(h+r_1-r_2)}_{jk}\|_\cS>\bPhi_{q,\alpha}^\bX x\right\}\\
            &\le C_{q,\alpha}L^2x^{-q/2}p^2D_{n,n/2} + Cp^2L^2\exp(-C_\alpha x^2 n).
        \end{align*}
        These two inequalities use results in \Cref{thm1}. Notice that we need to substitute $D_{n,h}$ in \Cref{thm1} with $D_{n,n/2} \le 2n^{-q/4}$ since we are considering many lag orders at the same time and $|h+r_1-r_2| < n/2$. Combining these two concentration inequalities, we have \begin{align*}
            \sum_{|r_1|,|r_2|\le L}\|\widehat{\Delta}^{(h+r_1-r_2)}_{jk}\|_{\cS,\max}= O_P\pth{\bPhi_{q,\alpha}^\bX C_\bX p^{4/q} L^{2+4/q}n^{-1/2}}=O_P(\cH_2),
        \end{align*}
        where $C_\bX$ is defined in \eqref{eq:constCX}.

        Now we analyze  \(P_{jkmlr_1r_2}^{(ha)},a = 1,2,3,4\). For \(P_{jkmlr_1r_2}^{(h1)}\), we have \begin{align*}
        \sup\limits_{j,k\in[p]}|P_{jkmlr_1r_2}^{(h1)}| &\le \sup\limits_{j\in[p]}\|\widehat{\omega}_{jmr_1}\|_\cH\sup\limits_{k\in[p]}\|\widehat{\omega}_{klr_2}\|_\cH\sup\limits_{j,k\in[p]}\|\widehat{\Sigma}_{jk}^{(h+r_1 -r_2)}\|_\mathcal{S} \\
        &\le C\delta_m\delta_l\left(\sup_{\theta\in[0,2\pi]}\|\widehat{\bbf}_\theta-\bbf_\theta\|_{\cS,\max}\right)^2\|\widehat{\bSigma}^{(h+r_1-r_2)}\|_{\cS,\max}.
        \end{align*}
        For \(P_{jkmlr_1r_2}^{(h2)}\), we have \begin{align*}
        \sup\limits_{j,k\in[p]}|P_{jkmlr_1r_2}^{(h2)}| \le &\sup\limits_{j\in[p]}\|\phi_{jmr_1}\|_\cH\sup\limits_{k\in[p]}\|\widehat{\omega}_{klr_2}\|_\cH\sup\limits_{j,k\in[p]}\|\widehat{\Delta}_{jk}^{(h+r_1 -r_2)}\|_\mathcal{S}\\ &+ \sup\limits_{j\in[p]}\|\widehat{\omega}_{jmr_1}\|_\cH\sup\limits_{k\in[p]}\|\phi_{klr_2}\|_\cH\sup\limits_{j,k\in[p]}\|\widehat{\Delta}_{jk}^{(h+r_1 -r_2)}\|_\mathcal{S}\\
        \le & C(\delta_m\vee \delta_l)\left(\sup_{\theta\in[0,2\pi]}\|\widehat{\bbf}_\theta-\bbf_\theta\|_{\cS,\max}\right)\left(\sup\limits_{j,k\in[p]}\|\widehat{\Delta}^{(h+r_1-r_2)}_{jk}\|_{\cS}\right).
        \end{align*}
        For \(P_{jkmlr_1r_2}^{(h3)}\), recalling that $\lambda_0 = \max_{j\in[p]}\int_0^1\Sigma^{(0)}_{jj}(u,u)\du $, we have \begin{align*}
        \sup\limits_{j,k\in[p]}|P_{jkmlr_1r_2}^{(h3)}| \le & \sup\limits_{j\in[p]}\|\phi_{jmr_1}\|_\cH\sup\limits_{k\in[p]}\|\widehat{\omega}_{klr_2}\|_\cH\sup\limits_{j,k\in[p]}\|\Sigma_{jk}^{(h+r_1 -r_2)}\|_\mathcal{S} \\
        &+ \sup\limits_{j\in[p]}\|\widehat{\omega}_{jmr_1}\|_\cH\sup\limits_{k\in[p]}\|\phi_{klr_2}\|_\cH\sup\limits_{j,k\in[p]}\|\Sigma_{jk}^{(h+r_1 -r_2)}\|_\mathcal{S}\\
        =& C\lambda_0(\delta_m\vee \delta_l)\sup_{\theta\in[0,2\pi]}\|\widehat{\bbf}_\theta-\bbf_\theta\|_{\cS,\max}.
        \end{align*}
        In the last line we use the fact that for any $h$, $\sup_{j,k\in[p]} \|\Sigma_{jk}^{(h)}\|_\cS \le \lambda_0$, which follows from Cauchy--Schwarz inequality. For \(P_{jkmlr_1r_2}^{(h4)}\), we have\begin{align*}
        &\sup\limits_{j,k\in[p]}|P_{jkmlr_1r_2}^{(h4)}| \le \sup\limits_{j\in[p]}\|\phi_{jmr_1}\|_\cH\sup\limits_{k\in[p]}\|\phi_{klr_2}\|_\cH\sup\limits_{j,k\in[p]}\|\widehat{\Delta}_{jk}^{(h+r_1 -r_2)}\|_\mathcal{S} \le  C\sup\limits_{j,k\in[p]}\|\widehat{\Delta}_{jk}^{(h+r_1 -r_2)}\|_\mathcal{S}.
        \end{align*}
        In the above argument we use the fact that $\varphi(\theta,\cdot)$ is the eigenfunction with $\|\varphi(\theta,\cdot)\|_\cH=1$. Thus $\|\phi_{jmr_1}\|_\cH = \left[\int_0^1\left|\int_0^{2\pi}\varphi(\theta,u)\exp(-\mathbbm{i} r_1\theta)\mathrm{d}\theta\right|^2\mathrm{d}u\right]^{1/2}\le \left\{\int_0^{2\pi}\int_0^{1}|\varphi(\theta,u)|^2\mathrm{d}u\mathrm{d}\theta\right\}^{1/2} = (2\pi)^{1/2}$. Taking summation, we have \begin{align*}
            &\sup\limits_{j,k\in[p]}\frac{1}{\delta_m\vee \delta_l}\left|P^{(h)}_{jkml}\right| \le  \sum\limits_{|r_1|,|r_2|\le L}\sum_{a = 1}^{4} \sup\limits_{j,k\in[p]}\left|P_{jkmlr_1r_2}^{(ha)}\right|\\
            \le & C(\delta_m\wedge \delta_l)\left(\sup_{\theta\in[0,2\pi]}\|\widehat{\bbf}_\theta-\bbf_\theta\|_{\cS,\max}\right)^2\sum_{|r_1|,|r_2|<L}\|\widehat{\bSigma}^{(h+r_1-r_2)}\|_{\cS,\max}\\ 
            & + C\left(\sup_{\theta\in[0,2\pi]}\|\widehat{\bbf}_\theta-\bbf_\theta\|_{\cS,\max}\right)\sum_{|r_1|,|r_2|<L}\|\widehat{\bSigma}^{(h+r_1-r_2)} - \bSigma^{(h+r_1-r_2)}\|_{\cS,\max}\\
            & + C\lambda_0 L^2\left(\sup_{\theta\in[0,2\pi]}\|\widehat{\bbf}_\theta-\bbf_\theta\|_{\cS,\max}\right) + C\sum_{|r_1|,|r_2|<L}\|\widehat{\bSigma}^{(h+r_1-r_2)} - \bSigma^{(h+r_1-r_2)}\|_{\cS,\max}.
        \end{align*}
        Recall that $\sup_{\theta\in[0,2\pi]}\|\widehat{\bbf}_\theta-\bbf_\theta\|_{\cS,\max} = O_P(\cH_1), \sum_{|r_1|,|r_2|<L}\|\widehat{\bSigma}^{(h+r_1-r_2)} - \bSigma^{(h+r_1-r_2)}\|_{\cS,\max} = O_P(\cH_2), \sum_{|r_1|,|r_2|<L}\|\bSigma^{(h+r_1-r_2)}\|_{\cS,\max} \le \lambda_0 L^2$, $m,l \in [M]$ and $\delta_m\wedge\delta_l \le \delta_M$.
        This gives \begin{align*}
            \sup\limits_{m,l\in[M]}\sup\limits_{j,k\in[p]}\frac{1}{\delta_m\vee \delta_l}\left|P^{(h)}_{jkml}\right| = O_P\{\delta_M\cH_1^2(L^2\lambda_0+\cH_2)+\cH_1\cH_2+\lambda_0 L^2\cH_1+\cH_2\}.
        \end{align*}
        Since we assume that $\lambda_0 = \max_{j\in[p]}\int_0^1\Sigma^{(0)}_{jj}(u,u)\du = O(1),$ $\cH_1 = o(1),$ and $ \cH_1\delta_M = O(1),$ we have \begin{align*}
            \sup\limits_{m,l\in[M]}\sup\limits_{j,k\in[p]}\frac{1}{\delta_m\vee \delta_l}\left|P^{(h)}_{jkml}\right| = O_P(L^2\cH_1 + \cH_2).
        \end{align*}

        \subsubsection{Control of $P^{(h)'}_{jkml}$}
        
        For $P^{(h)'}_{jkml}$ we have \begin{align*}
            \left|P^{(h)'}_{jkml}\right| \le &  \Big|\sum\limits_{|r_1|,|r_2| > L}\langle\phi^*_{jmr_1},\Sigma_{jk}^{(h+r_1-r_2)}(\phi^*_{klr_2})\rangle\Big| + \Big|\sum\limits_{|r_1|\le L,|r_2| > L}\langle\phi_{jmr_1},\Sigma_{jk}^{(h+r_1-r_2)}(\phi^*_{klr_2})\rangle\Big| \\
            & + \Big|\sum\limits_{|r_1|>L,|r_2|\le L}\langle\phi^*_{jmr_1},\Sigma_{jk}^{(h+r_1-r_2)}(\phi_{klr_2})\rangle\Big|\\
            \le & \lambda_0\pth{\sum_{|r_1|\le L,|r_2| > L}\|\phi_{jmr_1}\|_\cH\|\phi^*_{klr_2}\|_\cH + \sum_{|r_1|>L,|r_2|\le L}\|\phi^*_{jmr_1}\|_\cH\|\phi_{klr_2}\|_\cH} \\
            & + \lambda_0\sum_{|r_1|,|r_2| > L}\|\phi^*_{jmr_1}\|_\cH\|\phi^*_{klr_2}\|_\cH.
        \end{align*}
        Using integral by parts and noticing that $\varphi^*(u;\theta)$ satisfies $\partial_\theta^i\varphi^*(u;0) = \partial_\theta^i\varphi^*(u;2\pi)$ for $i=0,\dots,\kappa-1$, we have $|\phi^*_{jmr}(u)| = \left|\int_{0}^{2\pi}\varphi^*_{jm}(u ; \theta)\exp(-\mathbbm{i} r\theta)\mathrm{d}\theta\right|\le |r|^{-\kappa}\int_{0}^{2\pi}|\frac{\partial^\kappa}{\partial \theta^\kappa}\varphi^*_{jm}(u ; \theta)|\mathrm{d}\theta$.  By Condition \ref{cond5}, \(\|\phi^*_{jmr}\|_\cH = \left\{\int_{0}^{1} |\phi^*_{jmr}(u)|_2^2\mathrm{d}u\right\}^{1/2} \le |r|^{-\kappa}\left\{\int_{0}^{1}\int_{0}^{2\pi}\left|\frac{\partial^\kappa}{\partial \theta^\kappa}\varphi^*_m(u;\theta)\right|_2^2\mathrm{d}\theta\mathrm{d}u\right\}^{1/2}\le C|r|^{-\kappa}\). Recall $m,l\in[M]$. Under all conditions in Theorem \ref{thm4}, we have \begin{equation*}
		\sup\limits_{j,k\in[p],m,l\in[M]} P^{(h)'}_{jkml} = O(\lambda_0 L^{2-\kappa}).
	\end{equation*}
        Since \(\delta_m\rightarrow\infty\) and $\lambda_0 = O(1)$, we have \(\sup_{j,k\in[p],m,l\in[M]} P^{(h)'}_{jkml}/(\delta_m\vee \delta_l) = O(L^{2-\kappa})\).

    \subsection{Proof of Theorem \ref{thm5}}\label{Sec.A.5}

        First of all, we define the following uniform thresholding operators (across all frequencies).

        \begin{definition}[Uniform thresholding operators across all frequencies]\label{defS.1}
        We define $s_\lambda:[0,2\pi]\times\eS \rightarrow [0,2\pi]\times \eS$ as a uniform thresholding operators across all frequencies if it satisfies the following three conditions:

        (\romannumeral 1) For some $c > 0$, $\sup_{\theta\in[0,2\pi]}\|s_\lambda(Z_\theta)\|_{\cS}\le c\sup_{\theta\in[0,2\pi]}\|Y_\theta\|_{\cS}$ for all $Z_\theta, Y_\theta\in[0,2\pi]\times\eS$ that satisfy $\sup_{\theta\in[0,2\pi]}\|Z_\theta-Y_\theta\|_{\cS}\le \lambda$.

        (\romannumeral 2) $\sup_{\theta\in[0,2\pi]}\|s_\lambda(Z_\theta)\|_{\cS} = 0$ for all $\sup_{\theta\in[0,2\pi]}\|Z_\theta\|_{\cS} \le \lambda$.

        (\romannumeral 3) $\sup_{\theta\in[0,2\pi]}\|s_\lambda(Z_\theta)-Z_\theta\|_{\cS}\le \lambda$ for all $Z_\theta\in[0,2\pi]\times\eS$.
            
        \end{definition}

        It is straightforward to verify that $s_\lambda(Z_\theta) = Z_\theta(1-\lambda/\sup_{\theta\in[0,2\pi]}\|Z_\theta\|_\cS)_+$ satisfies \Cref{defS.1}. By \eqref{det12}, we have $\hat{f}_{\theta, jk}^{\cT} = s_\lambda\big(\hat{f}_{\theta, jk}\big)$. So it suffices to prove that the statement in Theorem \ref{thm5} holds for all uniform thresholding operators (across all frequencies).

        Assume that $\hat{f}_{\theta, jk}^{\cT} = s_\lambda\big(\hat{f}_{\theta, jk}\big)$ and $s_\lambda$ is a uniform thresholding operators (across all frequencies) as defined in \Cref{defS.1}. Let $\Omega_{n1} = \left\{\sup_{\theta\in[0,2\pi]}\max_{j,k\in[p]}\|\hat{f}_{\theta, jk}-f_{\theta, jk}\|_\cS\le\lambda\right\}$. Under $\Omega_{n1}$, $\sup_{\theta\in[0,2\pi]}\|s_\lambda(\hat{f}_{\theta, jk})\|_\cS\le c\sup_{\theta\in[0,2\pi]}\|f_{\theta, jk}\|$ and $\sup_{\theta\in[0,2\pi]}\|s_\lambda(\hat{f}_{\theta, jk})-\hat{f}_{\theta, jk}\|_\cS\le \lambda$. Then on the event $\Omega_{n1}$, we have \begin{align*}
            &\sup_{\theta\in[0,2\pi]}\sum_{k = 1}^p \|\hat{f}_{\theta, jk}^{\cT}-f_{\theta, jk}\|_\cS\\
            \le & \sum_{k = 1}^p \sup_{\theta\in[0,2\pi]} \|\hat{f}_{\theta, jk}^{\cT}-f_{\theta, jk}\|_\cS I\pth{\sup_{\theta\in[0,2\pi]}\|\hat{f}_{\theta, jk}\|_\cS \ge \lambda} \\
            &+ \sum_{k = 1}^p \sup_{\theta\in[0,2\pi]}\|f_{\theta, jk}\|_\cS I\pth{\sup_{\theta\in[0,2\pi]}\|\hat{f}_{\theta, jk}\|_\cS < \lambda}\\
            \le & \sum_{k = 1}^p\sup_{\theta\in[0,2\pi]}\|s_\lambda(\hat{f}_{\theta, jk}) - \hat{f}_{\theta, jk}\|_\cS I\pth{\sup_{\theta\in[0,2\pi]}\|\hat{f}_{\theta, jk}\|_\cS \ge \lambda} I\pth{\sup_{\theta\in[0,2\pi]}\|f_{\theta, jk}\|_\cS \ge \lambda} \\
            & + \sum_{k = 1}^p\sup_{\theta\in[0,2\pi]}\|\hat{f}_{\theta, jk} - f_{\theta, jk}\|_\cS I\pth{\sup_{\theta\in[0,2\pi]}\|\hat{f}_{\theta, jk}\|_\cS \ge \lambda} I\pth{\sup_{\theta\in[0,2\pi]}\|f_{\theta, jk}\|_\cS \ge \lambda}\\
            & + \sum_{k = 1}^p\sup_{\theta\in[0,2\pi]}\|s_\lambda(\hat{f}_{\theta, jk}) - f_{\theta, jk}\|_\cS I\pth{\sup_{\theta\in[0,2\pi]}\|\hat{f}_{\theta, jk}\|_\cS \ge \lambda} I\pth{\sup_{\theta\in[0,2\pi]}\|f_{\theta, jk}\|_\cS < \lambda}\\
            & + \sum_{k = 1}^p \sup_{\theta\in[0,2\pi]}\|f_{\theta, jk}\| I\pth{\sup_{\theta\in[0,2\pi]}\|f_{\theta, jk}\|_\cS < 2\lambda}\\
            \le & \sum_{k = 1}^p 2\lambda  I\pth{\sup_{\theta\in[0,2\pi]}\|f_{\theta, jk}\|_\cS \ge \lambda} + (1+c)\sum_{k = 1}^p\sup_{\theta\in[0,2\pi]}\|f_{\theta, jk}\|_\cS I\pth{\sup_{\theta\in[0,2\pi]}\|f_{\theta, jk}\|_\cS < \lambda}\\
            & + \sum_{k = 1}^p\sup_{\theta\in[0,2\pi]}\|f_{\theta, jk}\|_\cS I\pth{\sup_{\theta\in[0,2\pi]}\|f_{\theta, jk}\|_\cS < 2\lambda}\\
            \le & C\lambda^{1-q^*}\sum_{k = 1}^{p}\sup_{\theta\in[0,2\pi]}\|f_{\theta, jk}\|_{\cS}^{q^*} \le Cs_0(p)\lambda^{1-q^*}.
        \end{align*}
        Here $C$ is an absolute constant. 
        Then using the convergence rate result in \eqref{thm2eq1} and $\lambda^{-1}\cH_1\rightarrow 0$, where $\cH_1$ is exactly the convergence \eqref{thm2eq1}, we have  $1-P\{\Omega_{n1}\} = o(1)$, and this finishes the proof.

        \subsection{Proof of Theorem \ref{thm6}}

        For any $\bbf_\theta\in\cC_0(s_0)$ and its estimation $\widehat{\bbf}_\theta$, we define the following two sets: \begin{align*}
            S_{n1} & = \left\{(j,k):\sup_{\theta\in[0,2\pi]}\|\hat{f}_{\theta, jk}\|_\cS> \lambda,\sup_{\theta\in[0,2\pi]}\|f_{\theta, jk}\|_\cS = 0\right\},\quad \\
            S_{n2} & = \left\{(j,k):\sup_{\theta\in[0,2\pi]}\|\hat{f}_{\theta, jk}\|_\cS = 0,\sup_{\theta\in[0,2\pi]}\|f_{\theta, jk}\|_\cS > \lambda\right\}.
        \end{align*}
        We have $\{|S_{n1}|>0\}\subset \{\sup_{\theta\in[0,2\pi]}\|\widehat{\bbf}_\theta-\bbf_\theta\|_{\cS,\max}>\lambda\},\{|S_{n2}|>0\}\subset \{\sup_{\theta\in[0,2\pi]}\|\widehat{\bbf}_\theta-\bbf_\theta\|_{\cS,\max}>\lambda\}$. Recall that in the proof of Theorem \ref{thm5}, we prove that under our choice of $\lambda$, $P(\Omega_{n1})=1+o(1)$, where $\Omega_{n1} = \sth{\max_{j,k\in[p]}\|\hat{f}_{\theta, jk}-f_{\theta, jk}\|_\cS\le \lambda}$. This gives $P\{|S_{n1}|>0\} + P\{|S_{n2}|>0\} = o(1)$ and it is uniform over $\bbf_\theta\in\cC_0$. Since for any $\bbf_\theta\in\cC_0$, $\left\{\mathrm{supp}(\widehat{\bbf}_{\theta}^{\cT})\ne \mathrm{supp}(\bbf_\theta)\right\} \subset S_{n1}\cup S_{n2}$, we finish the proof.

    \subsection{Proof of Theorem \ref{thm7}}

        Without loss of generality, we deal with the case $h \ge 0$. We first have the following decomposition $$\|\widetilde{\bSigma}^{(h)}-\bSigma^{(h)}\|_{\cS,\max} \le \|\widetilde{\bSigma}^{(h)}-\widehat{\bSigma}^{(h)}\|_{\cS,\max} + \|\widehat{\bSigma}^{(h)}-\bSigma^{(h)}\|_{\cS,\max}.$$
        The second term is \(O_P(\cH_3)\) using results in Theorem \ref{thm1}. To proceed the analysis for the first term, we need to define a set of new notations. Denote \(\be_0 = (1,0)^\T,\widetilde{\bU}_{tji} = \{1,(U_{tji}-u)/b_{j}\}^\T\). Define $$\widehat{\bS}_{tj}(u) = \frac{1}{T_{tj}}\sum_{i = 1}^{T_{tj}}\widetilde{\bU}_{tji}\widetilde{\bU}_{tji}^\T K_{b_{j}}(U_{tji}-u),\quad \widehat{\bR}_{tj}(u) = \frac{1}{T_{tj}}\sum_{i = 1}^{T_{tj}}\widetilde{\bU}_{tji}Y_{tji} K_{b_{j}}(U_{tji}-u).$$
        Let \(\widetilde{X}_{tj}(u) = \be_0^\T\left[E\{\widehat{\bS}_{tj}(u)\}\right]^{-1}E\{\widehat{\bR}_{tj}(u)\mid \bX_t\}\). For any square matrix $\bB$, write $\|\bB\|_{\min} = \{\lambda_{\min}(\bB^\T\bB)\}^{1/2}$. Define the event $$\Omega_{tj1}(\delta')= \left\{\sup_{u\in[0,1]}\|\widehat{\bS}_{tj}(u)-E\{\widehat{\bS}_{tj}(u)\}\|_F\le C_S\delta'/2\right\},\quad \delta'\in(0,1]$$
        and $\Omega_{t1}(\delta') = \bigcap_{j\in[p]} \Omega_{tj1}(\delta')$. Here $C_S$ is a constant that can be decomposed as $C_S=m_fC_K$, where $m_f$ is the lower bound of density $f_U$, and $C_K$ is a constant that only depends on kernel function $K$. It also satisfies for any bandwidth $b_{j}$, $C_S\le \inf_{u\in[0,1]}\|E\{\widehat{\bS}_{tj}(u)\}\|_{\min}$. For the existence of such constants, see Lemma \ref{LemmaPartial3} for details.

        For the first term $\|\widetilde{\bSigma}^{(h)}-\widehat{\bSigma}^{(h)}\|_{\cS,\max}$, using triangle inequality, we have \begin{equation*}
            \begin{aligned}
                &\|\widetilde{\bSigma}^{(h)}-\widehat{\bSigma}^{(h)}\|_{\cS,\max}\\
                = &\Big\| \frac{1}{n-h}\sum_{t = 1}^{n-h}\pth{\widehat{\bX}_t\otimes\widehat{\bX}_{t+h}^\T - \bX_t\otimes\widehat{\bX}_{t+h}^\T +\bX_t\otimes\widehat{\bX}_{t+h}^\T- \bX_t\otimes\bX_{t+h}^\T}\Big\|_{\cS,\max}\\
                \le & \Big\|\frac{1}{n-h}\sum_{t = 1}^{n-h}\pth{\bX_t\otimes\widehat{\bX}_{t+h}^\T- \bX_t\otimes\bX_{t+h}^\T}\Big\|_{\cS,\max} \\
                & + \Big\|\frac{1}{n-h}\sum_{t = 1}^{n-h}\pth{\widehat{\bX}_t\otimes\widehat{\bX}_{t+h}^\T - \bX_t\otimes\widehat{\bX}_{t+h}^\T}\Big\|_{\cS,\max}\\
                := & I_1 + I_2.
            \end{aligned}
        \end{equation*}

        \subsubsection{Evaluation of $I_1$}
        Using triangle inequality, we have the following decomposition \begin{equation*}
            \begin{aligned}
                I_1 = &\Big\|\frac{1}{n-h}\sum_{t=1}^{n-h}\bX_{t}\otimes\widehat{\bX}_{t+h}^\T - \bX_{t}\otimes\bX_{t+h}^\T\Big\|_{\cS,\max}\\
                \le & \frac{1}{n-h}\sum_{t=1}^{n-h}\left\|\bX_{t}\otimes\widehat{\bX}_{t+h}^\T - \bX_{t}\otimes\widetilde{\bX}_{t+h}^\T\right\|_{\cS,\max} + \frac{1}{n-h}\sum_{t=1}^{n-h}\left\| \bX_{t}\otimes\widetilde{\bX}_{t+h}^\T-\bX_{t}\otimes \bX_{t+h}^\T\right\|_{\cS,\max} \\
                \le & \frac{1}{n-h}\sum_{t=1}^{n-h}\|\bX_t\|_{\cH,\infty}\|\widehat{\bX}_{t+h}-\widetilde{\bX}_{t+h}\|_{\cH,\infty} + \frac{1}{n-h}\sum_{t=1}^{n-h}\|\bX_t\|_{\cH,\infty}\|\widetilde{\bX}_t-\bX_t\|_{\cH,\infty}\\
                \le & \frac{1}{n-h}\sum_{t=1}^{n-h}\|\bX_t\|_{\cH,\infty}\|\widehat{\bX}_{t+h}-\widetilde{\bX}_{t+h}\|_{\cH,\infty}I\{\Omega_{(t+h)1}(1)\} + \frac{1}{n-h}\sum_{t=1}^{n-h}\|\bX_t\|_{\cH,\infty}\|\widetilde{\bX}_t-\bX_t\|_{\cH,\infty}\\
                & + \frac{1}{n-h}\sum_{t=1}^{n-h}\|\bX_t\|_{\cH,\infty}\|\widehat{\bX}_{t+h}-\widetilde{\bX}_{t+h}\|_{\cH,\infty}\qth{1-I\{\Omega_{(t+h)1}(1)\}}\\
                := & I_{11}+I_{12}+I_{13},
            \end{aligned}
        \end{equation*}   
        where $I_{11},I_{12},I_{13}$ are defined in an obvious way. We first deal with $I_{12}$. For any $q \ge 2$, using (a) of Lemma \ref{lemmaTechComposite1}, we have \begin{equation*}
            \begin{aligned}
                \left\|\|\bX_t\|_{\cH,\infty}\right\|_q & = \left\|\Big\|\sum_{h=0}^\infty \cP_{t-h}\bX_t\Big\|_{\cH,\infty}\right\|_q \le \sum_{h = 0}^\infty \left\|\Big\|\cP_{t-h}\bX_t\Big\|_{\cH,\infty}\right\|_q\le \sum_{h = 0}^\infty \omega_{h,q} = \Omega_{0,q}.
            \end{aligned}
        \end{equation*}
        Using Cauchy--Schwarz inequality with (\ref{lemmaB8eq2}) and Condition \ref{condpartial1}(\romannumeral 6), we have $E(\|\widetilde{\bX}_t-\bX_t\|_{\cH,\infty}\|\bX_t\|_{\cH,\infty}) \le C\Omega_{0,2}^2(\max_j b_j)^2$, which implies $E(I_{12}) \le C\Omega_{0,2}^2\max_jb_j^2$. This gives \begin{equation}\label{pfthm7-1}
            \begin{aligned}
                I_{12} = O_P\pth{\Omega_{0,2}^2 \max_{j\in[p]}b_{j}^2}.
            \end{aligned}
        \end{equation}
        For $I_{11}$, we have \begin{equation*}
            \begin{aligned}
                E(I_{11}) & = \frac{1}{n-h}\sum_{t=1}^{n-h}E\sth{E\pth{\|\bX_t\|_{\cH,\infty}\|\widehat{\bX}_{t+h}-\widetilde{\bX}_{t+h}\|_{\cH,\infty}I\{\Omega_{(t+h)1}(1)\}\mid \bX_t}}\\
                & =  \frac{1}{n-h}\sum_{t=1}^{n-h}E\sth{\|\bX_t\|_{\cH,\infty}(1+|\bX_{t+h}^*|_\infty)E\pth{\frac{\|\widehat{\bX}_{t+h}-\widetilde{\bX}_{t+h}\|_{\cH,\infty}}{1+|\bX_{t+h}^*|_\infty} I\{\Omega_{(t+h)1}(1)\}\mid \bX_t}}.
            \end{aligned}
        \end{equation*}
        Combining this with \eqref{lemmaB8eq1} of Lemma \ref{LemmaPartial1}, we have \begin{equation*}
            \begin{aligned}
                &E\pth{\frac{\|\widehat{\bX}_{t+h}-\widetilde{\bX}_{t+h}\|_{\cH,\infty}}{1+|\bX_{t+h}^*|_\infty}I\{\Omega_{(t+h)1}(1)\mid \bX_t} \\
                & = \int_0^\infty P\pth{\frac{\|\widehat{\bX}_{t+h}-\widetilde{\bX}_{t+h}\|_{\cH,\infty}}{1+|\bX_{t+h}^*|_\infty} > \delta, \Omega_{(t+h)1}(1) \mid \bX_t}\mathrm{d}\delta\\
                & \le \int_0^\infty 1\wedge C_1p\exp\{-C_2\cT_t\min(\delta,\delta^2)\}\mathrm{d}\delta,
            \end{aligned}
        \end{equation*}
        where $\cT_t = \min_j T_{tj}b_{j}$. By Lemma \ref{LemmaPartial2}, this integral is smaller than $C(\log p/\cT_t)^{1/2}$. Since in Condition \ref{condpartial1}(\romannumeral 4) we assume that $T_{tj}$ are of the same order, we have \begin{equation}\label{pfthm7-2}
            E(I_{11}) \le \frac{C}{n-h}\sum_{t=1}^{n-h}\frac{(\log p)^{1/2}}{\cT_t^{1/2}}E\sth{\|\bX_t\|_{\cH,\infty}(1+|\bX_{t+h}^*|_\infty)} \le C\pth{\min_j\widebar{T}_j b_{j}}^{-1/2}(\log p)^{1/2}\Omega_{0,2}^2.
        \end{equation}
        Finally we analyze $I_{13}$. By \eqref{lemmaB8eq3}, we have
        \begin{equation}\label{pfthm7-10}
            \begin{aligned}
                P(I_{13}\ne 0) &\le \sum_{t=1}^n P\qth{1-I\{\Omega_{t1}(1)\}\ne 0} = \sum_{t=1}^n 1-P\sth{\Omega_{tj}(1)}\\
            &\le C_1\sum_{t=1}^n p\exp(-C_2\cT_t).
            \end{aligned}
        \end{equation}
        Since we have assumed $\min_j\widebar{T}_j b_{j}/\log (p\vee n)\rightarrow \infty$, and also $T_{tj}$'s are of the same order, we have $P(I_{13}\ne 0)\rightarrow 0$. Combining this with (\ref{pfthm7-1}) and (\ref{pfthm7-2}), we have \begin{equation}\label{pfthm7-8}
            \begin{aligned}
                I_1 = O_P\big\{\Omega_{0,2}^2\max_{j}b_{j}^2 + \big(\min_j \widebar{T}_j b_{j}\big)^{-1/2}(\log p)^{1/2}\Omega_{0,2}^2\big\}.
            \end{aligned}
        \end{equation}

        \subsubsection{Evaluation of $I_2$}
        Using triangle inequality, we have the following decomposition 
            \begin{align*}
                I_2&=\Big\|\frac{1}{n-h}\sum_{t = 1}^{n-h}\pth{\widehat{\bX}_t\otimes\widehat{\bX}_{t+h}^\T - \bX_t\otimes\widehat{\bX}_{t+h}^\T}\Big\|_{\cS,\max}\le \frac{1}{n-h}\sum_{t = 1}^{n-h}\nth{\widehat{\bX}_t\otimes\widehat{\bX}_{t+h}^\T - \bX_t\otimes\widehat{\bX}_{t+h}^\T}_{\cS,\max}\\
                &\le \frac{1}{n-h}\sum_{t = 1}^{n-h}\nth{\widehat{\bX}_t - \bX_t}_{\cH,\infty}\nth{\widehat{\bX}_{t+h}^\T}_{\cH,\infty}\\
                &\le \frac{1}{n-h}\sum_{t = 1}^{n-h}\nth{\widehat{\bX}_t - \bX_t}_{\cH,\infty}\pth{\nth{\bX_{t+h}-\widehat{\bX}_{t+h}^\T}_{\cH,\infty} + \nth{\bX_{t+h}}_{\cH,\infty}}\\
                &= \frac{1}{n-h}\sum_{t = 1}^{n-h}\|\bX_{t+h}\|_{\cH,\infty}\|\widehat{\bX}_t - \bX_t\|_{\cH,\infty} + \frac{1}{n-h}\sum_{t = 1}^{n-h}\|\widehat{\bX}_{t+h}-\bX_{t+h}\|_{\cH,\infty}\|\widehat{\bX}_{t}-\bX_{t}\|_{\cH,\infty}\\
                &\le \frac{1}{n-h}\sum_{t = 1}^{n-h}\|\bX_{t+h}\|_{\cH,\infty}\|\widehat{\bX}_t - \bX_t\|_{\cH,\infty} + \frac{1}{2n-2h} \qth{\sum_{t = 1}^{n-h}\|\widehat{\bX}_{t+h}-\bX_{t+h}\|_{\cH,\infty}^2 + \sum_{t = 1}^{n-h}\|\widehat{\bX}_{t}-\bX_{t}\|_{\cH,\infty}^2}\\
                &\le \frac{1}{n-h}\sum_{t = 1}^{n-h}\|\bX_{t+h}\|_{\cH,\infty}\|\widehat{\bX}_t - \bX_t\|_{\cH,\infty}+ \frac{1}{n-h}\sum_{t = 1}^{n-h}\|\widehat{\bX}_{t}-\bX_{t}\|_{\cH,\infty}^2\\
                & \le \frac{1}{n-h}\sum_{t = 1}^{n-h}\|\bX_{t+h}\|_{\cH,\infty}\|\widehat{\bX}_t - \bX_t\|_{\cH,\infty} +  \frac{1}{n-h}\sum_{t = 1}^{n-h}(\|\widehat{\bX}_{t}-\widetilde{\bX}_t\|_{\cH,\infty} + \|\widetilde{\bX}_t - \bX_{t}\|_{\cH,\infty})^2\\
                &\le \frac{1}{n-h}\sum_{t = 1}^{n-h}\|\bX_{t+h}\|_{\cH,\infty}\|\widehat{\bX}_t - \bX_t\|_{\cH,\infty} + \frac{2}{n-h}\sum_{t = 1}^{n-h}\|\widehat{\bX}_{t}-\widetilde{\bX}_t\|_{\cH,\infty}^2 + \frac{2}{n-h}\sum_{t = 1}^{n-h}\|\widetilde{\bX}_t - \bX_{t}\|_{\cH,\infty}^2\\
                & \le \frac{1}{n-h}\sum_{t = 1}^{n-h}\|\bX_{t+h}\|_{\cH,\infty}\|\widehat{\bX}_t - \bX_t\|_{\cH,\infty} + \frac{2}{n-h}\sum_{t = 1}^{n-h}\|\widehat{\bX}_{t}-\widetilde{\bX}_t\|_{\cH,\infty}^2I\{\Omega_{t1}(1)\} \\
                &+ \frac{2}{n-h}\sum_{t = 1}^{n-h}\|\widetilde{\bX}_t - \bX_{t}\|_{\cH,\infty}^2 + \frac{2}{n-h}\sum_{t = 1}^{n-h}\|\widehat{\bX}_{t}-\widetilde{\bX}_t\|_{\cH,\infty}^2\qth{1-I\{\Omega_{t1}(1)\}}\\
                &:= I_{21}+I_{22}+I_{23}+I_{24},
            \end{align*}
        where $I_{21},I_{22},I_{23}$ and $I_{24}$ are defined in an obvious way. The analysis of $I_{21}$ is identical to $I_1$ so from \eqref{pfthm7-8}  we have \begin{equation}\label{pfthm7-7}
            I_{21} = O_P\left\{\Omega_{0,2}^2\max_{j}b_{j}^2 + \big(\min_j \widebar{T}_j b_{j}\big)^{-1/2}(\log p)^{1/2}\Omega_{0,2}^2\right\}.
        \end{equation}
        For $I_{22}$, we have \begin{equation}\label{pfthm7-5}
            \begin{aligned}
                E(I_{22}) & = \frac{2}{n-h}\sum_{t=1}^{n-h}E\sth{E\pth{\|\widehat{\bX}_{t}-\widetilde{\bX}_{t}\|_{\cH,\infty}^2I\{\Omega_{t1}(1)\}\mid \bX_t}}\\
                & = \frac{2}{n-h}\sum_{t=1}^{n-h}E\sth{(1+|\bX_t^*|_\infty)^2 E\pth{\frac{\|\widehat{\bX}_{t}-\widetilde{\bX}_{t}\|_{\cH,\infty}^2}{(1+|\bX_t^*|_\infty)^2}I\sth{\Omega_{t1}(1)}\mid \bX_t}}.
            \end{aligned}
        \end{equation}
        Using \eqref{lemmaB8eq1} in Lemma \ref{LemmaPartial1}, we have \begin{equation*}
            \begin{aligned}
                E\pth{\frac{\|\widehat{\bX}_{t}-\widetilde{\bX}_{t}\|_{\cH,\infty}^2}{(1+|\bX_t^*|_\infty)^2}I\sth{\Omega_{t1}(1)}\mid \bX_t} \le C\int_0^\infty \delta\qth{1\wedge C_1p\exp\{-C_2\cT_t\min(\delta,\delta^2)\}}\mathrm{d}\delta,
            \end{aligned}
        \end{equation*}
        where $\cT_t=\min_j T_{tj}b_{j}$. By Lemma \ref{LemmaPartial2}, this integral is smaller than $C\log p/\cT_t$. Since we assume $T_{tj}$'s are of the same order as in Condition \ref{condpartial1}(\romannumeral 4) as well as $\log(p\vee n)/\min_{j}\widebar{T}_j b_{j}\rightarrow 0$ in the condition of Theorem \ref{thm7}, there exists a positive absolute constant $C$ such that $\cT_t \ge C\log p$. Using Condition \ref{condpartial1}(\romannumeral 6) we have \begin{align*}
            E(I_{22}) & \le \frac{C}{n-h}\sum_{t=1}^{n-h}\frac{\log p}{\cT_t}E\sth{(1+|\bX_{t}^*|_\infty)^2} \\
            &\le \frac{C}{n-h}\sum_{t=1}^{n-h}\pth{\frac{\log p}{\cT_t}}^{1/2}E\sth{(1+|\bX_{t}^*|_\infty)^2} \le C\big(\min_j\widebar{T}_j b_{j}\big)^{-1/2}(\log p)^{1/2}\Omega_{0,2}^2.
        \end{align*}
        These $C$'s are generic constants that may vary from line to line. For $I_{23}$, by \eqref{lemmaB8eq2}, $b_{j} \le 1$ and Condition \ref{condpartial1}(\romannumeral 5), we have \begin{equation}\label{pfthm7-6}
            \begin{aligned}
                E(I_{23}) & = \frac{2}{n-h}\sum_{t=1}^{n-h}E\left(\|\widetilde{\bX}_{t+h}-\bX_{t+h}\|_{\cH,\infty}^2\right) \\
                &\le \frac{2}{n-h}\sum_{t=1}^{n-h}E\left\{\max_j b_{j}^4(\bX_t^{(2)*})^2\right\} \le C\Omega_{0,2}^2\max_{j\in[p]} b_{j}^2.
            \end{aligned}
        \end{equation}
        For $I_{24}$, using similar argument as in \eqref{pfthm7-10}, we have $P(I_{24}\ne 0)\rightarrow 0$.
        Combining \eqref{pfthm7-7}, \eqref{pfthm7-5} and \eqref{pfthm7-6}, we obtain
        \begin{equation*}
            I_2 = O_P\left\{\Omega_{0,2}^2\max_j b_{j}^2 + \pth{\min_j \widebar{T}_j b_{j}}^{-1/2}\Omega_{0,2}^2\log^{1/2}( p)\right\}.
        \end{equation*}
        This finishes the proof of Theorem \ref{thm7}.

        \subsection{Proof of Theorem \ref{thm8}}

        We first have the following decomposition $$\sup\limits_{\theta\in[0,2\pi]}\|\widetilde{\bbf}_{\theta}-\bbf_\theta\|_{\max} \le \sup\limits_{\theta\in[0,2\pi]}\|\widetilde{\bbf}_{\theta}-\widehat{\bbf}_\theta\|_{\max} + \sup\limits_{\theta\in[0,2\pi]}\|\widehat{\bbf}_\theta-\bbf_\theta\|_{\max}.$$
        The last term is \(O_P(\cH_1)\) using results in Theorem \ref{thm2}. For the first term, let $\bX_{t+h} = \widehat{\bX}_{t+h} = 0$ if $t+h < 0$ or $t+h > n$. We have \begin{equation*}
            \begin{aligned}
                &\widetilde{\bbf}_{\theta}-\widehat{\bbf}_\theta = \frac{1}{2\pi}\sum_{h=-m_0}^{m_0}K(h/m_0)\pth{\widetilde{\bSigma}^{(h)}-\widehat{\bSigma}}\exp(-\mathbbm{i} h\theta)\\
                = & \frac{1}{2\pi }\sum_{h=-m_0}^{m_0}\frac{\exp(-\mathbbm{i} h\theta)K(h/m_0)}{n-|h|} \sum_{t = 1}^{n}\pth{\widehat{\bX}_t\otimes\widehat{\bX}_{t+h}^\T - \bX_t\otimes\widehat{\bX}_{t+h}^\T +\bX_t\otimes\widehat{\bX}_{t+h}^\T- \bX_t\otimes\bX_{t+h}^\T}\\
                = & \frac{1}{2\pi }\sum_{h=-m_0}^{m_0}\frac{\exp(-\mathbbm{i} h\theta)K(h/m_0)}{n-|h|} \sum_{t = 1}^{n}\pth{\widehat{\bX}_t\otimes\widehat{\bX}_{t+h}^\T - \bX_t\otimes\widehat{\bX}_{t+h}^\T}\\
                & + \frac{1}{2\pi }\sum_{h=-m_0}^{m_0}\frac{\exp(-\mathbbm{i} h\theta)K(h/m_0)}{n-|h|}\sum_{t = 1}^{n}\pth{\bX_t\otimes\widehat{\bX}_{t+h}^\T- \bX_t\otimes\bX_{t+h}^\T}\\
                := & \bI_1(\theta) + \bI_2(\theta)
            \end{aligned}
        \end{equation*}

        We first take a look at $\bI_2(\theta)$. We have \begin{equation*}
            \begin{aligned}
                \sup_{\theta\in[0,2\pi]}\|\bI_2(\theta)\|_{\cS,\max} \le &  \frac{C}{n}\sum_{h=-m_0}^{m_0}\left\|\sum_{t = 1}^{n}\bX_t\otimes\widehat{\bX}_{t+h}^\T- \bX_t\otimes\bX_{t+h}^\T\right\|_{\cS,\max}.
            \end{aligned}
        \end{equation*}
        Then using argument similar to the evaluation of $I_1$ in the proof of Theorem \ref{thm7}, we have \begin{equation}\label{pfthm8-1}
            \sup_{\theta\in[0,2\pi]}\|\bI_2(\theta)\|_{\cS,\max} = O_P\sth{m_0\Omega_{0,2}^2\max_{j}b_{j}^2 + m_0\pth{\min_j \widebar{T}_j b_{j}}^{-1/2}\Omega_{0,2}^2(\log p)^{1/2}}.
        \end{equation}
        We then deal with $\bI_1(\theta)$. We have \begin{equation*}
            \sup_{\theta\in[0,2\pi]}\|\bI_1(\theta)\|_{\cS,\max} \le \frac{C}{n}\sum_{h=-m_0}^{m_0}\left\|\sum_{t = 1}^{n}\widehat{\bX}_t\otimes\widehat{\bX}_{t+h}^\T- \widehat{\bX}_t\otimes\bX_{t+h}^\T\right\|_{\cS,\max}.
        \end{equation*}
        Then using argument similar to the evaluation of $I_2$ in the proof of Theorem \ref{thm7}, we have \begin{equation}\label{pfthm8-2}
            \begin{aligned}
                \sup_{\theta\in[0,2\pi]}\|\bI_1(\theta)\|_{\cS,\max} = O_P\sth{m_0\Omega_{0,2}^2\max_{j}b_{j}^2 + m_0\pth{\min_j \widebar{T}_j b_{j}}^{-1/2}\Omega_{0,2}^2(\log p)^{1/2}}.
            \end{aligned}
        \end{equation}
        The result of Theorem \ref{thm8} is the direct conclusion of (\ref{pfthm8-1}) and (\ref{pfthm8-2}).

        \section{Additional Technical Proofs}\label{supsec.tech}
\renewcommand{\theequation}{\thesection.\arabic{equation}}
  \setcounter{equation}{0}
                \subsection{Lemma \ref{leTechComposite2} and its proof}

        \begin{lemma}\label{leTechComposite2}

            (\romannumeral 1) Let $\bZ_t,t = 1,\dots,n$ be a $p$-dimensional martingale difference sequence or backward martingale difference sequence taking values in $L_2([0,1])$ with respect to the filtration $(\cG_t)_{t\in[n]}$. Let $l = 1\vee \log p$ and $q \ge 2$. Then we have  \begin{align*}
                \Bigg\|\Big\|\sum_{t = 1}^n\bZ_{t}\Big\|_{\cH,\infty}\Bigg\|_q^2 \le C\frac{q^2l}{(\log q)^2}\sum_{t=1}^n\|\|\bZ_t\|_{\cH,\infty}\|_q^2.
            \end{align*}

            (\romannumeral 2)  Let $\bZ_t,t = 1,\dots,n$ be a $p$-dimensional martingale difference sequence or backward martingale difference sequence taking values in $\eS = L_2([0,1])\otimes L_2([0,1])$ with respect to the filtration $(\cG_t)_{t\in[n]}$. Let $l = 1\vee \log p$ and $q \ge 2$. Then we have \begin{align*}
                \Bigg\|\Big\|\sum_{t = 1}^n\bZ_{t}\Big\|_{\cS,\max}\Bigg\|_q^2 \le C\frac{q^2l}{(\log q)^2}\sum_{t=1}^n\|\|\bZ_t\|_{\cS,\max}\|_q^2.
            \end{align*}
        \end{lemma}

        \begin{proof}
            We deal with (\romannumeral 1) and (\romannumeral 2) can be proved identically. We only prove results for martingale difference sequences; results for backward martingale difference sequences follow mutatis mutandis. Assume $s \ge 2$. We first verify that under the norm $\|\cdot\|_{\cH,s}$ the $p$-dimensional $L_2([0,1])$ curves consist of a $(2,D)$-smooth Banach space. For the concept of $(2,D)$-smooth Banach space, see (2.1) of \cite{pinelis1994optimum}.  We first need to derive the smoothness parameter $D$. For $\bX(\cdot),\bV(\cdot)$, define $\bY(t,\cdot) = \bX(\cdot) + t\bV(\cdot)$, and define $\widetilde{\bY}(t)$ as a $p$-dimensional vector with $\widetilde{Y}_{j}(t) = \|Y_{j}(t,\cdot)\|_{\cH} = (\|X_j\|_{\cH}^2 + 2t\langle X_j,V_j\rangle + t^2 \|V_j\|_{\cH}^2)^{1/2}$. Notice that $\|\bY(t,\cdot)\|_{\cH,s} = (\sum_{j = 1}^p \widetilde{Y}_{j}(t)^{s})^{1/s} $. This gives \begin{align*}
                \partial_t\|\bY(t,\cdot)\|_{\cH,s}^2 = & \frac{2}{s}\left(\sum_{j=1}^p \widetilde{Y}_{j}(t)^s \right)^{2/s-1}\sum_{j=1}^p \frac{s}{2}\widetilde{Y}_{j}(t)^{s-2}(2\langle X_j,V_j \rangle + 2t\|V_j\|_{\cH}^2)\\
                = & 2\left(\sum_{j=1}^p \widetilde{Y}_{j}(t)^s \right)^{2/s-1}\sum_{j=1}^p \widetilde{Y}_{j}(t)^{s-2}(\langle X_j,V_j \rangle + t\|V_j\|_{\cH}^2).
                \end{align*}
                For the second derivative, we have \begin{align*}
                \partial^2_{tt}\mid_{t = 0}\|\bY(t,\cdot)\|_{\cH,s}^2  =& 2\left(\sum_{j=1}^p \widetilde{Y}_{j}(0)^s \right)^{2/s-1}\sum_{j=1}^p\left\{(s-2)\widetilde{Y}_{j}(0)^{s-4}\langle X_j,V_j \rangle^2 + \widetilde{Y}_{j}(0)^{s-2}\|V_j\|_{\cH}^2\right\}\\
                & + \left(4-2s\right)\left(\sum_{j=1}^p \widetilde{Y}_{j}(0)^{s-2}\langle X_j,V_j \rangle\right)^2 \\
                \le & 2\left(\sum_{j=1}^p \widetilde{Y}_{j}(0)^s \right)^{2/s-1}\sum_{j=1}^p\left\{(s-2)\widetilde{Y}_{j}(0)^{s-4}\langle X_j,V_j \rangle^2 + \widetilde{Y}_{j}(0)^{s-2}\|V_j\|_{\cH}^2\right\}\\
                \le & 2\|\bX\|_{\cH,s}^{2-s}(s-1)\|\bX\|_{\cH,s}^{s-2}\|\bV\|_{\cH,s}^2 \le 2(s-1)\|\bV\|_{\cH,s}^2,
            \end{align*}
             where in the last line we use Cauchy--Schwarz inequality for $\langle X_j,V_j\rangle$ and H{\"o}lder's inequality for $\sum_{j=1}^p \|X_j\|_\cH^{s-2}\|V_j\|_\cH^{2}$. By equation (2.2) in \cite{pinelis1994optimum}, we have $(\otimes^p L_2([0,1]),\|\cdot\|_{\cH,s})$ is a $(2,\sqrt{s-1})$-smooth Banach space.  It is straightforward to verify that this Banach space is also separable. \color{black} By \color{blue} (4.3) in \color{black} Theorem 4.1 of \cite{pinelis1994optimum}, we have  \begin{align*}
                \left\|\left\|\sum_{t = 1}^n\bZ_t\right\|_{\cH,s}\right\|_q \le C\frac{q}{\log q} \left[\left\|\sup_{t\in[n]}\|\bZ_t\|_{\cH,s}\right\|_q + \sqrt{s-1}\left\|\left\{\sum_{t = 1}^n E(\|\bZ_t\|_{\cH,s}^2\mid \cG_{t-1})\right\}^{1/2}\right\|_q \right].
            \end{align*}
        For the first item, we have\begin{align*}
            &\left\|\sup_{t\in[n]}\|\bZ_t\|_{\cH,s}\right\|_q \le \left(\sum_{t = 1}^n \|\|\bZ_t\|_{\cH,s}\|_q^q \right)^{1/q} \le \left(\sum_{t = 1}^n \|\|\bZ_t\|_{\cH,s}\|_q^2 \right)^{1/2}.
        \end{align*}
        For the second item, using triangle inequality and Jensen's inequality, we have
        \begin{align*}
            \left\|\left\{\sum_{t = 1}^n E(\|\bZ_t\|_{\cH,s}^2\mid \cG_{t-1})\right\}^{1/2}\right\|_q \le \left\{\sum_{t=1}^n \|E(\|\bZ_t\|_{\cH,s}^2\mid \cG_{t-1})\|_{q/2}\right\}^{1/2} \le \left(\sum_{t=1}^n \|\|\bZ_t\|_{\cH,s}\|_q^2\right)^{1/2}.
        \end{align*}
        This implies \begin{align}\label{pflemB1eq1}
            \left\|\left\|\sum_{t = 1}^n\bZ_{t}\right\|_{\cH,s}\right\|_q^2 \le C\frac{q^2s}{(\log q)^2}\sum_{t=1}^n\|\|\bZ_t\|_{\cH,s}\|_q^2.
        \end{align}
        When dimension $p \le 7 < e^2$, we have $\|\bZ_t\|_{\cH,\infty} \le \|\bZ_t\|_{\cH,2} \le \sqrt{7}\|\bZ_t\|_{\cH,\infty}$. When dimension $p \ge 8 > e^2$, pick $s = \log p$, and then$\|\bZ_t\|_{\cH,\infty} \le \|\bZ_t\|_{\cH,\log p} \le p^{1/\log p}\|\bZ_t\|_{\cH,\infty} = e\|\bZ_t\|_{\cH,\infty}$. Combining these facts with (\ref{pflemB1eq1}) we obtain \begin{align*}
           \left\|\Big\|\sum_{t = 1}^n\bZ_{t}\Big\|_{\cH,\infty}\right\|_q^2 \le C\frac{q^2l}{(\log q)^2}\sum_{t=1}^n\|\|\bZ_t\|_{\cH,\infty}\|_q^2. 
        \end{align*}
        \end{proof} 
        
        \subsection{Lemma \ref{lemmaTechComposite3} and its proof}

        \begin{lemma}\label{lemmaTechComposite3}
        (\romannumeral 1) Let $\bX_1,\dots,\bX_n$ be a sequence of independent $p$-dimensional vector functions with each entry in $(L_2([0,1]),\|\cdot\|_{\cH})$. Assume they are centered and $\|\|X_{tj}\|_{\cH}\|_q < \infty$ for all $t\in[n],j\in[p]$ and some fixed $q\ge 2$. Let $\bS_n = \sum_{t=1}^n\bX_t, \sigma^2 = \max_{j\in[p]}\sum_{t=1}^n E(\|X_{tj}\|_{\cH}^2)$ and $l = 1\vee \log p$. Then for any $x > 0$, \begin{equation*}\label{lemmaB6-1}
            P(\|\bS_n\|_{\cH,\infty}\ge x) \le C_qx^{-q}l^q\sum_{t=1}^nE\{\|\bX_t\|_{\cH,\infty}^q\}+2p\exp\left(-\frac{x^2}{C\sigma^2}\right).
        \end{equation*}

        (\romannumeral 2) Let $\bX_1,\dots,\bX_n$ be a sequence of independent $p$-dimensional vector functions with each entry in $(L_2([0,1])\otimes L_2([0,1]),\|\cdot\|_{\cS})$. Assume they are centered and $\|\|X_{tj}\|_{\cS}\|_q < \infty$ for all $t\in[n],j\in[p]$ and $q\ge 2$. Let $\bS_n= \sum_{t=1}^n\bX_t, \sigma^2 = \max_{j\in[p]}\sum_{t=1}^n E(\|X_{tj}\|_{\cS}^2)$ and $l = 1\vee \log p$. Then for any $x > 0$, \begin{equation}\label{lemmaB6-2}
            P(\|\bS_n\|_{\cS,\infty}\ge x) \le C_qx^{-q}l^q\sum_{t=1}^nE\{\|\bX_t\|_{\cS,\infty}^q\}+2p\exp\left(-\frac{x^2}{C\sigma^2}\right).
        \end{equation}
        \end{lemma}

        \begin{proof}
        We prove (\romannumeral 2) and (\romannumeral 1) can be proved similarly. Using Theorem 4 in \cite{einmahl2008characterization}, we have 
            \begin{align*}
                P\{\|\bS_n\|_{\cS,\infty} \ge 2E(\|\bS_n\|_{\cS,\infty}) + x\} \le \exp\left(-\frac{x^2}{3\Lambda_n}\right) + C_qx^{-q}\sum_{t=1}^n E(\|\bX_t\|_{\cS,\infty}^q)
            \end{align*}
        Here $\Lambda_n = \sup\{\sum_{t=1}^n E\{f^2(\bX_t)\}:f\in B_1^*\}$, where $B_1^*$ is the unit ball of dual space of $(\otimes^p \{L_2([0,1])\otimes L_2([0,1])\},\|\cdot\|_{\cS,\infty})$. 

        In the following we characterize functions in $B_1^*$. Define inner product $\langle\cdot,\cdot\rangle_{\cS,2}$ that for any $\bV_1,\bV_2\in \otimes^p \{L_2([0,1])\otimes L_2([0,1])\}$, we have \begin{align*}
            \langle \bV_1,\bV_2\rangle_{\cS,2} = \sum_{j=1}^p\int_{[0,1]^2} V_{1j}(u,v)V_{2j}(u,v)\du\dv.
        \end{align*}
        Then it induces the norm $\|\cdot\|_{\cS,2}$, and $(\otimes^p \{L_2([0,1])\otimes L_2([0,1])\},\|\cdot\|_{\cS,2})$ can be easily verified to be a Hilbert space. Using Riesz representation theorem, any bounded linear functional $f:\otimes^p \{L_2([0,1])\otimes L_2([0,1])\}\rightarrow\eR$ on Hilbert space $(\otimes^p \{L_2([0,1])\otimes L_2([0,1])\},\|\cdot\|_{\cS,2})$ takes the form \begin{align}\label{reizerep}
            f(\bX_t) = \sum_{j=1}^p\int_0^1X_{tj}(u,v)Y_{j}(u,v)\du\dv,\quad \|\bY\|_{\cS,2} \le 1.
        \end{align}
        We next show that any bounded linear functional of $(\otimes^p \{L_2([0,1])\otimes L_2([0,1])\},\|\cdot\|_{\cS,\infty})$ is also a bounded linear functional of $(\otimes^p \{L_2([0,1])\otimes L_2([0,1])\},\|\cdot\|_{\cS,2})$. Assume $f$ is a bounded linear functional of $(\otimes^p \{L_2([0,1])\otimes L_2([0,1])\},\|\cdot\|_{\cS,\infty})$, then there exists $M$ such that $f(\bX) \le M\|\bX\|_{\cS,\infty}$. Since $\|\bX\|_{\cS,\infty}\le \|\bX\|_{\cS,2}$, we have $f(\bX) \le M\|\bX\|_{\cS,2}$. The linearity still holds. Thus it is a bounded linear functional of $(\otimes^p \{L_2([0,1])\otimes L_2([0,1])\},\|\cdot\|_{\cS,2})$.

        For $f\in B_1^*$, it is a bounded linear functional of $(\otimes^p \{L_2([0,1])\otimes L_2([0,1])\},\|\cdot\|_{\cS,2})$, so it has the representation of (\ref{reizerep}). Next we show that $\|\bY\|_{\cS,1}\le 1$. Assume otherwise $\|\bY\|_{\cS,1} > 1$. Define $X_j = 0$ if $\|Y_j\|_\cS =0$, and $X_j=Y_j/\|Y_j\|_\cS$ if $\|Y_j\|_\cS >0$. Let $\bX=(X_1,\dots,X_p)$, then we have $\|\bX\|_{\cS,\infty}=1, f(\bX) > 1$, which contradicts $f\in B_1^*$. Thus any function $f\in B_1^*$ takes the form \begin{align*}
            f(\bX_t) = \sum_{j=1}^p\int_0^1X_{tj}(u,v)Y_{j}(u,v)\du\dv,\quad \|\bY\|_{\cS,1} \le 1.
        \end{align*}
        This gives \begin{equation}\label{pflemmaB6-1}
            \begin{aligned}
                &\sum_{t=1}^n E\{f^2(\bX_t)\} = \sum_{t=1}^n E\left\{\left(\sum_{j=1}^p\int_0^1 X_{tj}(u,v)Y_j(u,v)\du\dv\right)^2\right\} \\
                \le & \sum_{t=1}^n E\left\{\left(\sum_{j=1}^p\|Y_j\|_{\cS}\|X_{tj}\|_{\cS}\right)^2\right\} \le \sum_{t=1}^n \sum_{j=1}^p \|Y_j\|_{\cS} E(\|X_{tj}\|_\cS^2) \le \max_{j\in[p]}\sum_{t=1}^n E(\|X_{tj}\|_\cS^2) = \sigma^2.
            \end{aligned}
        \end{equation}
        Thus $\Lambda_n \le \sigma^2$. The remaining task is to bound $E(\|\bS_n\|_{\cS,\infty})$. We first use symmetrization method. Define $\varepsilon_t,t\in[n]$ to be i.i.d. Rademacher random variables with $P(\varepsilon_t=1)=P(\varepsilon_t=-1)=1/2$. Let $X_{tj}'$ be an i.i.d. copy of $X_{tj}$. Define $\bX^n = (\bX_1,\dots,\bX_n)$. For any random variable $X$ that takes value in $L_2([0,1])\otimes L_2([0,1])$, recall that $\|X\|_{\cS}$ is a non-negative random variable. Then we have \begin{equation*}
            \begin{aligned}
                E(\|\bS_n\|_{\cS,\infty}) & = E\pth{\max_{j\in[p]}\nth{\sum_{t=1}^n X_{tj}}_\cS} = E\sth{\max_{j\in[p]}\nth{\sum_{t=1}^n X_{tj}-E\pth{X_{tj}'}}_\cS}\\
                & = E\qth{\max_{j\in[p]}\nth{E\sth{\sum_{t=1}^n \pth{X_{tj}-X_{tj}'}\mid\bX^n}}_\cS}\\
                & \le E\qth{\max_{j\in[p]}E\sth{\nth{\sum_{t=1}^n \pth{X_{tj}-X_{tj}'}}_\cS\mid\bX^n}}\\
                & \le E\sth{\max_{j\in[p]}\nth{\sum_{t=1}^n \pth{X_{tj}-X_{tj}'}}_\cS} = E\sth{\max_{j\in[p]}\nth{\sum_{t=1}^n \varepsilon_t\pth{X_{tj}-X_{tj}'}}_\cS}\\
                & \le 2E\pth{\max_{j\in[p]}\nth{\sum_{t=1}^n \varepsilon_tX_{tj}}_\cS} := E(I_1).
            \end{aligned}
        \end{equation*}
        In the first inequality we use Proposition 1.12 in \cite{pisier2016martingales}. Then we have \begin{equation*}
            \begin{aligned}
                E\pth{I_1\mid \bX^n} \le & 2E\sth{\max_{j\in[p]}\left|\nth{\sum_{t=1}^n\varepsilon_t X_{tj}}_\cS - E\pth{\nth{\sum_{t=1}^n\varepsilon_t X_{tj}}_\cS\mid \bX^n}\right| \mid \bX^n} \\
                & + 2\max_{j\in[p]}E\pth{\nth{\sum_{t=1}^n\varepsilon_t X_{tj}}_\cS \mid \bX^n} := 2I_2 + 2I_3.
            \end{aligned}
        \end{equation*}
        We first deal with $I_3$. We have \begin{equation*}
            \begin{aligned}
                I_3 & = \max_{j\in[p]}E\qth{\sth{\int_0^1\pth{\sum_{t=1}^n\varepsilon_tX_{tj}(u,v)}^2\du\dv}^{1/2}\mid \bX^n}\\
                & \le \max_{j\in[p]}\qth{E\sth{\int_0^1\pth{\sum_{t=1}^n\varepsilon_tX_{tj}(u,v)}^2\du\dv\mid \bX^n}}^{1/2}\\
                & = \max_{j\in[p]}\qth{\sum_{t_1,t_2 =1}^nE\sth{\int_0^1\varepsilon_{t_1}\varepsilon_{t_2}X_{t_1j}(u,v)X_{t_2j}(u,v)\du\dv\mid\bX^n}}^{1/2}\\
                & = \max_{j\in[p]}\sth{\sum_{t=1}^n\int_0^1X_{tj}(u,v)^2\du\dv}^{1/2}  = \pth{\max_{j\in[p]}\sum_{t=1}^n\|X_{tj}\|_\cS^2\du}^{1/2}.
            \end{aligned}
        \end{equation*}
        Next we deal with $I_2$. Recall that for any $X_{tj}$ it is a random variable defined on probability space $(\Omega,\cF,P)$ and takes value in $L_2([0,1])\otimes L_2([0,1])$ endowed with Borel algebra $\mathcal{B}$ induced by norm $\|\cdot\|_2$. The space $(L_2([0,1])\otimes L_2([0,1]),\mathcal{B})$ is a Polish space, and thus nice and admitting existence of regular conditional probabilities (see Theorem 4.1.17 in \citealt{durrett2019probability}). This enables us to generalize many unconditional concentration inequalities to conditional cases. Using McDiarmid's Inequality (see Theorem 2.9.1 in \citealt{vershynin2018high}), we have \begin{equation*}
            \begin{aligned}
                P\sth{\left|\nth{\sum_{t=1}^n\varepsilon_tX_{tj}}_\cS - E\pth{\nth{\sum_{t=1}^n\varepsilon_tX_{tj}}_\cS}\right|>x \mid \bX^n} & \le \exp\pth{-\frac{x^2}{2\sum_{t=1}^n\|X_{tj}\|_\cS^2}} \quad \mbox{a.s.}
            \end{aligned}
        \end{equation*}
        Define $\tilde{\psi}(x) = \exp(x^2)-1$ and the Orlicz norm $\nth{X}_{\tilde{\psi}} = \inf\sth{c>0:E\sth{\tilde{\psi}\pth{|X|/c}}}\le 1$. This gives $\left\|E\sth{\nth{\sum_{t=1}^n\varepsilon_tX_{tj}}_\cS - E\pth{\nth{\sum_{t=1}^n\varepsilon_tX_{tj}}_\cS}\mid \bX^n}\right\|_{\tilde{\psi}} \le C\pth{\sum_{t=1}^n\|X_{tj}\|_\cS^2}^{1/2}$ holds almost surely. Here $C$ is an absolute constant that may vary from line to line. Using Lemma 2.2.2 in \cite{van1997weak}, we obtain $I_2 \le C\pth{\log p\max_{j\in[p]}\sum_{t=1}^n\|X_{tj}\|_\cS^2}^{1/2}$. Here we use the fact that for any random variable $X$, $E(|X|)\le E(|X|^2)^{1/2} \le \|X\|_{\tilde{\psi}}$, see Section 2.2 in \cite{van1997weak}. Combining the bounds of $I_2$ and $I_3$, we obtain \begin{equation}\label{pflemmaB2-5}
            E(I_1\mid \bX^n) \le Cl^{1/2}\sth{\max_{j\in[p]}\sum_{t=1}^n\|X_{tj}\|_\cS^2}^{1/2}.
        \end{equation}
        Using Lemma 9 in \cite{chernozhukov2015comparison}, we obtain \begin{equation}\label{pflemmaB2-6}
            E\sth{\max_{j\in[p]}\sum_{t=1}^n\|X_{tj}\|_\cS^2} \le C\log pE\pth{\max_{j\in[p]}\max_{t\in[t]}\|X_{tj}\|_\cS^2} + C\max_{j\in[p]}E\pth{\sum_{t=1}^n\|X_{tj}\|_\cS^2}.
        \end{equation}
        Combining (\ref{pflemmaB2-5}) and (\ref{pflemmaB2-6}), we have there exists an absolute large enough constant $C'$, such that \begin{equation}\label{pflemmaB6-2}
        \begin{aligned}
            E(I_1) = E\sth{E\pth{I_1\mid\bX^n}} &\le C'l^{1/2}\sigma + C'l\sth{E\pth{\max_{j\in[p]}\max_{t\in[n]}\|X_{tj}\|_\cS^2}}^{1/2}\\
            & = C'l^{1/2}\sigma + C'l\qth{E\sth{\pth{\max_{t\in[n]}\|\bX_{t}\|_{\cS,\infty}}^2}}^{1/2}\\
            & \le C'l^{1/2}\sigma + C'l\qth{E\sth{\pth{\max_{t\in[n]}\|\bX_{t}\|_{\cS,\infty}}^q}}^{1/q}\\
            & \le C'\left\{\sigma l^{1/2}+\left(\sum_{t=1}^n E(\|\bX_t\|_{\cS,\infty}^q)\right)^{1/q}l\right\}.
        \end{aligned}
        \end{equation}

        Assume first $x < 4C'\left\{\sigma l^{1/2}+\left(\sum_{t=1}^n E(\|\bX_t\|_{\cS,\infty}^q)\right)^{1/q}l\right\}$. Under this assumption, if $\sigma l^{1/2} < \left(\sum_{t=1}^n E(\|\bX_t\|_{\cS,\infty}^q)\right)^{1/q}l$, we have $x < 8C'\left(\sum_{t=1}^n E(\|\bX_t\|_{\cS,\infty}^q)\right)^{1/q}l$, and we can find a large enough constant $C_q$ such that $C_qx^{-q}l^q\sum_{t=1}^nE\{\|\bX_t\|_{\cS,\infty}^q\} > 1$. If $\sigma l^{1/2} > \left(\sum_{t=1}^n E(\|\bX_t\|_{\cS,\infty}^q)\right)^{1/q}l$, we have $x < 8C'\sigma l^{1/2}$, and we can find a constant $C$ large enough such that $2p\exp(-x^2/C\sigma^2) > 1$. To sum up, we can pick large enough constant $C_q,C$ such that, when $x < 4C'\left\{\sigma l^{1/2}+\left(\sum_{t=1}^n E(\|\bX_t\|_{\cS,\infty}^q)\right)^{1/q}l\right\}$, the upper bound in (\ref{lemmaB6-2}) trivially holds.
        
        Next, for the case $x > 4C'\left\{\sigma l^{1/2}+\left(\sum_{t=1}^n E(\|\bX_t\|_{\cS,\infty}^q)\right)^{1/q}l\right\}$, combining (\ref{pflemmaB6-1}) and (\ref{pflemmaB6-2}), we have \begin{align*}
            P(\|\bS_n\|_{\cS,\infty}\ge x) &\le P\{\|\bS_n\|_{\cS,\infty} \ge 2E(\|\bS_n\|_{\cS,\infty}) + x/2\}\\
            & \le \exp\left(-\frac{x^2}{12\sigma^2}\right)+2^{-q}C_qx^{-q}\sum_{t=1}^n E(\|\bX_t\|_{\cS,\infty}^q).
        \end{align*}
        Thus the upper bound in (\ref{lemmaB6-2}) holds. Combining the above argument, we see that there exist large enough constants $C_q,C$ such that (\ref{lemmaB6-2}) holds for all $x$ and we finish the proof.
        \end{proof}
        
        \subsection{Lemma \ref{lemmaTechComposite1} and its proof}

        \begin{lemma}\label{lemmaTechComposite1}
         Recall \eqref{Sec1-1}. Define the projection operator $\cP_t$ as $\cP_t\bX = E(\bX\mid\cF_t)-E(\bX\mid\cF_{t-1})$. Define $l = 1\vee \log p$. Recall for two curves $\bX(\cdot),\bY(\cdot)$, their tensor product is defined as $(\bX\otimes\bY^\T) (u,v) = \bX(u)\bY(v)^\T$. Then we have the following bounds: \begin{enumerate}[label=(\alph*)]
                \item For $h \ge 0$, $\|\|\cP_{t-h}\bX_t\|_{\cH,\infty}\|_q \le \omega_{h,q}$ and $\|\|\cP_{t-h}X_{tj}\|_\cH\|_q \le \delta_{h,q,j}$.
                \item For any $c_t,t\in[n]$, we have $\|\|\sum_{t = 1}^n c_t\bX_t\|_{\cH,\infty}\|_q \le C\Omega_{0,q}ql^{1/2}(\sum_{t = 1}^nc_t^2)^{1/2}/\log q$.
                \item For any $c_{st},s,t\in[n]$, we have that  $\|\|\sum_{s,t=1}^n c_{st}\{\bX_s\otimes\bX_t^\T - E\{\bX_s\otimes\bX_t^\T\}\|_{\cS,\max}\|_q \le Cn^{1/2}q^2l\Omega_{0,2q}^2\cC/(\log q)^2$. Here $\cC = \max\{\max_{t\in[n]}(\sum_{s=1}^n c_{st}^2)^{1/2},\max_{s\in[n]}(\sum_{t=1}^n c_{st}^2)^{1/2}\}$.
                \item For any $c_{st},s,t\in[n]$, and $j,k\in[p]$, we have  $\|\|\sum_{s,t=1}^n c_{st}\{X_{sj}\otimes X_{tk} - E\{X_{sj}\otimes X_{tk}\}\|_{\cS}\|_q \le Cn^{1/2}q^2\Delta_{0,2q,j}\Delta_{0,2q,k}\cC/(\log q)^2$. Here $\cC$ is defined in the same way as in (c).
            \end{enumerate}
        \end{lemma}

    \begin{proof}
        (a) For $h \ge 0$, $\cP_{t-h}\bX_t = E\{\bX_t\mid\cF_{t-h}\}-E\{\bX_t\mid\cF_{t-h-1}\}$. This gives \begin{align*}
            \|\|\cP_{t-h}\bX_t(u)\|_{\cH,\infty}\|_q & = \|\|E\{\bX_t-\bX_{t,\{t-h\}}\mid\cF_{t-h-1}\}\|_{\cH,\infty}\|_q\\
            & \le \left\|\max_{j\in[p]}E\{\|\bG_j(\cdot,\cF_{t})-\bG_j(\cdot,\cF_{t,\{t-h\}})\|_{\cH}\mid\cF_{t-h-1}\}\right\|_q\\
            & \le \|E\{\|\bG(\cdot,\cF_t)-\bG(\cdot,\cF_{t,\{t-h\}})\|_{\cH,\infty}\mid\cF_{t-h-1}\}\|_q\\
            & \le \|\|\bG(\cdot,\cF_t)-\bG(\cdot,\cF_{t,\{t-h\}})\|_{\cH,\infty}\|_q = \omega_{t,q}.
        \end{align*}
        The first inequality uses Proposition 1.12 in \cite{pisier2016martingales}. The second and third inequalities follow from Jensen's inequality. The same argument gives $\|\|\cP_{t-h}\bX_{tj}(u)\|_\cH\|_q \le \delta_{h,q,j}$.

        (b) We write $\sum_{t = 1}^n c_t\bX_t = \sum_{r = 0}^\infty\sum_{t = 1}^n c_t\cP_{t-r}\bX_t$. For any fixed $r$, $\cP_{t-r}\bX_t,t = 1,\dots,n$ is a martingale difference of Banach space with respect to $\cF_{t-r}$. Using Lemma \ref{leTechComposite2}, we have \begin{align*}
            \left\|\left\|\sum_{t=1}^nc_t\bX_t\right\|_{\cH,\infty}\right\|_q &\le \sum_{r=0}^\infty \left\|\left\|\sum_{t=1}^nc_t\cP_{t-r}\bX_t\right\|_{\cH,\infty}\right\|_q\\
            & \le C\left(\frac{q^2l}{(\log q)^2}\right)^{1/2}\sum_{r = 0}^\infty \left(\sum_{t=1}^n c_t^2\right)^{1/2}\omega_{r,q} = C\Omega_{0,q}\frac{ql^{1/2}}{\log q}\left(\sum_{t=1}^n c_t^2\right)^{1/2}.
        \end{align*}

        (c) We write $\sum_{s,t=1}^n c_{st}\{\bX_s\otimes\bX_t^\T - E\{\bX_s\otimes\bX_t^\T\} = \sum_{r = -\infty}^n\cP_r(\sum_{s,t=1}^n c_{st}\bX_s\otimes\bX_t^\T)$. Recall that $\bX_s(\cdot)=\bG(\cdot,\cF_s)$. Denote $\bX_{s,\{r\}}(\cdot)=\bG(\cdot,\cF_{s,\{r\}})$, and the definition of $\cF_{s,\{r\}}$ can be seen in Definition \ref{def1}. For any fixed $r$, using Jensen's inequality and triangle inequality, we have \begin{align*}
            \left\|\left\|\cP_r\left(\sum_{s,t=1}^n c_{st}\bX_s\otimes\bX_t^\T\right)\right\|_{\cS,\max}\right\|_q & \le \left\|\left\|\sum_{s,t=1}^n c_{st}\left(\bX_s\otimes\bX_t^\T-\bX_{s,\{r\}}\otimes\bX_{t,\{r\}}^\T\right)\right\|_{\cS,\max}\right\|_q \\
            & \le I_{r1} + I_{r2},
        \end{align*}
        where \begin{align*}
            I_{r1} & = \left\|\left\|\sum_{s,t=1}^n c_{st}\left(\bX_s-\bX_{s,\{r\}}\right)\otimes\bX_{t}^\T\right\|_{\cS,\max}\right\|_q,\\
            I_{r2} & = \left\|\left\|\sum_{s,t=1}^n c_{st}\bX_{s,\{r\}}\otimes\left(\bX_t^\T-\bX_{t,\{r\}}^\T\right)\right\|_{\cS,\max}\right\|_q.
        \end{align*}
        Notice that $\|X_{sj}\otimes X_{tk}\|_{\cS} = \{\int_0^1\int_0^1 X_{sj}^2(u)X_{tk}^2(v) \du\dv\}^{1/2} = \{\int_0^1 X_{sj}^2(u)\du\}^{1/2}\{\int_0^1 X_{tk}^2(v)\dv\}^{1/2}  = \|X_{sj}\|_{\cH}\|X_{tk}\|_{\cH}$. By Cauchy--Schwarz inequality and the result in (b), it follows that  \begin{align*}
            I_{r1} & \le \sum_{s=1}^n \|\|\bX_s-\bX_{s,\{r\}}\|_{\cH,\infty}\|_{2q} \left\|\left\|\sum_{t=1}^n c_{st}\bX_t\right\|_{\cH,\infty}\right\|_{2q}\le C\Omega_{0,2q}q(\log q)^{-1}l^{1/2}\cC\sum_{s=1}^n\omega_{s-r,2q}.
        \end{align*}
        Via some elementary calculation we can obtain \begin{align*}
            \sum_{r=-\infty}^n I_{r1}^2 & \le C\{\Omega_{0,2q}^2q^2l/(\log q)^2\}\cC^2\sum_{r=-\infty}^n\left(\sum_{s=1}^n\omega_{s-r,2q}\right)^2 \le C\Omega_{0,2q}^4\cC^2nq^2l/(\log q)^2.
        \end{align*}
        The same upper bound applies to $I_{r2}$. Using results in Lemma \ref{leTechComposite2},  we have \begin{align*}
            & \left\|\left\|\sum_{s,t=1}^n c_{st}\{\bX_s\otimes\bX_t^\T - E(\bX_s\otimes\bX_t^\T)\}\right\|_{\cS,\max}\right\|_q^2 \\
            & \le C\frac{q^2l}{(\log q)^2}\sum_{r=-\infty}^n \left\|\left\|\cP_r\left(\sum_{s,t=1}^n c_{st}\bX_s\otimes\bX_t^\T\right)\right\|_{\cS,\max}\right\|_q^2\\
            & \le C\frac{q^2l}{(\log q)^2}\sum_{r=-\infty}^n(I_{r1}^2+I_{r2}^2) \le C\frac{nq^4l^2}{(\log q)^{4}}\Omega_{0,2q}^4\cC^2.
        \end{align*}
        This implies the result in (c).

        (d) Similar to (c), we have the following decomposition $\sum_{s,t=1}^n c_{st}\{X_{sj}\otimes X_{tk} - E\{X_{sj}\otimes X_{tk}\} = \sum_{r = -\infty}^n\cP_r(\sum_{s,t=1}^n c_{st}X_{sj}\otimes X_{tk})$, and we have \begin{align*}
            \left\|\left\|\cP_r\left(\sum_{s,t=1}^n c_{st}X_{sj}\otimes X_{tk}\right)\right\|_{\cS}\right\|_q & \le \left\|\left\|\sum_{s,t=1}^n c_{st}\left(X_{sj}\otimes X_{tk}-X_{sj,\{r\}}\otimes X_{tk,\{r\}}\right)\right\|_{\cS}\right\|_q \\
            & \le I_{r1} + I_{r2},
        \end{align*}
        where \begin{align*}
            I_{r1} & = \left\|\left\|\sum_{s,t=1}^n c_{st}\left(X_{sj}-X_{sj,\{r\}}\right)\otimes X_{tk}\right\|_{\cS}\right\|_q,\\
            I_{r2} & = \left\|\left\|\sum_{s,t=1}^n c_{st} X_{sj,\{r\}}\otimes\left( X_{tk}-X_{tk,\{r\}}\right)\right\|_{\cS}\right\|_q.
        \end{align*}
        By Cauchy--Schwarz inequality and the result in (b), it follows that \begin{align*}
            I_{r1} & \le \sum_{s=1}^n \|\|X_{sj}-X_{sj,\{r\}}\|_{\cH}\|_{2q} \left\|\Big\|\sum_{t=1}^n c_{st}X_{tk}\Big\|_{\cH}\right\|_{2q}\le C\Delta_{0,2q,k}q(\log q)^{-1}\cC\sum_{s=1}^n\delta_{s-r,2q,j}.
        \end{align*}
        Via some elementary calculation we can obtain \begin{align*}
            \sum_{r=-\infty}^n I_{r1}^2 & \le C\{\Delta_{0,2q,k}^2 q^2/(\log q)^2\}\cC^2\sum_{r=-\infty}^n\left(\sum_{s=1}^n\delta_{s-r,2q,j}\right)^2 \le C\Delta_{0,2q,j}^2\Delta_{0,2q,k}^2\cC^2nq^2/(\log q)^2.
        \end{align*}
        The same upper bound applies to $I_{r2}$. Using result in Lemma \ref{leTechComposite2},  we have \begin{align*}
            & \left\|\left\|\sum_{s,t=1}^n c_{st}\{X_{sj}\otimes X_{tk} - E(X_{sj}\otimes X_{tk})\}\right\|_{\cS}\right\|_q^2 \le C\frac{q^2}{(\log q)^2}\sum_{r=-\infty}^n \left\|\left\|\cP_r\left(\sum_{s,t=1}^n c_{st}X_{sj}\otimes X_{tk}\right)\right\|_{\cS}\right\|_q^2\\
            & \le C\frac{q^2}{(\log q)^2}\sum_{r=-\infty}^n(I_{r1}^2+I_{r2}^2) \le C\frac{nq^4}{(\log q)^{4}}\Delta_{0,2q,j}^2\Delta_{0,2q,k}^2\cC^2.
         \end{align*}
        This finishes the proof of (d).
    \end{proof}
        
        \subsection{Lemma \ref{lemmaTechComposite4} and its proof }

        \begin{lemma}\label{lemmaTechComposite4}
            For the stationary process $\bX_t(\cdot) = \bG(\cdot,\cF_t)$ with innovation $\cF_t=(\cdots,\varepsilon_{t-1},\varepsilon_t)$ defined in (\ref{Sec1-1}), assume it is centered. Let $B$ be a positive integer that is smaller than $n$, and define $\eta_d = (\varepsilon_{(d-1)B+1},\dots,\varepsilon_{dB})$ for $d\in\eZ$. 
            
            (\romannumeral 1) For $k\in \eN$ such that $Bk\le n$ and $h\in\eN$, define $\bV_k = \sum_{t=(k-1)B+1}^{kB\wedge n}\sum_{1\le s\le t\le n}a_{st}\bX_s\otimes\bX_t^\T$, and $\bV_{k,h} = E(\bV_k \mid \eta_{k-h},\dots,\eta_{k})$. Assume that $a_{st}=0$ if $|s-t| \ge B$ and $|a_{st}| \le 1$. Then for $h\ge 2$, there exists some constant $C_q$ such that \begin{align*}
                \|\|\bV_{k,h}-\bV_{k,h-1}\|_{\cS,\max}\|_{q/2} \le C_q(1\vee\log p)B\Omega_{0,q}\sum_{d=(h-2)B+1}^{(h+1)B}\omega_{d,q}.
            \end{align*}
            The same bound also applies to $\bV_k = \sum_{t=(k-1)B+1}^{kB\wedge n}\sum_{1\le t\le s\le n}a_{st}\bX_s\otimes\bX_t^\T$.

            (\romannumeral 2) For $k\in \eN$ such that $Bk\le n$ and $h\in\eN$, define $\bV^*_k = \sum_{t=(k-1)B+1}^{kB\wedge n}a_{t}\bX_{t-B}\otimes\bX_t^\T$, and $\bV^*_{k,h} = E(\bV^*_k \mid \eta_{k-h},\dots,\eta_{k})$. Assume $|a_t|\le 1$. Then for $h\ge 2$, there exists some constant $C_q$ such that \begin{align*}
                \|\|\bV^*_{k,h}-\bV^*_{k,h-1}\|_{\cS,\max}\|_{q/2} \le C_q(1\vee\log p)^{1/2}B^{1/2}\Omega_{0,q}\sum_{d=(h-2)B+1}^{(h+1)B}\omega_{d,q}.
            \end{align*}
            The same bound also applies to $\bV_k^* = \sum_{t=(k-1)B+1}^{kB\wedge n}a_{t}\bX_t\otimes\bX_{t-B}^\T$.

           (\romannumeral 3) For $j_1,j_2\in[p]$, let $V_{k,j_1j_2,h},V_{k,j_1j_2,h}^*$ be the $(j_1,j_2)$-th element of $\bV_{k,h}, \bV_{k,h}^*$ respectively. Assume all assumptions in (\romannumeral 1) and (\romannumeral 2). Then we have \begin{align*}
              &\|\|V_{k,j_1j_2,h}-V_{k,j_1j_2,h-1}\|_{\cS,\max}\|_{q/2} \\
              &\le C_qB \pth{ \Delta_{0,q,j_1}\sum_{d=(h-2)B+1}^{(h+1)B}\delta_{d,q,j_2} + \Delta_{0,q,j_2}\sum_{d=(h-2)B+1}^{(h+1)B}\delta_{d,q,j_1}},\\
              &\|\|V^*_{k,j_1j_2,h}-V^*_{k,j_1j_2,h-1}\|_{\cS,\max}\|_{q/2} \\
              &\le C_qB^{1/2}\pth{ \Delta_{0,q,j_1}\sum_{d=(h-2)B+1}^{(h+1)B}\delta_{d,q,j_2} + \Delta_{0,q,j_2}\sum_{d=(h-2)B+1}^{(h+1)B}\delta_{d,q,j_1}}.
            \end{align*}
        \end{lemma}

        \begin{proof}
            (\romannumeral 1) We prove the result for $\bV_k = \sum_{t=(k-1)B+1}^{kB\wedge n}\sum_{1\le s\le t\le n}a_{st}\bX_s\otimes\bX_t^\T$ and the other case can be identically dealt with. We define the truncated innovation sequence $\cF_{a}^b = (\varepsilon_{a},\dots,\varepsilon_b)$ for $a\le b$. Define truncated projection operator $\cP_{a,b}\bX = E(\bX\mid\cF_a^b)-E(\bX\mid\cF_{a+1}^b)$. For $h\ge 2$, we write \begin{equation}\label{PfLB4eq1}
            \begin{aligned}
                \bV_{k,h}-\bV_{k,h-1} &= E(\bV_k\mid \eta_{k-h},\dots,\eta_k)-E(\bV_k\mid\eta_{k-h+1},\dots,\eta_k)\\
                &= \sum_{m=1}^B \sth{E(\bV_k\mid \cF^{kB}_{(k-h-1)B+m})-E(\bV_k\mid \cF^{kB}_{(k-h-1)B+m+1})}\\
                & = \sum_{m=1}^B \cP_{(k-h-1)B+m,kB}\bV_k.
            \end{aligned}\end{equation}
            Using Jensen's inequality and triangle inequality, we have \begin{align*}
            &\|\|\cP_{(k-h-1)B+m,kB}\bV_k\|_{\cS,\max}\|_{q/2} \le I_1 + I_2,\\
            & I_1 = \left\|\left\|\sth{\sum_{t=(k-1)B+1}^{kB\wedge n}\left(\bX_t-\bX_{t,\{(k-h-1)B+m\}}\right)}\otimes\pth{\sum_{s = (t-B)\vee 1}^t a_{st}\bX_s^\T}\right\|_{\cS,\max}\right\|_{q/2},\\
            & I_2 = \left\|\left\|\sth{\sum_{s=\{(k-2)B+1\}\vee 1}^{kB\wedge n}\left(\bX_s-\bX_{s,\{(k-h-1)B+m\}}\right)}\otimes\pth{\sum_{t=s}^{(s+B)\wedge n}a_{st}\bX_{t,\{(k-h-1)B+m\}}^\T}\right\|_{\cS,\max}\right\|_{q/2}.
        \end{align*}
        Notice that for $Y_1,Y_2\in L_2([0,1])$, we have $\|Y_1\otimes Y_2\|_\cS = \|Y_1\|_\cH\|Y_2\|_\cH$. Thus we can apply H{\"o}lder's inequality and (a), (b) of Lemma \ref{lemmaTechComposite1} to obtain that \begin{equation}\label{pflemmaB8-1}
            \begin{aligned}
             I_1 &\le \sum_{t=(k-1)B+1}^{(kB)\wedge n} \|\|\bX_t-\bX_{t,\{(k-h-1)B+m\}}\|_{\cH,\infty}\|_q \left\|\left\|\sum_{s=(t-B)\vee 1}^t a_{st}\bX_s\right\|_{\cH,\infty}\right\|_q\\
            & \le C_q(1\vee\log p)^{1/2}B^{1/2}\Omega_{0,q} \sum_{t=(k-1)B+1}^{kB\wedge n}\omega_{t-(k-h-1)B-m,q}.   
            \end{aligned}
        \end{equation}
        Here $C_q$ is a constant that only depends on $q$. Similarly, we have $$I_2 \le C_q(1\vee\log p)^{1/2}B^{1/2}\Omega_{0,q}\sum_{s=(k-2)B+1}^{kB}\omega_{s-(k-h-1)B-m,q}.$$ Combining the results above, we have \begin{align*}
            \|\|\cP_{(k-h-1)B+m,kB}\bV_k\|_{\cS,\max}\|_{q/2} \le C_q(1\vee\log p)^{1/2}B^{1/2}\Omega_{0,q}\sum_{d=(h-1)B-m+1}^{(h+1)B-m}\omega_{d,q}.
        \end{align*}
        Notice that $\{\cP_{(k-h-1)B+m,kB}\bV_k:1\le m\le B\}$ forms a backward martingale differences with respect to $\left(\cF_{(k-h-1)B+m}^{kB}\right)_{1\le m\le B}$. Thus by Lemma \ref{leTechComposite2}, we have \begin{align*}
            \|\|\bV_{k,h}-\bV_{k,h-1}\|_{\cS,\max}\|_{q/2}^2 &\le C_q(1\vee\log p)\sum_{m=1}^B\|\cP_{(k-h-1)B+m,kB}\|_{q/2}^2\\
            & \le C_q(1\vee\log p)^2B\Omega_{0,q}^2\sum_{m=1}^B\left(\sum_{d=(h-1)B-m+1}^{(h+1)B-m} \omega_{d,q}\right)^2\\
            & \le C_q(1\vee\log p)^2B^2\Omega_{0,q}^2\left(\sum_{d=(h-2)B+1}^{(h+1)B}\omega_{d,q}\right)^2.
        \end{align*}
        This finishes the proof of (\romannumeral 1).

        (\romannumeral 2) We explain why there is an elimination of $(1\vee \log p)^{1/2}B^{1/2}$ terms in the bound of $\|\|\bV_{k,h}-\bV_{k,h-1}\|_{\cS,\max}\|_{q/2}$. We still have the decomposition in \eqref{PfLB4eq1} with $V_{k,h},V_{k,h-1},V_k$ substituted by $V_{k,h}^*,V_{k,h-1}^*,V_k^*$. Compared to (\ref{pflemmaB8-1}), we have \begin{align*}
            \|\|\cP_{(k-h-1)B+m,kB}\bV_k^*\|_{\cS,\max}\|_{q/2} & \le \sum_{t=(k-1)B+1}^{(kB)\wedge n} \|\|\bX_t-\bX_{t,\{(k-h-1)B+m\}}\|_{\cH,\infty}\|_q \left\|\left\| a_{t}\bX_{t-B}\right\|_{\cH,\infty}\right\|_q\\
            & \le C_q\Omega_{0,q} \sum_{t=(k-1)B+1}^{kB\wedge n}\omega_{t-(k-h-1)B-m,q}. 
        \end{align*}
        In the above calculation we use the fact $|a_t|\le 1$ and $\|\|\bX_t\|_{\cH,\infty}\|_q\le \sum_{s=-\infty}^t \|\|\cP_s\bX_t\|_{\cH,\infty}\|_{q} \le \sum_{s=-\infty}^t\omega_{t-s,q}=\Omega_{0,q}$. The remaning derivation remains the same. 

        (\romannumeral 3) Comparing the result of (\romannumeral 3) with those of (\romannumeral 1) and (\romannumeral 2), we make two major changes. The first difference is the elimination of the factor \(1 \vee \log p\), which is due to our consideration of one-dimensional cases. The second difference involves the dependence measures  \(\Omega_{0,q}\) and \(\sum_{d=(h-2)B+1}^{(h+1)B} \omega_{d,q}\). Here, \(\Omega_{0,q}\) is changed to \(\max_{j\in[p]} \Delta_{0,q,j}\), and the summation \(\sum_{d=(h-2)B+1}^{(h+1)B} \omega_{d,q}\) is transformed into the summation of one-dimensional physical dependence measure, which is given by \(\max_{j\in[p]} \sum_{d=(h-2)B+1}^{(h+1)B} \delta_{d,q,j}\).
        \end{proof}    
    
        \subsection{Lemma \ref{nonGaussianConcen} and its proof }

    \begin{lemma}\label{nonGaussianConcen}
        Consider the quadratic form $\bQ_n = \sum_{1\le s\le t\le n}a_{st}\bX_s\otimes\bX_t^\T$. Assume $\bX_t$ comes from (\ref{Sec1-1}) and it is mean zero. Assume $\sup_{s,t\in[n]}|a_{st}|\le 1$ and $a_{st}=0$ if $|t-s| > B$ and $B<n$. Then we have \begin{align*}
            P\left\{\|\bQ_n-E(\bQ_n)\|_{\cS,\max}\ge x\right\}\le & C_{q,\alpha}x^{-q/2}(1\vee \log p)^{5q/4}\|\|\bX_1\|_{\cH,\infty}\|_{q,\alpha}^qF_{n,B}' \\
            &+ C_\alpha p^2\exp\left\{-\frac{x^2}{C_\alpha\left(\bPhi_{4,\alpha}^\bX\right)^2nB}\right\}.
        \end{align*}
        Here $F'_{n,B}=nB^{q/2-1}$ (resp., $nB^{q/2-1}+n^{q/4-\alpha q/2}B^{q/4}$) if $\alpha > 1/2-2/q$ (resp., $\alpha \le 1/2-2/q$). The same bound also holds for $\bQ_n = \sum_{1\le t\le s\le n}a_{st}\bX_s\otimes\bX_t^\T$.
    \end{lemma}

    \begin{proof}
        We prove the result for $\bQ_n = \sum_{1\le s\le t\le n}a_{st}\bX_s\otimes\bX_t^\T$ and the other case can be similarly handled. Let $K=\lceil n/B\rceil \ge 2$. For $k\in[K]$, define $\bV_k = \sum_{t=(k-1)B+1}^{kB\wedge n}\sum_{s=1}^t a_{st}\bX_s\otimes\bX_t^\T$, and $V_{kij} = \sum_{t=(k-1)B+1}^{(kB)\wedge n}\sum_{s=1}^t a_{st}X_{si}\otimes X_{tj}$ for $i,j\in[p]$. Define the innovation set $\eta_k=(\varepsilon_{(k-1)B+1},\dots,\varepsilon_{kB})$. Let $L=\lfloor\log(K)/\log(2)\rfloor$, $\tau_l = 2^l$ for $1 \le l \le L-1$ and $\tau_L=K$. Define $\bV_{k\tau} = E(\bV_k\mid \eta_{k-\tau},\dots,\eta_k), V_{kij\tau}=E(V_{kij}\mid\eta_{k-\tau},\dots,\eta_k)$ for $\tau \ge 0$, and $\bM_{Kl} = \sum_{k=1}^K(\bV_{k\tau_l}-\bV_{k\tau_{l-1}})$ for $2\le l\le L$. Then $\bQ_{n}-E(\bQ_n)$ can be decomposed as \begin{align}\label{pflemmaB9-1}
            \bQ_n-E(\bQ_n) & =\sum_{k=1}^K\pth{\bV_k-\bV_{kK}}+\sum_{l=2}^L\bM_{Kl}+\sum_{k=1}^K\{\bV_{k2}-E(\bV_{k2})\} = I_1 + I_2 + I_3.
        \end{align}
        We first deal with $I_1$. Notice that $I_1 = \sum_{k=1}^K\sum_{h=K+1}^\infty(\bV_{kh}-\bV_{k(h-1)})$. Then we have \begin{equation}\label{pflemmaB9-2}
            \begin{aligned}
                \left\|\left\|\sum_{k=1}^K(\bV_k-\bV_{kK})\right\|_{\cS,\max}\right\|_{q/2} &\le \sum_{h=K+1}^\infty \left\|\left\|\sum_{k=1}^K(\bV_{kh}-\bV_{k(h-1)})\right\|_{\cS,\max}\right\|_{q/2}\\
            & \le \sum_{h=K+1}^\infty C_q (1\vee \log p)^{3/2}K^{1/2}B\Omega_{0,q}\sum_{d=(h-2)B+1}^{(h+1)B}\omega_{d,q}\\
            & \le C_q(1\vee \log p)^{3/2}K^{1/2}B\Omega_{0,q}\Omega_{(K-1)B+1,q}\\
            & \le C_{q,\alpha}(1\vee \log p)^{3/2}K^{1/2}Bn^{-\alpha}\|\|\bX_1\|_{\cH,\infty}\|_{q,\alpha}^2,
            \end{aligned}
        \end{equation}
        where we combine results in Lemma \ref{leTechComposite2} and Lemma \ref{lemmaTechComposite4}, and utilize the fact $\Omega_{0,q} \le \|\|\bX_1\|_{\cH,\infty}\|_{q,\alpha}$, and for $K\ge 2$, $\Omega_{(K-1)B}\le \{(K-1)B\}^{-\alpha} \|\|\bX_1\|_{\cH,\infty}\|_{q,\alpha} \le C_\alpha n^{-\alpha}\|\|\bX_1\|_{\cH,\infty}\|_{q,\alpha}$. Using Markov's inequality and the fact that $K=\lceil n/B\rceil$, we have \begin{align*}
            P\left\{\left\|\sum_{k=1}^K(\bV_k-\bV_{kK})\right\|_{\cS,\max}\ge x\right\} \le C_{q,\alpha}x^{-q/2}(1 \vee \log p)^{3q/4}n^{q/4-\alpha q/2}B^{q/4}\|\|\bX_1\|_{\cH,\infty}\|_{q,\alpha}^q.
        \end{align*}

        Next we deal with the term $I_2$ in (\ref{pflemmaB9-1}). Define $\bY_{hl} = \sum_{k=(h-1)\tau_l+1}^{(h\tau_l)\wedge n} (\bV_{k\tau_l}-\bV_{k\tau_{l-1}}), Y_{hijl} = \sum_{k=(h-1)\tau_l+1}^{(h\tau_l)\wedge n} (V_{kij\tau_l}-V_{kij\tau_{l-1}})$ for $1\le h\le \lceil K/\tau_l\rceil$. Define $\eN^e$ to be the set of even positive integers, and $\bR_{nl}^e = \sum_{h\in \eN^e,1\le h\le \lceil K/\tau_l\rceil}\bY_{hl}, \bR_{nl}^o = \sum_{h\in \eN/\eN^e,1\le h\le \lceil K/\tau_l\rceil}\bY_{hl}$. Define a sequence of constant $\lambda_l = 3(l-1)^{-2}\pi^{-2}$ if $2\le l\le L/2$ and $\lambda_l = 3(L+1-l)^{-2}\pi^{-2}$ if $L/2 < l \le L$. Since $\sum_{k=1}^\infty k^{-2} = \pi^2/6$, we have $\sum_{l=2}^L\lambda_l \le 1$. Notice that $\bY_{h,l}$ and $\bY_{h',l}$ are independent if $|h-h'|\ge 2$. Thus we can use the Nagaev-type inequality in Lemma \ref{lemmaTechComposite3} to obtain that, for any $x > 0$\begin{equation}\label{pflemmaB9-3}
            \begin{aligned}
            P(\|\bR_{nl}^e\|_{\cS,\max}\ge \lambda_l x) = & P\{\|\mathrm{vec}(\bR_{nl}^e)\|_{\cS,\infty}\ge \lambda_l x\} \\
            \le & C_q(\lambda_l x)^{-q/2}(1 \vee \log p)^{q/2}\sum_{h\in\eN^e,1\le h\le \lceil K/\tau_l\rceil}E\{\|\mathrm{vec}(\bY_{hl})\|_{\cS,\infty}^{q/2}\} \\
            &+ 2p^2\exp\left\{-C\frac{\lambda_l^2x^2}{\max_{i,j\in[p]}\sum_{h\in\eN^e,1\le h\le \lceil K/\tau_l\rceil}E(\|Y_{hijl}\|_{\cS}^2)}\right\}.
        \end{aligned}
        \end{equation}
        Using similar argument in (\ref{pflemmaB9-2}), for any $h\in\lceil K/\tau_l\rceil$, we obtain\begin{equation}\label{pflemmaB9-4}
            \begin{aligned}
                \|\|\bY_{hl}\|_{\cS,\max}\|_{q/2} &\le C_q(1 \vee \log p)^{3/2}\tau_l^{1/2}B\Omega_{0,q}\Omega_{(\tau_{l-1}-1)B+1,q}\\
                &\le C_{q,\alpha}(1 \vee \log p)^{3/2}\tau_l^{1/2}B(\tau_l B)^{-\alpha}\|\|\bX_1\|_{\cH,\infty}\|_{q,\alpha}^2.
            \end{aligned}
        \end{equation}
        And similar to \eqref{pflemmaB9-4}, for any $h\in\lceil K/\tau_l\rceil,i,j\in[p]$, we have \begin{equation}\label{pflemmaB9-5}
            \|\|Y_{hijl}\|_{\cS}\|_{2}\le C_{\alpha}\tau_l^{1/2}B(\tau_l B)^{-\alpha}\|\|X_{1i}\|_\cH\|_{4,\alpha}\|\|X_{1j}\|_\cH\|_{4,\alpha} \le C_{\alpha}\tau_l^{1/2}B(\tau_l B)^{-\alpha}\bPhi_{4,\alpha}^\bX.
        \end{equation}
        The concentration in (\ref{pflemmaB9-3}) also holds for $\bR_{nl}^o$. Combining with (\ref{pflemmaB9-4}), (\ref{pflemmaB9-5}) and the fact that $\sum_{l=1}^L \lambda_l \le 1$, we have \begin{equation}\label{pflemmaB9-6}
            \begin{aligned}
                & P\left(\left\|\sum_{l=2}^L\bM_{Kl}\right\|_{\cS,\max} > 2x\right) \le \sum_{l=2}^L P\left(\|\bM_{Kl}\|_{\cS,\max} > 2\lambda_l x\right)\\
                \le & \sum_{l=2}^L P\left(\|\bR_{nl}^e\|_{\cS,\max} > \lambda_l x\right) + P\left(\|\bR_{nl}^o\|_{\cS,\max} > \lambda_l x\right)\\
                \le & C_{q,\alpha}x^{-q/2}(1 \vee \log p)^{5q/4}\|\|\bX_1\|_{\cH,\infty}\|_{q,\alpha}^qnB^{q/2-\alpha q/2-1}\sum_{l=2}^L\lambda_l^{-q/2}\tau_{l}^{q/4-q\alpha/2-1}\\
                & + 4p^2\sum_{l=2}^L\exp\left\{- \frac{x^2\lambda_l^2(\tau_l B)^{2\alpha}}{C_\alpha\left(\bPhi_{4,\alpha}^\bX\right)^2nB}\right\}\\
                := & I_4 + I_5.
            \end{aligned}
        \end{equation}
        
        Recall the fact $\lambda_l = 3(l-1)^{-2}\pi^{-2},\tau_{l} = 2^{l}$.
        By some elementary calculation, if $\alpha > 1/2-2/q$, then $q/4-1-\alpha q/2 < 0$, and $\sum_{l=2}^L\lambda_l^{-q/2}\tau_{l}^{q/4-q\alpha/2-1} \le C$. If $\alpha \le 1/2-2/q$, then $\sum_{l=2}^L\lambda_l^{-q/2}\tau_{l}^{q/4-q\alpha/2-1} \le CK^{q/4-q\alpha /2 -1}$. This implies $I_4 \le C_{q,\alpha}x^{-q/2}(1 \vee \log p)^{5q/4}\|\|\bX_1\|_{\cH,\infty}\|_{q,\alpha}^qF'_{n,B}$.
        
        For any $\alpha$, $\min_{L\in\eN}\min_{l\in[L]}\lambda_l^2\tau_l^{2\alpha} > 0$ and there exists an absolute constant integer $K_\alpha'$ such that for any $l\ge K_\alpha'$, $\lambda_{l+1}^2\tau_{l+1}^\alpha - \lambda_{l}^2\tau_{l}^\alpha \ge 1$. If $\exp\left\{- x^2\lambda_{K'_\alpha}^2(\tau_{K'_\alpha} B)^{2\alpha}/C_\alpha\left(\bPhi_{4,\alpha}^\bX\right)^2nB\right\} < 1/4p^2$, we have $$\sum_{l=K'_\alpha}^\infty\exp\left\{- x^2\lambda_{l}^2(\tau_l B)^{2\alpha}/C_\alpha\left(\bPhi_{4,\alpha}^\bX\right)^2nB\right\} \le C_\alpha'\exp\left\{- x^2\lambda_{K'_\alpha}^2(\tau_{K'_\alpha} B)^{2\alpha}/C_\alpha\left(\bPhi_{4,\alpha}^\bX\right)^2nB\right\}$$ where $C_\alpha'$ is an absolute constant. This implies $I_5 \le C_\alpha p^2 \exp\sth{- x^2/C_\alpha\pth{\bPhi_{4,\alpha}^\bX}^2nB}$. Notice that here we use the same constant $C_\alpha$ since we can enlarge the smaller one of $C_\alpha,C'_\alpha$ and the bound still holds. Thus \begin{equation*}
            \begin{aligned}
                P\left(\left\|\sum_{l=2}^L\bM_{Kl}\right\|_{\cS,\max} > 2x\right) \le & C_{q,\alpha}x^{-q/2}(1 \vee \log p)^{5q/4}\|\|\bX_1\|_{\cH,\infty}\|_{q,\alpha}^qF'_{n,B} \\
                & + C_\alpha p^2\exp\left\{-\frac{x^2}{C_\alpha \left(\bPhi_{4,\alpha}^\bX\right)^2nB}\right\}.
            \end{aligned}
        \end{equation*}
        If $\exp\left\{-C_\alpha x^2\lambda_{K'_\alpha}^2(\tau_{K'_\alpha} B)^{2\alpha}/\left(\bPhi_{4,\alpha}^\bX\right)^2nB\right\} \ge 1/4p^2$, the above bound trivially holds since the probability in the left hand side is always smaller than $1$.

        Now it remains to deal with $I_3$ in (\ref{pflemmaB9-1}). By the definition of $\bV_{k2}$, we have $\bV_{k2}$ and $\bV_{k'2}$ are independent if $|k-k'| \ge 3$. Using (c) of Lemma \ref{lemmaTechComposite1} and the similar argument in (\ref{pflemmaB9-6}), we obtain \begin{equation*}
            \begin{aligned}
                &P\left\{\left\|\sum_{k=1}^K \bV_{k2}-E(\bV_{k2})\right \|_{\cS,\max}\ge x\right \}\\
                & \le C_qx^{-q/2}(1 \vee \log p)^{q/2}\sum_{k=1}^K \|\|\bV_{k2}-E(\bV_{k,2})\|_{\cS,\max}\|_{q,\alpha}^{q/2} + 6p^2\exp\left\{-\frac{x^2}{C_\alpha\left(\bPhi_{4,\alpha}^\bX\right)^2nB }\right\}\\
                & \le C_qx^{-q/2}(1 \vee \log p)^{q}\|\|\bX_1\|_{\cH,\infty}\|_{q,\alpha}^qnB^{q/2-1}+ 6p^2\exp\left\{-\frac{x^2}{C_\alpha\left(\bPhi_{4,\alpha}^\bX\right)^2nB }\right\}.
            \end{aligned}
        \end{equation*}
        Combining all the bounds for $I_1,I_2$ and $I_3$ above we finish our proof.
        \end{proof}

        \subsection{Lemma \ref{nonGaussianConcen2} and its proof }

    \begin{lemma}\label{nonGaussianConcen2}
        Consider the quadratic form $\bQ_{n}^{(B)} = \sum_{B+1 \le t\le n}a_{t}\bX_{t-B}\otimes\bX_t^\T$. Assume $\bX_t$ comes from \eqref{Sec1-1} and it is mean zero. Assume $\sup_{t\in[n]}|a_{t}|\le 1$. Then we have \begin{align*}
            P(\|\bQ_n^{(B)}-E(\bQ_n^{(B)})\|_{\cS,\max}\ge x)&\le C_{q,\alpha}x^{-q/2}(1 \vee \log p)^q\|\|\bX_1\|_{\cH,\infty}\|_{q,\alpha}^qD'_{n,B} \\
            &+ C_\alpha p^2\exp\left\{-\frac{x^2}{C_\alpha\left(\bPhi_{4,\alpha}^\bX\right)^2n}\right\}.
        \end{align*}
        Here $D’_{n,B}=nB^{q/4-1}$ (resp., $nB^{q/4-1}+n^{q/4-\alpha q/2}$) if $\alpha > 1/2-2/q$ (resp., $\alpha \le 1/2-2/q$). And the same bound also holds for $\bQ_{n}^{(B)} = \sum_{B+1 \le t\le n}a_{t}\bX_{t}\otimes\bX_{t-B}^\T$.
    \end{lemma}

    \begin{proof}
        Lemma \ref{nonGaussianConcen2} can be proved similarly as Lemma \ref{nonGaussianConcen} with two necessary modifications. 
        
        (\romannumeral 1) Let $K=\lceil n/B\rceil \ge 2$. For $k\in[K]$, define $\bV^*_k = \sum_{t=k(B-1)}^{kB\wedge n} a_{t}\bX_{t-B}\otimes\bX_t^\T$, and $V_{kij} = \sum_{t=k(B-1)}^{kB\wedge n}a_{t}X_{(t-B)i}X_{tj}$ for $i,j\in[p]$ and $k\in[K]$. Define the innovation set $\eta_k=(\varepsilon_{(k-1)B+1},\dots,\varepsilon_{kB})$. Let $L=\lfloor\log(K)/\log(2)\rfloor$, $\tau_l = 2^l$ for $1 \le l \le L-1$ and $\tau_L=K$. Define $\bV^*_{k\tau} = E(\bV^*_k\mid \eta_{k-\tau},\dots,\eta_k), V^*_{kij\tau}=E(V^*_{kij}\mid\eta_{k-\tau},\dots,\eta_k)$ for $\tau \ge 0$. Following the same process as in (\ref{pflemmaB9-2}) (but using (\romannumeral 2) instead of (\romannumeral 1) of Lemma \ref{lemmaTechComposite4}) we have \begin{equation*}
            \begin{aligned}
                \left\|\left\|\sum_{k=1}^K \bV_k^*-\bV^*_{kK}\right\|_{\cS,\max}\right\|_{q/2} & \le \sum_{h=K+1}^\infty C_q (1\vee \log p)K^{1/2}B^{1/2}\Omega_{0,q}\sum_{d=(h-2)B+1}^{(h+1)B}\omega_{d,q}\\
                & \le C_q(1\vee \log p)K^{1/2-\alpha}B^{1/2-\alpha}\|\|\bX_1\|_{\cH,\infty}\|^2_{q,\alpha}\\
                & \le C_q(1\vee \log p)n^{1/2-\alpha}\|\|\bX_1\|_{\cH,\infty}\|^2_{q,\alpha}.
            \end{aligned}
        \end{equation*}
        
        (\romannumeral 2) Define $\bY^*_{hl} = \sum_{k=(h-1)\tau_l+1}^{(h\tau_l)\wedge n} (\bV^*_{k\tau_l}-\bV^*_{k\tau_{l-1}}), Y^*_{hijl} = \sum_{k=(h-1)\tau_l+1}^{(h\tau_l)\wedge n} (V^*_{kij\tau_l}-V^*_{kij\tau_{l-1}})$ for $1\le h\le \lceil K/\tau_l\rceil$. Then similar as in (\romannumeral 1), using (\romannumeral 2) of Lemma  \ref{lemmaTechComposite4} and (\romannumeral 2) of Lemma \ref{leTechComposite2}, \begin{align*}
            \|\|\bY^*_{hl}\|_{\cS,\max}\|_{q/2} & \le (1\vee \log p)\tau_l^{1/2}B^{1/2}(\tau_lB)^{-\alpha}\|\|\bX_1\|_{\cH,\infty}\|^2_{q,\alpha} \\
            \|\|Y^*_{hijl}\|_{\cS}\|_{2} & \le C_{\alpha}\tau_l^{1/2}B^{1/2}(\tau_lB)^{-\alpha}\left(\bPhi_{4,\alpha}^\bX\right)^2.
        \end{align*}

        The modification of (\romannumeral 1) and (\romannumeral 2) will result in an elimination of $(1 \vee \log p)^{q/4}B^{q/4}$ in polynomial bound, and an elimination of $B$ in exponential bound.
        \end{proof}

        \subsection{Lemma \ref{lm:elementwisecon1} and its proof}

        \begin{lemma}\label{lm:elementwisecon1}
            Assume $\bX_t$ comes from (\ref{Sec1-1}) and it is mean zero. Consider the $(j,k)$-th element of quadratic form $Q_{njk} = \sum_{1\le s\le t\le n}a_{st}X_{sj}\otimes X_{tk}$. Assume $\sup_{s,t\in[n]}|a_{st}|\le 1$ and $a_{st}=0$ if $|t-s| > B$ and $B<n$. Then we have \begin{align*}
            P\left\{\|Q_{njk}-E(Q_{njk})\|_{\cS,\max}\ge x\right\}\le & C_{q,\alpha}x^{-q/2}\|\|\bX_{1j}\|_{\cH,\infty}\|_{q,\alpha}^{q/2}\|\|\bX_{1k}\|_{\cH,\infty}\|_{q,\alpha}^{q/2}F_{n,B}' \\
            &+ C_\alpha \exp\left\{-\frac{x^2}{C_\alpha\left(\bPhi_{4,\alpha}^\bX\right)^2nB}\right\}.
        \end{align*}
        Here $F'_{n,B}=nB^{q/2-1}$ (resp., $nB^{q/2-1}+n^{q/4-\alpha q/2}B^{q/4}$) if $\alpha > 1/2-2/q$ (resp., $\alpha \le 1/2-2/q$). And the same bound also holds for $\bQ_n = \sum_{1\le t\le s\le n}a_{st}\bX_s\otimes\bX_t^\T$.
        \end{lemma}

        \begin{proof}
            Lemma \ref{lm:elementwisecon1} can be proven using the same procedure as in the proof of Lemma \ref{nonGaussianConcen}. There are two major differences. First, since we are considering only one-dimensional concentration, the terms \((1 \vee \log p)^{5q/4}\) and \(p^2\) are eliminated. Second, for the concentration of each dimension, we utilize the results from (\romannumeral 3) instead of (\romannumeral 1) of Lemma \ref{lemmaTechComposite4}. Consequently, the dependence measures are modified from \(\Omega_{0,q}, \omega_{d,q}\) to \(\Delta_{0,q,j}, \delta_{d,q,j}\). As a result, in the concentration inequality, \(\|\|\bX_1\|_{\cH,\infty}\|_{q,\alpha}^q\) is modified to \(\|\|\bX_{1j}\|_{\cH,\infty}\|_{q,\alpha}^{q/2}\|\|\bX_{1k}\|_{\cH,\infty}\|_{q,\alpha}^{q/2}\).
        \end{proof}

        \subsection{Lemma \ref{lm:elementwisecon2} and its proof}

        \begin{lemma}\label{lm:elementwisecon2}
            Assume $\bX_t$ comes from \eqref{Sec1-1} and it is mean zero. Consider the $(j,k)$-th element of the quadratic form $Q_{njk}^{(B)} = \sum_{B+1 \le t\le n}a_{t}X_{(t-B)j}\otimes X_{tk}$. Assume that $\sup_{t\in[n]}|a_{t}|\le 1$. Then we have \begin{align*}
            P(\|Q_{njk}^{(B)}-E(Q_{njk}^{(B)})\|_{\cS}\ge x) \le & C_{q,\alpha}x^{-q/2}\|\|X_{1j}\|_{\cH,\infty}\|_{q,\alpha}^{q/2}\|\|X_{1k}\|_{\cH,\infty}\|_{q,\alpha}^{q/2}D'_{n,B} \\
            &+ C_\alpha \exp\left\{-\frac{x^2}{C_\alpha\left(\bPhi_{4,\alpha}^\bX\right)^2n}\right\}.
        \end{align*}
        Here $D’_{n,B}=nB^{q/4-1}$ (resp., $nB^{q/4-1}+n^{q/4-\alpha q/2}$) if $\alpha > 1/2-2/q$ (resp., $\alpha \le 1/2-2/q$). And the same bound also holds for $Q_{njk}^{(B)} = \sum_{B+1 \le t\le n}a_{t} X_{tj}\otimes X_{(t-B)k}$.
        \end{lemma}

        \begin{proof}
            Lemma \ref{lm:elementwisecon2} can be proven following the same procedure as in the proof of Lemma \ref{nonGaussianConcen2}. The differences between these two results can be explained in the same manner as in the proof of Lemma \ref{lm:elementwisecon1}. For simplicity, we omit the details.
        \end{proof}

        \subsection{Lemma \ref{lemmaTechComposite5} and its proof}

        \begin{lemma}\label{lemmaTechComposite5}
            Assume $\|\nth{\bX_1}_{\cH,\infty}\|_{q,\alpha} < \infty$. Then we have \begin{align*}
                \max_{j,k\in[p]}\sum_{|h|\ge m_0}\|\Sigma_{jk}^{(h)}\|_\cS \le 2m_0^{-\alpha}(\bPhi_{2,0}^\bX\bPhi_{2,\alpha}^\bX)^{1/2}.
            \end{align*}
            Also we have $\max_{j,k\in[p]}\sum_{h\in\eZ}\|\Sigma_{jk}^{(h)}\|_\cS \le 2\bPhi_{2,0}^\bX$.
        \end{lemma}

        \begin{proof}
            Recall the projection operator $\cP_t$ that $\cP_tX = E(X\mid\cF_t)-E(X\mid\cF_{t-1})$. For any fixed $r$, since $X_{rj}$ is mean zero, we have $X_{rj} = \sum_{t=0}^\infty \cP_{r-t}X_{tj}$ for any $r,j$. This gives \begin{align*}
           \|\Sigma_{jk}^{(h)}\|_\cS = & \left\|E\{X_{(r-h)j}\otimes X_{rk}\}\right\|_{\cS}\\
            = & \left\|E\left[\left\{\sum\limits_{t = 0}^\infty \cP_{r-t}X_{(r-h)j}\right\}\otimes\left\{\sum\limits_{t = 0}^\infty \cP_{r-t}X_{rk}\right\}\right]\right\|_{\cS}
            \end{align*}
            Using Proposition 1.12 in \cite{pisier2016martingales} and the fact that for any $Y_1,Y_2$ we have $\|Y_1\otimes Y_2\|_\cS = \|Y_1\|_\cH\|Y_2\|_\cH$, for any $h > 0$ we obtain 
            \begin{align*}
             \|\Sigma_{jk}^{(h)}\|_\cS \le & \sum\limits_{t = 0}^\infty E\left\{\|\cP_{r-t}X_{(r-h)j}\|_\cH\left\|\cP_{r-t}X_{rk}\right\|_{\cH}\right\}\\
            \le & \sum\limits_{t = 0}^\infty\|\|\cP_{r-t}X_{(r-h)j}\|_\cH\|_2\|\|\cP_{r-t}X_{rk}\|_\cH\|_2,
        \end{align*}
        where the last inequality is Cauchy--Schwarz inequality. Using (a) of Lemma \ref{lemmaTechComposite1}, we have $\|\Sigma_{jk}^{(h)}\|_{\cS} \le \sum\limits_{t = 0}^\infty\delta_{t-h,2,j}\delta_{t,2,k}$.  Hence \begin{align*}
            \max\limits_{j,k\in[p]}\sum_{|h|\ge m_0} \|\Sigma_{jk}^{(h)}\|_{\cS} & \le \max\limits_{j,k\in[p]}\sth{\left(\sum\limits_{t\ge 0}\delta_{t,2,j}\right)\left(\sum\limits_{t\ge m_0}\delta_{t,2,k}\right) + \left(\sum\limits_{t\ge 0}\delta_{t,2,k}\right)\left(\sum\limits_{t\ge m_0}\delta_{t,2,j}\right)}\\
            & \le 2m_0^{-\alpha}(\bPhi_{2,0}^{\bX}\bPhi_{2,\alpha}^{\bX})^{1/2}.
        \end{align*}
        And also we have $\max_{j,k\in[p]}\sum_{h\in\eZ} \|\Sigma_{jk}^{(h)}\|_\cS \le 2\bPhi_{2,0}^\bX$.
        \end{proof}
        
        \subsection{Proof of Lemma \ref{lem1}}

        First we have \begin{align*}
            R(m_0) \le \max\limits_{j,k\in[p]}\left\{\sum_{|h|\ge m_0}\|\Sigma_{jk}^{(h)}\|_{\cS}\right\} + \max\limits_{j,k\in[p]}\left[\sum_{|h|<m_0}\left\{1-K(h/m_0)\right\} \|\Sigma_{jk}^{(h)}\|_{\cS}\right].
        \end{align*}
        For the first part, Lemma \ref{lemmaTechComposite5} implies \begin{align*}
            \max\limits_{j,k\in[p]}\sum_{|h|\ge m_0} \|\Sigma_{jk}^{(h)}\|_\cS & \le \max\limits_{j,k\in[p]}2\left(\sum\limits_{t\ge 0}\delta_{t,2,j}\right)\left(\sum\limits_{t\ge m_0}\delta_{t,2,k}\right)\le 2m_0^{-\alpha}(\bPhi_{2,0}^{\bX}\bPhi_{2,\alpha}^{\bX})^{1/2}.
        \end{align*}
        For the second part, Lemma \ref{lemmaTechComposite5} implies \begin{align*}
            & \max\limits_{j,k\in[p]}\sum_{|h|<m_0}\{1-K(h/m_0)\}\|\Sigma_{jk}^{(h)}\|_{\cS} \\
            \le & \max_{j,k\in[p]}\left[\sum_{|h|<m_0^\beta}|1-K(h/m_0)|\|\Sigma_{jk}^{(h)}\|_\cS + \sum_{m_0^\beta\le|h|<m_0}\|\Sigma_{jk}^{(h)}\|_{\cS} \right]\\
            \le &  Cm_0^{(\beta-1)\tau}\bPhi_{2,0}^\bX + 2m_0^{-\alpha\beta}(\bPhi_{2,0}^\bX\bPhi_{2,\alpha}^\bX)^{1/2}.
        \end{align*}
        Taking $\beta = \tau/(\tau+\alpha)$ and noticing $\bPhi^\bX_{2,0}\le \bPhi_{2,\alpha}^\bX$, we have \begin{align*}
            R(m_0) \le C(\bPhi_{2,0}^\bX\bPhi_{2,\alpha}^\bX)^{1/2}\left\{m_0^{-\alpha} + m_0^{-\alpha \tau/(\tau+\alpha )}\right\}\le C\bPhi_{2,\alpha}^\bX m_0^{-\alpha \tau/(\tau+\alpha)}.
        \end{align*}

    \subsection{Lemma \ref{LemmaPartial1} and its proof}

    \begin{lemma}\label{LemmaPartial1}
            Suppose Condition \ref{condpartial1} holds. Define $\cT_{tj} = T_{tj}b_{j},\cT_t = \min_{j\in[p]}\cT_{tj}$. For any given curve \(\bX_t\), assume for all $t\in[n]$ we have $\cT_t \ge C$, here $C$ is an absolute constant. Recall that \(\be_0 = (1,0)^\T,\widetilde{\bU}_{tji} = \{1,(U_{tji}-u)/b_{j}\}^\T\) and $$\widehat{\bS}_{tj}(u) = \frac{1}{T_{tj}}\sum_{i = 1}^{T_{tj}}\widetilde{\bU}_{tji}\widetilde{\bU}_{tji}^\T K_{b_{j}}(U_{tji}-u),\quad \widehat{\bR}_{tj}(u) = \frac{1}{T_{tj}}\sum_{i = 1}^{T_{tj}}\widetilde{\bU}_{tji}Y_{tji} K_{b_{j}}(U_{tji}-u).$$
            Let \(\widetilde{X}_{tj}(u) = \be_0^\T\left[E\{\widehat{\bS}_{tj}(u)\}\right]^{-1}E\{\widehat{\bR}_{tj}(u)\mid \bX_t\}\). Define the event $$\Omega_{tj1}(\delta')= \left\{\sup_{u\in[0,1]}\|\widehat{\bS}_{tj}(u)-E\{\widehat{\bS}_{tj}(u)\}\|_F\le C_S\delta'/2\right\},\delta'\in(0,1]$$
            and $\Omega_{t1}(\delta') = \bigcap_{j\in[p]} \Omega_{tj1}(\delta')$. The detail of constant $C_S$ is in the proof of Theorem \ref{thm7}.  Recall in Condition \ref{condpartial1}(\romannumeral 6) we define that $\bX_t^* = (X_{t1}^*,\dots,X_{tp}^*)^\T, X_{tj}^* = \sup_{u\in[0,1]}|X_{tj}(u)|$, and $\bX_{t}^{(2)*} = (X_{t1}^{(2)*},\dots,X_{tp}^{(2)*})^\T,X_{tj}^{(2)*} = \sup_{u\in[0,1]}|\partial^2_u X_{tj}(u)|$. Then for any $\delta > 0$, there exists absolute constants $C_1,C_2$ such that for any dimension $p > 0$ \begin{equation}\label{lemmaB8eq1}
                \begin{aligned}
                P\left(\frac{\|\widehat{\bX}_{t} - \widetilde{\bX}_{t}\|_{\cH,\infty}}{1+|\bX^*_{t}|_\infty} \ge \delta, \Omega_{t1}(1) \mid \bX_t\right)  \le C_1p\exp\{-C_2\cT_{t}\min(\delta,\delta^2)\} \quad \mbox{a.s.,}
                \end{aligned}
            \end{equation} 
            and $\Omega_{t1}(1)$ satisfies there exists constants $C_3,C_4$ such that \begin{equation}\label{lemmaB8eq3}
                1-P\{\Omega_{t1}(1)\} \le C_4p\exp(-C_5\cT_{t}).
            \end{equation}
            Additionally we have there exists absolute constant $C_5$ such that \begin{equation}\label{lemmaB8eq2}
                \max_{j\in[p]}\|\widetilde{X}_{tj}-X_{tj}\|_\cH \le C_3\max_{j\in[p]}b_{j}^2X_{tj}^{(2)*}\quad \mbox{a.s.}
            \end{equation}
    \end{lemma}

    \begin{proof}
        We organize our proof in four steps. 

        \subsubsection{Definition and Decomposition}\label{SecB.9.1}

       For any square matrix \(\bB\), write \(\|\bB\|_{\min} = \{\lambda_{\min}(\bB^\T\bB)\}^{1/2}, \|\bB\|_F = (\sum_{j,k}B_{jk}^2)^{1/2}\).  Recall that \(\widehat{X}_{tj}(u) = \be_0^\T\{\widehat{\bS}_{tj}(u)\}^{-1}\widehat{\bR}_{tj}(u)\).  From Lemma \ref{LemmaPartial3}, $E\{\widehat{\bS}_{tj}(u)\}$ is positive definite. If $\widehat{\bS}_{tj}(u)$ is positive definite, we can decompose $\widehat{X}_{tj}(u)-\widetilde{X}_{tj}(u)$ as \begin{align*}
            \widehat{X}_{tj}(u)-\widetilde{X}_{tj}(u) = & \be_0^\T [E\{\widehat{\bS}_{tj}(u)\}]^{-1}[\widehat{\bR}_{tj}(u) - E\{\widehat{\bR}_{tj}(u)\mid\bX_t\}]\\
            & -\be_0^\T \{\widehat{\bS}_{tj}(u)\}^{-1} [\widehat{\bS}_{tj}(u)-E\{\widehat{\bS}_{tj}(u)\}][E\{\widehat{\bS}_{tj}(u)\}]^{-1}\widehat{\bR}_{tj}(u),
        \end{align*}
        which implies that \begin{equation}\label{ProofThm5-1}
            \begin{aligned}
                |\widehat{X}_{tj}(u)-\widetilde{X}_{tj}(u)| \le & \|E\{\widehat{\bS}_{tj}(u)\}\|_{\min}^{-1}|\widehat{\bR}_{tj}(u) - E\{\widehat{\bR}_{tj}(u)\mid\bX_t\}|_2\\
            & + \|\widehat{\bS}_{tj}(u)\|_{\min}^{-1}\|E\{\widehat{\bS}_{tj}(u)\}\|_{\min}^{-1}|\widehat{\bR}_{tj}(u)|_2\|\widehat{\bS}_{tj}(u)-E\{\widehat{\bS}_{tj}(u)\}\|_F.
            \end{aligned}  
            \end{equation}
        \subsubsection{Proof of equation (\ref{lemmaB8eq1})}

        Similar to equation (A.3) in \cite{guo2025sparse}, we have there exists some positive absolute constant \(C\) such that for any \(\delta > 0\) and \(u\in[0,1]\), \begin{equation*}
            P\left[\|\widehat{\bS}_{tj}(u)-E\{\widehat{\bS}_{tj}(u)\}\|_F\ge \delta\right] \le 8\exp\left(-\frac{CT_{tj}b_{j}\delta^2}{1+\delta}\right).
        \end{equation*}
        Also notice that $\widehat{\bS}_{tj}(u)$ is independent of $\bX_t$, so adding conditioning on $\bX_t$ does not change the bound. So we have \begin{equation}\label{ProofThm5-2}
            P\left[\|\widehat{\bS}_{tj}(u)-E\{\widehat{\bS}_{tj}(u)\}\|_F\ge \delta\mid \bX_t\right] \le 8\exp\left(-\frac{CT_{tj}b_j\delta^2}{1+\delta}\right) \quad\mbox{a.s.}
        \end{equation} 
        Lemma \ref{LemmaPartial4} implies that there exist some positive absolute constants \(C_1, C_2\) such that for any \(\delta > 0\) and \(u\in[0,1]\) \begin{equation}\label{ProofThm5-3}
            P\left[|\widehat{\bR}_{tj}(u) - E\{\widehat{\bR}_{tj}(u)\mid\bX_t\}|_2 \ge\delta\mid \bX_t \right] \le C_1\exp\left\{-\frac{C_2T_{tj}b_{j}\delta^2/(1+X_{tj}^*)^2}{1+\delta/(1+X_{tj}^*)}\right\}\quad \mbox{a.s.}
        \end{equation}

        Denote \(E\{\widehat{\bS}_{tj}(u)\}\) as \(\bM_{b_{j}}(u)\) since it only depends on bandwidth $b_j$ and the kernel function $K$. Lemma \ref{LemmaPartial3} implies \begin{equation}\label{PfB8eq1}
            \inf_{b_j\in(0,1/2]}\inf_{u\in[0,1]}\|\bM_{b_{j}}(u)\|_{\min}\ge C_S,
        \end{equation} where $C_S$ is an absolute positive constant that can be explicitly determined by density $f_U$ and kernel function $K$. Recall that \(\Omega_{tj1}(\delta')= \left\{\sup_{u\in[0,1]}\|\widehat{\bS}_{tj}(u)-E\{\widehat{\bS}_{tj}(u)\}\|_F\le C_S\delta'/2\right\},\delta'\in(0,1]\). On the event \(\Omega_{tj1}(1)\), we have for all \(u\in[0,1]\), by \eqref{PfB8eq1}, \(\|\widehat{\bS}_{tj}(u)\|_{\min} \ge C_S/2\) and $\widehat{\bS}_{tj}(u)$ is inversible. So using decomposition in (\ref{ProofThm5-1}), we have on the event $\Omega_{tj1}(1)$, \begin{align*}
            |\widehat{X}_{tj}(u)-\widetilde{X}_{tj}(u)| \le & C_S^{-1}|\widehat{\bR}_{tj}(u)-E\{\widehat{\bR}_{tj}(u)\mid \bX_t\}|_2 \\
            & + 2C_S^{-2}|\widehat{\bR}_{tj}(u)|_2\|\widehat{\bS}_{tj}(u)-E\{\widehat{\bS}_{tj}(u)\}\|_F \quad \mbox{a.s.}
        \end{align*}
        Now we calculate $|E\{\widehat{\bR}_{tj}(u)\mid\bX_t\}|_2$. We first compute $E\{\widehat{R}_{tj1}(u)\mid\bX_t\}$ while $E\{\widehat{R}_{tj2}(u)\mid\bX_t\}$ can be similarly evaluated. Noticing that $E(\varepsilon_{tji})=0$ and $\varepsilon_{tji},i\in [T_{tj}]$ are independent of $X_{tj}$, we have \begin{align*}
            |E\{\widehat{R}_{tj1}(u)\mid\bX_t\}| &= |E\{Y_{tj1}K_{b_{j}}(U_{tj1}-u)\mid\bX_t\}| = |E\{X_{tj}(U_{tj1})K_{b_{j}}(U_{tj1}-u)\mid\bX_t\}|\\
            & = \frac{1}{b_{j}}\left|\int_0^1X_{tj}(v)K\left(\frac{v-u}{b_{j}}\right)f_U(v)\dv\right| \le M_fc_KX_{tj}^*\quad\mbox{a.s.}
        \end{align*}
        Similarly $|E\{\widehat{R}_{tj2}(u)\mid\bX_t\}| \le M_fc_KX_{tj}^*$ a.s. Since $|\widehat{\bR}_{tj}(u)|_2 \le |E\{\widehat{\bR}_{tj}(u)\mid \bX_t\}|_2 + |\widehat{\bR}_{tj}(u)-E\{\widehat{\bR}_{tj}(u)\mid \bX_t\}|_2$, we have on the event $\Omega_{tj1}(1)$, $$|\widehat{X}_{tj}(u)-\widetilde{X}_{tj}(u)| \le C|\widehat{\bR}_{tj}(u)-E\{\widehat{\bR}_{tj}(u)\mid\bX_t\}|_2 + CX_{tj}^*\|\widehat{\bS}_{tj}(u)-E\{\widehat{\bS}_{tj}(u)\}\|_F.$$
        Here $C$ is an absolute constant since it only depends on $C_S$, which is also an absolute constant. In the following, let $C_1,C_2$ be absolute constants that may vary from line to line. Combining the probability bound (\ref{ProofThm5-2}) and (\ref{ProofThm5-3}), we have $$P\left\{|\widehat{X}_{tj}(u)-\widetilde{X}_{tj}(u)|\ge \delta, \Omega_{tj1}(1) \mid \bX_t\right\} \le C_1\exp\left\{-\frac{C_2T_{tj}b_{j}\delta^2/(1+X^*_{tj})^2}{1+\delta/(1+X_{tj}^*)}\right\}\quad \mbox{a.s.}$$ 
        Now applying the first part of Lemma 6 in \cite{guo2023consistency}, we have for any \(u\in[0,1]\), and integer \(q \ge 1\), $$E\left\{\frac{|\widehat{X}_{tj}(u)-\widetilde{X}_{tj}(u)|^{2q}}{(1+X^*_{tj})^{2q}} I\{\Omega_{tj1}(1)\}\mid \bX_t\right\} \le q!C_1(4C_2T_{tj}b_{j})^{-q} + (2q)!C_1(4C_2T_{tj}b_{j})^{-2q},$$
        where $I$ is the indicator function. Using Fubini's Theorem, we have \begin{equation*}
            \begin{aligned}
                &E\left\{\frac{\|\widehat{X}_{tj}(u)-\widetilde{X}_{tj}(u)\|_\cH^{2q}}{(1+X^*_{tj})^{2q}}I\{\Omega_{tj1}(1)\}\mid \bX_t\right\}\\
                & = 
                E\left\{\frac{\left[\int_0^1\{\widehat{X}_{tj}(u)-\widetilde{X}_{tj}(u)\}^2\du\right]^{q}}{(1+X^*_{tj})^{2q}}I\{\Omega_{tj1}(1)\}\mid \bX_t\right\}\\
                & \le E\left\{\frac{\int_0^1\{\widehat{X}_{tj}(u)-\widetilde{X}_{tj}(u)\}^{2q}\du}{(1+X^*_{tj})^{2q}}I\{\Omega_{tj1}(1)\}\mid \bX_t\right\}\\
                & = \int_0^1 E\left\{\frac{|\widehat{X}_{tj}(u)-\widetilde{X}_{tj}(u)|^{2q}}{(1+X^*_{tj})^{2q}}I\{\Omega_{tj1}(1)\}\mid \bX_t\right\}\du\\
                &\le q!C_1(4C_2T_{tj}b_{j})^{-q} + (2q)!C_1(4C_2T_{tj}b_{j})^{-2q}.
            \end{aligned}
        \end{equation*}
        Applying the second part of Lemma 6 in \cite{guo2025sparse} again, noticing $T_{tj}b_{j}$ goes to infinity, we have \begin{equation*}
            P\left\{\|\widehat{X}_{tj} - \widetilde{X}_{tj}\|_\cH \ge (1+X_{tj}^*)\delta, \Omega_{tj1}(1) \mid \bX_t\right\} \le C_1\exp\{-C_2T_{tj}b_{j}\min(\delta,\delta^2)\}.
        \end{equation*}
        Using union bound argument, we conclude that $$P\left\{\|\widehat{\bX}_{t} - \widetilde{\bX}_{t}\|_{\cH,\infty} \ge (1+|\bX_{t}^*|_\infty) \delta, \Omega_{t1}(1)\mid \bX_t\right\}  \le C_1p\exp\{-C_2\cT_{t}\min(\delta,\delta^2)\},$$
        and we finish the proof of \eqref{lemmaB8eq1}. 

        \subsubsection{Proof of equation (\ref{lemmaB8eq3})}
        Denote \(W = \sup_{u\in[0,1]}\|\widehat{\bS}_{tj}(u)-E\{\widehat{\bS}_{tj}(u)\}\|_F\). Similar to equation (A.15) in \cite{guo2025sparse}, using Theorem 12.5 of \cite{Boucheron2013}, there exists constant $C_1,C_2$ such that for any \(\delta > 0\), \begin{equation}\label{PfB8eq2}
            E(W) \le C_1(T_{tj}b_{j})^{-1/2},
        \end{equation}\begin{equation}\label{ProofThm5-6}
            P\{W-E(W) > \delta\} \le 4\exp\left(-\frac{C_2T_{tj}b_{j}\delta^2}{1+\delta}\right).
        \end{equation}
        Recall that \(\Omega_{tj1}(\delta')= \left\{\sup_{u\in[0,1]}\|\widehat{\bS}_{tj}(u)-E\{\widehat{\bS}_{tj}(u)\}\|_F\le C_S\delta'/2\right\},\delta'\in(0,1]\). Then according to \eqref{PfB8eq2}, for any $\delta' \ge 3C_S^{-1}C_1(T_{tj}b_{j})^{-1/2}$, we have $E(W) \le C_S\delta'/3$. Thus from \eqref{ProofThm5-6}, if $\delta' \ge 3C_S^{-1}C_1(T_{tj}b_{j})^{-1/2}$ there exists absolute constant $C$ such that \begin{equation*}
            1-P\{\Omega_{tj1}(\delta')\} \le 4\exp\left\{-\frac{CT_{tj}b_{j}(\delta')^2}{1+\delta'}\right\}.
        \end{equation*}
        For $\delta' < 3C_S^{-1}C_1(T_{tj}b_{j})^{-1/2}$, we can choose large enough constants $C,C'$ such that for all $\delta'\in(0,1]$, \begin{equation*}
            1-P\{\Omega_{tj1}(\delta')\} \le C\exp\left\{-\frac{C'T_{tj}b_{j}(\delta')^2}{1+\delta'}\right\}.
        \end{equation*}
        Since $\Omega_{t1}(1)=\bigcap_{j\in[p]}\Omega_{tj1}(1)$. Using union bound argument we have there exists absolute constant $C_1,C_2$ such that \begin{align*}
            1-P\{\Omega_{t1}(1)\} \le C_1p\exp(-C_2\cT_t),
        \end{align*}
        and this finishes the proof of \eqref{lemmaB8eq3}.

        \subsubsection{Proof of Equation \ref{lemmaB8eq2}}

        The remaining task is to bound \(\max_{j\in[p]}\|\widetilde{X}_{tj} - X_{tj}\|_\cH\). Using argument in the proof of equation (A.35) in \cite{guo2025sparse}, we have $\max_{j\in[p]}\|\widetilde{X}_{tj}-X_{tj}\|_\cH\le C\max_{j\in[p]}b_{j}^2X^{(2)*}_{tj}$. This finishes the proof of Lemma \ref{LemmaPartial1}.
    \end{proof}

    \subsection{Lemma \ref{LemmaPartial4} and its proof}

    \begin{lemma}\label{LemmaPartial4}
        Recall the definition $\widehat{\bR}_{tj}(u) = \frac{1}{T_{tj}}\sum_{i = 1}^{T_{tj}}\widetilde{\bU}_{tji}Y_{tji} K_{b_{j}}(U_{tji}-u)$ as in Section \ref{SecB.9.1}. Recall that $\bX_t^* = (X_{t1}^*,\dots,X_{tp}^*)^\T, X_{tj}^* = \sup_{u\in[0,1]}|X_{tj}(u)|$. Assume all conditions in Lemma \ref{LemmaPartial1} hold. Then there exists some positive absolute constant \(C\) such that for any \(\delta > 0\) and \(u\in[0,1]\) \begin{equation}\label{PfLmB9-1}
            P\left[|\widehat{\bR}_{tj}(u) - E\{\widehat{\bR}_{tj}(u)\mid\bX_t\}|_2 \ge\delta\mid \bX_t \right] \le C\exp\left\{-\frac{CT_{tj}b_{j}\delta^2/(1+X_{tj}^*)^2}{1+\delta/(1+X_{tj}^*)}\right\}\quad \mbox{a.s.}
        \end{equation}
    \end{lemma}

    \begin{proof}
        In this proof, let $C$ be an absolute constant which might vary from line to line. Let \(\widehat{\bR}_{tj}(u) = \{\widehat{R}_{tj1}(u),\widehat{R}_{tj2}(u)\}^\T\). We focus on $\widehat{R}_{tj1}$, while \(\widehat{R}_{tj2}\) can be demonstrated in a similar manner. Define \(\widehat{R}_{tj3}(u) = T_{tj}^{-1}\sum_{i = 1}^{T_{tj}}X_{tj}(U_{tji})K_{b_{j}}(U_{tji}-u), \widehat{R}_{tj4}(u) = T_{tj}^{-1}\sum_{i = 1}^{T_{tj}}\varepsilon_{tji}K_{b_{j}}(U_{tji}-u)\). Then \(\widehat{R}_{tj1}(u)-E\{\widehat{R}_{tj1}(u)\}\) can be rewritten as $$\widehat{R}_{tj1}(u)-E\{\widehat{R}_{tj1}(u)\} = \widehat{R}_{tj3}(u)-E\{\widehat{R}_{tj3}(u)\} + \widehat{R}_{tj4}(u).$$
        We first deal with $\widehat{R}_{tj4}$. Since we have assumed $\varepsilon_{tji}$ are sub-Gaussian random variables with $E\{\exp(\varepsilon_{tji}z)\}\le \exp(C^2\sigma_j^2z^2/2)$ and the variance $\sigma_j^2$ are uniformly bounded by $\sigma^2$, by Proposition 2.5.2 of \cite{vershynin2018high}, for any integer $q$, we have $E(|\varepsilon|^q) \le C^q\sigma^qq^{q/2}$, where $C$ is an absolute constant. Also notice $q^{q/2} \le Cq!/2$. Define $c_K = \sup_{u\in[-1,1]}K(u)$. We have \begin{equation*}
            \begin{aligned}
                \sum_{l=1}^{T_{tj}}E\sth{\varepsilon_{tji}^2K_{b_{j}}(U_{tji}-u)^2} & = T_{tj}E\pth{\varepsilon_{tji}^2}\int_0^1{b_{j}}^{-2}K\left(\frac{v-u}{b_{j}}\right)^2f_U(v)\dv\\
                & \le CM_fT_{tj}b_{j}^{-1}c_K^2\sigma^2,\\
                \sum_{l=1}^{T_{tj}}E\sth{|\varepsilon_{tji}|^qK_{b_{j}}(U_{tji}-u)^q} & = T_{tj}E\pth{\varepsilon_{tji}^q}\int_0^1{b_{j}}^{-q}K\left(\frac{v-u}{b_{j}}\right)^qf_U(v)\dv\\
                & \le CM_fT_{tj}b_{j}^{-q+1}c_K^q\sigma^q q!/2.
            \end{aligned}
        \end{equation*}
        By Bernstein inequality (Theorem 2.10 and Corollary 2.11 of \citealt{Boucheron2013}) we have \begin{equation*}
            \begin{aligned}
                P\qth{|\widehat{R}_{tj4}(u) - E\{\widehat{R}_{tj4}(u)\}|\ge\delta} \le 2\exp\pth{-\frac{CT_{tj}b_{j}\delta^2}{1+\delta}}.
            \end{aligned}
        \end{equation*}
        Here $C$ is an absolute constant. Since $\hat{R}_{tj3}(u)$ is independent of $\bX_t$, we have the conditional inequality takes the same form, which is \begin{equation}\label{ProofThm5-4}
            \begin{aligned}
                P\qth{|\widehat{R}_{tj4}(u) - E\{\widehat{R}_{tj4}(u)\}|\ge\delta\mid \bX_t} \le 2\exp\pth{-\frac{CT_{tj}b_{j}\delta^2}{1+\delta}}\quad\mbox{a.s.}
            \end{aligned}
        \end{equation} 
        As for the concentration bound for \(\widehat{R}_{tj3}(u)\), we have\begin{equation*}
            \begin{aligned}
                \sum_{i = 1}^{T_{tj}}E\{X_{tj}(U_{tji})^2K_{b_{j}}(U_{tji}-u)^2\mid \bX_t\} &= T_{tj}\int_0^1 X_{tj}^2(v)b_{j}^{-2}K\left(\frac{v-u}{b_{j}}\right)^2f_U(v)\dv \\
                &\le M_fc_K^2T_{tj}b_{j}^{-1}(X_{tj}^*)^2.
            \end{aligned}
        \end{equation*}
        Similarly we have \(\sum_{i = 1}^{T_{tj}}E\{X_{tj}(U_{tji})^qK_{b_{j}}(U_{tji}-u)^q\mid \bX_t\}\le T_{tj}M_fc_K^q b_j^{-q+1}(X_{tj}^*)^q\). By Bernstein inequality (Theorem 2.10 and Corollary 2.11 of \citealt{Boucheron2013}) we have \begin{equation}\label{ProofThm5-5}
            P\left[|\widehat{R}_{tj3}(u) - E\{\widehat{R}_{tj3}(u)\}|\ge\delta\mid\bX_t\right] \le 2\exp\left\{-\frac{CT_{tj}b_{j}\delta^2/(X_{tj}^*)^2}{1+\delta/X_{tj}^*}\right\}.
        \end{equation}
        And then (\ref{PfLmB9-1}) follows from (\ref{ProofThm5-4}) and (\ref{ProofThm5-5}).
    \end{proof}

    \subsection{Lemma \ref{LemmaPartial3} and its proof}

    \begin{lemma}\label{LemmaPartial3}
        Recall in Section \ref{SecB.9.1} we defined $\widehat{\bS}_{tj}(u) = T_{tj}^{-1}\sum_{i = 1}^{T_{tj}}\widetilde{\bU}_{tji}\widetilde{\bU}_{tji}^\T K_{b_{j}}(U_{tji}-u)$, here $\widetilde{\bU}_{tji} = \{1,(U_{tji}-u)/b_{j}\}^\T$. Under all conditions of Lemma \ref{LemmaPartial1}, we have \begin{align*}
            \inf_{b_{j}\in(0,1/2]}\inf_{u\in[0,1]}\|E\{\widehat{\bS}_{tj}(u)\}\|_{\min}\ge m_fC_K,
        \end{align*}
        where $C_K$ is a positive constant that only depends on kernel function $K$, and $m_f$ 
        is the infimum of density of $U$ defined as in Condition \ref{condpartial1}.
    \end{lemma}

    \begin{proof}
        In this section of proof we abbreviate $b_{j}$ as $b$. Denote \(E\{\widehat{\bS}_{tj}(u)\}\) as \(\bM_b(u)\) since it only depends on bandwidth $b$ and the kernel function $K$. 
        
        (\romannumeral 1) We first restrict ourselves with the case where the density function of $U$ satisfies $f_U(u) = 1$. Then for $u\in (b,1-b)$, we have \begin{align*}
            \bM_b(u) = \begin{pmatrix}
                \int_{-1}^1 K(v)\mathrm{d}v & \int_{-1}^1 vK(v)\mathrm{d}v\\
                \int_{-1}^1 vK(v)\mathrm{d}v & \int_{-1}^1 v^2K(v)\mathrm{d}v
            \end{pmatrix}.
        \end{align*}
        For other scenarios, we only consider $u\in[0,b]$, since $u\in[1-b,1]$ can be tackled in the similar manner due to symmetry of kernel. Assume that $u = \Tilde{u}h$ with $ \Tilde{u}\in[0,1]$. Then we have \begin{align*}
            \bM_b(u) = \widetilde{\bM}(\Tilde{u}) =\begin{pmatrix}
                \int_{-\Tilde{u}}^1 K(v)\mathrm{d}v & \int_{-\Tilde{u}}^1 vK(v)\mathrm{d}v\\
                \int_{-\Tilde{u}}^1 vK(v)\mathrm{d}v & \int_{-\Tilde{u}}^1 v^2K(v)\mathrm{d}v
            \end{pmatrix}.
        \end{align*}
        Elementary calculation gives $$\|\widetilde{\bM}(\Tilde{u})\|_{\min} = \{a(\Tilde{u})+c(\Tilde{u})-\sqrt{a(\Tilde{u})^2-2a(\Tilde{u})c(\Tilde{u})+4b(\Tilde{u})^2+c(\Tilde{u})^2}\}/2$$ with $a(\Tilde{u}) = \int_{-\Tilde{u}}^1 K(v)\dv, b(\Tilde{u}) = \int_{-\Tilde{u}}^1 vK(v)\mathrm{d}v, c(\Tilde{u}) = \int_{-\Tilde{u}}^1 v^2K(v)\mathrm{d}v$, and $\|\widetilde{\bM}(\Tilde{u})\|_{\min}$ is a continuous function of $\Tilde{u}$.

        For any $\Tilde{u}\in[0,1]$, define $f(\gamma;\Tilde{u})=\int_{-\tilde{u}}^1 (u-\gamma)^2K(u)\mathrm{d}u$ as a quadratic function of $\gamma$. It satisfies $f(\gamma;\Tilde{u}) > 0$ for any real number $\gamma$ since we have assumed $K(u)$ is a symmetric Lipschitz continuous probability density function with support $[-1,1]$. So the discriminant of $f(\gamma;\Tilde{u})$ must be strictly greater than $0$, which results in $a(\Tilde{u})c(\Tilde{u}) > b(\Tilde{u})^2$ and further $\|\widetilde{\bM}(\Tilde{u})\|_{\min} > 0$ for any $\Tilde{u}\in[0,1]$. Since $\|\widetilde{\bM}(\Tilde{u})\|_{\min}$ is a continuous function on a compact interval, it must achieve minimum at some point in $[0,1]$. This shows $\|\widetilde{\bM}(\Tilde{u})\|_{\min} \ge C_K$, where $C_K$ is a constant that only depends on kernel function $K$. It implies $\inf_{u\in[0,1-b]}\|\bM_b(u)\|_{\min} \ge C_K$ for any $b\in(0,1/2]$, which further implies $\inf_{b\in(0,1/2]}\inf_{u\in[0,1-b]}\|\bM_b(u)\|_{\min} \ge C_K$. The symmetry of kernel function $K$ gives $\inf_{b\in(0,1/2]}\inf_{u\in[0,1]}\|\bM_b(u)\|_{\min} \ge C_K$.

        (\romannumeral 2) We then consider the case where $f_U(u)$ does not equal $1$. Define $f_{U_1}(u)=1$, $f_{U_2}(u)= \{f_U(u)-m_f\}/(1-m_f)$. They are both density functions. Matrix $\bM_b(u)$ can be decomposed as $\bM_b(u)= m_f\bM_b^{U_1}(u)+ (1-m_f)\bM_b^{U_2}(u)$, where \begin{align*}
            \bM_b^{U_1}(u) &= \begin{pmatrix}
                \int_{-1}^1 K(v/h-u)\mathrm{d}v & \int_{-1}^1 vK(v/h-u)\mathrm{d}v\\
                \int_{-1}^1 vK(v/h-u)\mathrm{d}v & \int_{-1}^1 v^2K(v/h-u)\mathrm{d}v
            \end{pmatrix},\\
            \bM_b^{U_2}(u) &= \begin{pmatrix}
                \int_{-1}^1 f_{U_2}(v)K(v/h-u)\mathrm{d}v & \int_{-1}^1 vf_{U_2}(v)K(v/h-u)\mathrm{d}v\\
                \int_{-1}^1 vf_{U_2}(v)K(v/h-u)\mathrm{d}v & \int_{-1}^1 v^2f_{U_2}(v)K(v/h-u)\mathrm{d}v
            \end{pmatrix}.
        \end{align*}
        By identical analysis in (\romannumeral 1),  $\inf_{b\in(0,1/2]}\inf_{u\in[0,1]}\|\bM_b^{U_1}(u)\|_{\min} \ge C_K$. Also $\bM_b^{U_2}(u)$ is semi-positive definite. This gives $\inf_{b\in(0,1/2]}\inf_{u\in[0,1]}\|\bM_b(u)\|_{\min} \ge m_f C_K$ and we finish the proof.
    \end{proof}

    \subsection{Lemma \ref{LemmaPartial2} and its proof}

    \begin{lemma}\label{LemmaPartial2}
        Assume $p > 2$. Let $a>C_0\log p$ for some absolute constant $C_0 > 0$. For any real numbers $C_1 > 1, C_2 > 0$, there exists an absolute constant $C$ that only depends on $C_0,C_1,C_2$ such that
        \begin{equation}\label{lemB9-1}
            \int_0^\infty 1\wedge C_1p\exp\{-C_2a\min(\delta,\delta^2)\}\mathrm{d}\delta \le C\pth{\frac{\log p}{a}}^{1/2},
        \end{equation}\begin{equation}\label{lemB9-2}
            \int_0^\infty \delta\qth{1\wedge C_1p\exp\{-C_2a\min(\delta,\delta^2)\}}\mathrm{d}\delta \le C\frac{\log p}{a}.
        \end{equation}
    \end{lemma}
    (\romannumeral 1) We first tackle (\ref{lemB9-1}). We first deal with the case $\exp(-C_2a)C_1p > 1$, namely $a<C_2^{-1}\log(C_1 p)$, then we have \begin{equation*}
        \begin{aligned}
            &\int_0^\infty 1\wedge C_1p\exp\{-C_2a\min(\delta,\delta^2)\}\mathrm{d}\delta\le \int_0^{\frac{\log(C_1p)}{C_2a}} 1\mathrm{d}\delta + \int_{\frac{\log(C_1p)}{C_2a}}^\infty C_1p\exp\{-C_2a\delta\}\mathrm{d}\delta\\
            & \le \frac{\log(C_1p)}{C_2a} + \frac{C_1p}{C_2a}\exp\{-\log(C_1p)\} \le \frac{\log(C_1p)+1}{C_2a}\\
            &\le \frac{\log(C_1)+1}{C_2 a} + \frac{(\log p)^{1/2}}{a^{1/2}C_2C_0^{1/2}}.
        \end{aligned}
    \end{equation*}
    In the last line we use the fact that $a^{1/2}C_0^{-1/2}(\log p)^{-1/2} > 1$ from our condition. Also notice that $\log (C_1) + 1$ is dominated by $(\log p)^{1/2}$ by an absolute constant. Thus we can find large enough constant $C$ such that \eqref{lemB9-1} holds.
    
    Then we deal with the case $\exp(-C_2a)C_1p < 1$, namely $a>C_2^{-1}\log(C_1 p)$. We have \begin{equation*}
        \begin{aligned}
            &\int_0^\infty 1\wedge C_1p\exp\{-C_2a\min(\delta,\delta^2)\}\mathrm{d}\delta\\
            &\le \int_0^{\sqrt{\frac{\log(C_1p)}{C_2a}}} 1\mathrm{d}\delta + \int_{\sqrt{\frac{\log(C_1p)}{C_2a}}}^1 C_1p\exp\{-C_2a\delta^2\}\mathrm{d}\delta + \int_{1}^\infty C_1p\exp\{-C_2a\delta\}\mathrm{d}\delta\\
            & \le \sth{\frac{\log(C_1p)}{C_2a}}^{1/2} + \frac{C_1p}{(C_2 a)^{1/2}}\int_{\sqrt{\log(C_1p)}}^{\sqrt{C_2a}}\exp(-\delta^2)\mathrm{d}\delta + \frac{C_1p}{C_2a}\exp(-C_2a)\\
            & \le \sth{\frac{\log(C_1p)}{C_2a}}^{1/2} + \frac{C_1p}{(C_2 a)^{1/2}}\int_{\sqrt{\log(C_1p)}}^\infty \exp(-\delta^2)\mathrm{d}\delta + \frac{1}{C_2a}\\
            & \le \sth{\frac{\log(C_1p)}{C_2a}}^{1/2} + \frac{1}{C_2a} + \frac{C_1p}{(C_2 a)^{1/2}}\int_0^\infty\exp\{-\log(C_1p)\}\exp\{-\delta^2 - 2\delta\sqrt{\log(C_1p)}\}\mathrm{d}\delta\\
            & \le \sth{\frac{\log(C_1p)}{C_2a}}^{1/2} + \frac{1}{C_2a} + \pth{\frac{\pi/2}{C_2a}}^{1/2}.
        \end{aligned}
    \end{equation*}
    Since we assume $a >C_0\log p\vee 1$, we can find another sufficiently large constant $C$ such that (\ref{lemB9-1}) holds. 

    (\romannumeral 2) For the proof of (\ref{lemB9-2}), we first deal with the case $\exp(-C_2a)C_1p > 1$, namely $a<C_2^{-1}\log(C_1 p)$, then we have \begin{equation*}
        \begin{aligned}
            &\int_0^\infty \delta\qth{1\wedge C_1p\exp\{-C_2a\min(\delta,\delta^2)\}}\mathrm{d}\delta\\
            &\le \int_0^{\frac{\log(C_1p)}{C_2a}} \delta \mathrm{d}\delta + \int_{\frac{\log(C_1p)}{C_2a}}^\infty C_1p\delta\exp\{-C_2a\delta\}\mathrm{d}\delta\\
            & \le \sth{\frac{\log(C_1p)}{C_2a}}^2 + \frac{C_1p\exp\sth{-C_2a\frac{\log(C_1p)}{C_2a}}(1+C_2a\frac{\log(C_1p)}{C_2a})}{C_2^2a^2}\\
            & \le \frac{1+\log(C_1p)+\log^2(C_1p)}{C_2a^2}.
        \end{aligned}
    \end{equation*}
    The last line is dominated by $a^{-1} \log p$. Therefore we can find a sufficiently large constant $C$ such that \eqref{lemB9-2} holds.
    
    Next we deal with the case $\exp(-C_2a)C_1p < 1$, namely $a>C_2^{-1}\log(C_1 p)$. We have \begin{equation*}
        \begin{aligned}
            &\int_0^\infty \delta [1\wedge C_1p \exp\{-C_2a\min(\delta,\delta^2)\}]\mathrm{d}\delta\\
            &\le \int_0^{\sqrt{\frac{\log(C_1p)}{C_2a}}} \delta\mathrm{d}\delta + \int_{\sqrt{\frac{\log(C_1p)}{C_2a}}}^1 C_1p\delta\exp\{-C_2a\delta^2)\}\mathrm{d}\delta + \int_{1}^\infty C_1p\delta\exp\{-C_2a\delta)\}\mathrm{d}\delta\\
            & \le \frac{\log(C_1p)}{C_2a} + \frac{C_1p}{2C_2a}\exp\sth{-C_2a\frac{\log(C_1p)}{C_2a}} + C_1p\frac{\exp(-C_2a)(1+C_2a)}{C_2^2a^2}\\
            & \le \frac{\log(C_1p)+1}{C_2a} + \frac{1+C_2a}{C_2^2a^2}.
        \end{aligned}
    \end{equation*}
    Since we assume $a>C_0\log p $, we can find another sufficiently large constant $C$ such that (\ref{lemB9-2}) holds.

    \subsection{Lemma \ref{LemmaB13} and its proof}

    \begin{lemma}\label{LemmaB13}
        Assume $\bA,\bB\in\eR^{p\times p}$ and for any $j,k\in[p]$, $A_{jk} \ge 0, B_{jk}\ge 0$. Then $\|\bA+\bB\|_2 \ge \|\bA\|_2, \|\bA+\bB\|_2\ge \|\bB\|_2$. 
    \end{lemma}

    \begin{proof}
        For any matrix $\bA$, we have \begin{align*}
            \nth{\bA}_2=\sup_{\bx,\by\in\eR^p,|\bx|_2=|\by|_2=1}\bx^\T\bA\by.
        \end{align*}
        Assume $\|\bA\|_2 = \bx_1^\T\bA\by_1$ with $|\bx_1|_2=|\by_1|_2=1$. Define $\bx_2,\by_2\in\eR^p$ as $x_{2j}=|x_{1j}|,y_{2j}=|y_{2j}|$ for any $j\in[p]$. Then \begin{align*}
        \|\bA+\bB\|_2\ge \bx_2^\T(\bA+\bB)\by_2 \ge \bx_2^\T\bA\by_2 \ge \bx_1^\T\bA\by_1 =\|\bA\|_2.
        \end{align*}
        Here the first inequality comes from the definition of spectral norm of matrix. The second inequality follows from the fact that each entry of $\bB$ and each entry of $\bx_2,\by_2$ are positive. The third inequality uses the fact that $\bA$ have positive entries. Similarly we can prove $\|\bA+\bB\|_2 \ge \|\bB\|_2$.
        
    \end{proof}
        
    \section{Derivations in examples}\label{supsec.exmcal}

   \subsection{Derivation in Example \ref{example1}}\label{SectionC1}

   In this section we calculate the dependence adjusted norms in Example \ref{example1}. Let $\bvarepsilon_0^*(\cdot) = \bvarepsilon_0(\cdot) - \bvarepsilon_0'(\cdot)$, where $\bvarepsilon_0'(\cdot)$ is an i.i.d. copy of $\bvarepsilon_0(\cdot)$. For calculation of $\omega_{t,q}$, from (\ref{Ex1eq1}) we have 
   \begin{equation*}\begin{aligned}
            \omega_{t,q} &= \left\|\left\|\int_0^1 \bA_t(\cdot,v)\bvarepsilon^*_0(v)\dv\right\|_{\cH,\infty}\right\|_q = \left\|\max_{j\in[p]}\left\|\sum_{k=1}^p\int_0^1A_{tjk}(\cdot,v)\varepsilon_{0k}^*(v)\dv\right\|_\cH\right\|_q\\
            & \le \left\|\max_{j\in[p]}\pth{\sum_{k=1}^p\|A_{tjk}\|_\cS}\max_{k\in[p]}\|\varepsilon_{0k}^*\|_\cH\right\|_q \le \|\bA_t\|_{\cS,\infty}\left\|\max_{k\in[p]}\|\varepsilon_{0k}^*\|_\cH\right\|_q \\
            & \le \|\bA_t\|_{\cS,\infty} \sth{\sum_{j=1}^p E\pth{\|\varepsilon_{0j}^*\|_\cH^q}}^{1/q} \le C_q\|\bA_t\|_{\cS,\infty}p^{1/q}\mu_q^{1/q}.
            \end{aligned}\end{equation*}
            For calculation of $\delta_{t,q,j}$, analogously we have \begin{align*}
            \delta_{t,q,j} &= \left\|\left\|\int_0^1 A_{tj\cdot}(\cdot,v)\bvarepsilon_0^*(v)\dv\right\|_{\cH}\right\|_q \le \left\|\left\|\sum_{k=1}^p\int_0^1A_{tjk}(\cdot,v)\varepsilon_{0k}^*(v)\dv\right\|_\cH\right\|_q \\
            &\le \left\|\pth{\sum_{k=1}^p\|A_{tjk}\|_\cS}\max_{k\in[p]}\|\varepsilon_{0k}^*\|_\cH\right\|_q \le \|\bA_{tj\cdot}\|_{\cS,1}\left\|\max_{k\in[p]}\|\varepsilon_{0k}^*\|_\cH\right\|_q\\
            & \le C_q\|\bA_{tj\cdot}\|_{\cS,1}p^{1/q}\mu_q^{1/q}.
        \end{align*} 

   \subsection{Derivation in Example \ref{example2}}\label{SectionC2}

   In this section we calculate  dependence adjusted norms in Example \ref{example2}. From Theorem 3.1 in \cite{Bbosq1},  the stationary solution of (\ref{armodel}) is \begin{align*}
       \bX_t(\cdot) = \sum\limits_{m = 0}^{\infty}\bA^{(m)}(\bvarepsilon_{t-m})(\cdot),
   \end{align*}
   where with a little abuse of notation, we use $\bA$ to denote both the coefficient matrix function $\bA(u,v)$ and the integral operator $\bA(\bbf)(\cdot) = \int_0^1 \bA(\cdot,v)\bbf(v)\mathrm{d}v$, and we use $\bA^{(m)}$ to denote $m$-times composition of integral operator $\bA$. For example, we have $\bA^{(2)}(\bbf)(\cdot)=\iint_{[0,1]^2} \bA(\cdot,u)\bA(u,v)\bbf(v)\mathrm{d}u\dv$. 
   
   Let $\bvarepsilon_0^*(\cdot) = \bvarepsilon_0(\cdot) - \bvarepsilon_0'(\cdot)$, where $\bvarepsilon_0'(\cdot)$ is an i.i.d. copy of $\bvarepsilon_0(\cdot)$. It follows from \Cref{def1} that \begin{align}\label{C2eq2}
        \omega_{m,q} &= \left\|\|\bA^{(m)}(\bvarepsilon_0^*)(\cdot)\|_{\cH,\infty}\right\|_q = \left\|\|(\bA^{(j)})^{(k_1)}\bA^{(k_2)}(\bvarepsilon_0^*)(\cdot)\|_{\cH,\infty}\right\|_q,
    \end{align}
    where $k_1 = \max\{k\in\eN: kj\le m\},k_2 = m-k_1j$. 

    Recall that in Example \ref{example2}, we define $\widetilde{\bA}$ to be a numeric matrix with $\widetilde{A}_{jk}=\|A_{jk}\|_\cS$. In the following we want to prove for any positive integer $k$, $\|\bA^{(k)}\|_{\cS,2} \le \|\widetilde{\bA}^k\|_2$. The case $k=1$ can be directly verified. 
    
    For $k=2$, we have for any $j,k\in[p]$, \begin{equation}\label{C2eq1}
        \begin{aligned}
            \|(\bA^{(2)})_{jk}\|_{\cS} &= \sth{\int_{[0,1]^2} \pth{\int_{0}^1\sum_{l=1}^p A_{jl}(u,w)A_{lk}(w,v)\mathrm{d}w}^2\du\dv}^{1/2}\\
        &\le \sum_{l=1}^p\sth{\int_{[0,1]^2} \pth{\int_{0}^1 A_{jl}(u,w)A_{lk}(w,v)\mathrm{d}w}^2\du\dv}^{1/2}\\
        &\le \sum_{l=1}^p\sth{\int_{[0,1]^2} \pth{\int_{0}^1 A_{jl}(u,w)^2\mathrm{d}w}\pth{\int_{0}^1 A_{lk}(w,v)^2\mathrm{d}w}\du\dv}^{1/2}\\
        &= \sum_{l=1}^p\sth{\pth{\int_{[0,1]^2} A_{jl}(u,w)^2\mathrm{d}w\du}\pth{\int_{[0,1]^2} A_{lk}(w,v)^2\mathrm{d}w\dv}}^{1/2}\\
        &= \sum_{l=1}^p\|A_{jl}\|_\cS\|A_{lk}\|_\cS = \pth{\widetilde{\bA}^2}_{jk}.
        \end{aligned}
    \end{equation}
    The first line follows from the definition. The second line uses Minkowski inequality and the third line applies Cauchy--Schwarz inequality. This shows that each entry of $\widetilde{\bA^{(2)}}$ is smaller than $\widetilde{\bA}^2$. Using Lemma \ref{LemmaB13}, we have $\|\bA^{(2)}\|_{\cS,2} = \|\widetilde{\bA^{(2)}}\|_2 \le \|\widetilde{\bA}^2\|_2$. Induction gives $\|\bA^{(k)}\|_{\cS,2}\le \|\widetilde{\bA}^k\|_2$.
    
    For any curve $\bbf(\cdot)$, $\bA(\bbf)$ is a $p$-dimensional curve. Let $\widetilde{\bbf}$ be a scalar vector such that $\widetilde{f}_j=\|f_j\|_\cH$. Similar to \eqref{C2eq1}, we can show that the $j$-th entry of $\bA(\bbf)$ satisfies \begin{align*}
        \|\sth{\bA(\bbf)}_j\|_\cH & \le \sum_{l=1}^p \|A_{jl}\|_\cS\|f_l\|_\cH \le (\widetilde{\bA}\widetilde{\bbf})_j.
    \end{align*}
    Thus we have $\|\bA(\bbf)\|_{\cH,2} \le \|\bA\|_{\cS,2}\|\bbf\|_{\cH,2}$. Similarly $\|\bA^{(k)}(\bbf)\|_{\cH,2} \le \|\bA^{(k)}\|_{\cS,2}\|\bbf\|_{\cH,2}$. Using \eqref{C2eq2}, we have \begin{align*}
        \omega_{m,q} &= \left\|\|(\bA^{(j)})^{(k_1)}\bA^{(k_2)}(\bvarepsilon_0^*)(\cdot)\|_{\cH,\infty}\right\|_q \le \left\|\|(\bA^{(j)})^{(k_1)}\bA^{(k_2)}(\bvarepsilon_0^*)(\cdot)\|_{\cH,2}\right\|_q\\
        & \le \nth{\nth{(\widetilde{\bA}^j)^{k_1}\widetilde{\bA}^{k_2}}_2\nth{\bvarepsilon_0^*}_{\cH,2}}_q \le c'c^{m/j-1}\nth{\pth{\sum_{j=1}^p\|\varepsilon^*_{0j}\|_\cH^2}^{1/2}}_q\\
        & = c'c^{m/j-1}\nth{\sum_{j=1}^p\|\varepsilon^*_{0j}\|_\cH^2}_{q/2}^{1/2} \le c'c^{m/j-1}p^{1/2}\|\|\varepsilon_{0j}^*\|_\cH\|_q \le C_qc'c^{m/j-1}p^{1/2}\mu_q^{1/q}.
    \end{align*}
    Define $f(\alpha) = \sup_{m\ge 0}(m+1)^\alpha c^{m/j-1}/(1-c^{1/j})$ and some calculation gives (\ref{ex2eq2}).

    \subsection{Proof of Example \ref{example4}}

    For any continuously differentiable function $f(u)$ defined on the interval $[0,1]$, $|f(u)|$ is a continuous function defined on $[0,1]$ so it must achieve minimum and maximum at some point. Assume it achieves minimum at $u_0$ and achieves maximum at $u_1$. Then $|f(u_0)|\le \|f\|_\cH$. So we have \begin{align*}
        \sup_{u\in[0,1]}|f(u)| &= |f(u_1)| = |f(u_0)| + |f(u_1)|-|f(u_0)| \\
        &\le \|f\|_\cH + \left|\int_{u_0}^{u_1} |\partial_u f(u)|\mathrm{d}u\right| \le \|f\|_\cH + \int_{0}^{1} |\partial_u f(u)|\mathrm{d}u\\
        &\le \|f\|_{\cH} + \|\partial_u f\|_{\cH}.
    \end{align*}
    In the last line we use Jensen's inequality. This implies $|\bX_t^*|_\infty \le \|\bX_t\|_{\cH,\infty} + \|\partial_u\bX_t\|_{\cH,\infty}$. Thus $E(|\bX_t^*|_\infty^2) \le 2E(\|\bX_t\|_{\cH,\infty}^2) + 2E(\|\partial_u\bX_t\|_{\cH,\infty}^2) \lesssim E(\|\bX_t\|_{\cH,\infty}^2)$. Similar argument gives $E(|\bX_t^{(2)*}|_\infty^2)\lesssim E(\|\bX_t\|_{\cH,\infty}^2)$.

    \subsection{Proof of Example \ref{example3}}

    We prove $E(|\bX_t^{*}|_\infty^2) \lesssim E(\|\bX_t\|_{\cH,\infty}^2)$ and the other claim $E(|\bX_t^{(2)*}|_\infty^2) \lesssim E(\|\bX_t\|_{\cH,\infty}^2)$ follows similarly. For the right hand side we have \begin{equation}\label{Eq2C3}
        \begin{aligned}
            E(\|\bX_t\|_{\cH,\infty}^2) &= E\pth{\max_{j\in[p]}\|X_{tj}\|_\cH^2} = E\pth{\max_{j\in[p]}\nth{\sum_{l=1}^\infty \xi_{tjl}\vartheta_{tjl}\psi_{tjl}}_\cH^2} = E\pth{\max_{j\in[p]}\sum_{l=1}^\infty \xi_{tjl}^2\vartheta_{tjl}^2}\\
            & \ge CE\pth{\max_{j\in[p]}\xi_{tj1}^2}.
        \end{aligned}
    \end{equation}
    For the left hand side we have \begin{equation}\label{Eq3C3}
        \begin{aligned}
        E(|\bX_t^{*}|_\infty^2) &= E\qth{\max_{j\in[p]}\sth{\sup_{u\in[0,1]}\left|\sum_{l=1}^\infty \xi_{tjl}\vartheta_{tjl}\psi_{tjl}(u)\right|}^2}\\
        & \le E\qth{\max_{j\in[p]}\sth{\sum_{l=1}^\infty\sup_{u\in[0,1]}\left|\psi_{tjl}(u)\right|l^{-\tilde\delta_1-1} \xi_{tjl}\vartheta_{tjl}l^{\tilde\delta_1+1}}^2}\\
        & \le E\qth{\max_{j\in[p]}\sth{\sum_{l=1}^\infty\sup_{u\in[0,1]}\left|\psi_{tjl}(u)\right|^2l^{-2\tilde\delta_1-2}}\sth{\sum_{l=1}^\infty \xi_{tjl}^2\vartheta_{tjl}^2l^{2\tilde\delta_1+2}}}\\
        & \lesssim E\pth{\max_{j\in[p]}\sum_{l=1}^\infty \xi_{tjl}^2\vartheta_{tjl}^2l^{2\tilde\delta_1+2}} \lesssim E\pth{\max_{j\in[p]}\xi_{tj1}^2}\sum_{l=1}^\infty\vartheta_{tjl}^2l^{2\tilde\delta_1+2}\\
        & \lesssim E\pth{\max_{j\in[p]}\xi_{tj1}^2},
    \end{aligned}
    \end{equation}
    where in the third line we use Cauchy--Schwarz inequality, and in the last line we use $\vartheta_{tjl}\asymp\tilde\delta^{-l},\tilde\delta > 1$ and $\sup_{u\in[0,1]}|\psi_{tjl}(u)|\asymp l^{\tilde\delta_1}$. Combining \eqref{Eq2C3} and \eqref{Eq3C3}, we show that $E(|\bX_t^{*}|_\infty^2) \lesssim E(\|\bX_t\|^2_{\cH,\infty})$.

    \section{Some additional results of covariance function and spectral density function estimation}\label{sec.suppD}

    \begin{proposition}[Convergence rate of spectral density estimation at a fixed frequency]\label{propS.3}
        Assume all conditions for Theorem \ref{thm2} hold. Then at a fixed frequency $\theta$, we have \begin{align}\label{propS3eq1}
            P\left\{\Big\|\widehat{\bbf}_\theta-E(\widehat{\bbf}_\theta)\Big\|_{\cS,\max} > \cM_{q,\alpha}^\bX x\right\} & \le C_{q,\alpha}x^{-q/2}(\log p)^{5q/4}F^*_{n,m_0}+C_\alpha p^2\exp\left(-\frac{C_\alpha' x^2n}{m_0}\right),\notag\\
            P\left\{\Big\|\widehat{\bbf}_\theta-E(\widehat{\bbf}_\theta)\Big\|_{\cS,\max} > \bPhi_{q,\alpha}^\bX x\right\} & \le C_{q,\alpha}x^{-q/2}F^*_{n,m_0}+C_\alpha \exp\left(-\frac{C_\alpha' x^2n}{m_0}\right).
        \end{align}
        The above concentration inequality gives that, at a fixed frequency $\theta$, \begin{align*}
            \left\|\widehat{\bbf}_\theta-\bbf_\theta\right\|_{\cS,\max} = O_P(\cH_6),
        \end{align*}
        where \begin{align}\label{RPropS.1}
            \cH_6 & = R(m_0)+\bPhi_{q,\alpha}^\bX C_\bX'\left\{(F^*_{n,m_0})^{2/q}(\log p)^{5/2} +\sqrt{\frac{\log(p\vee m_0)}{n}}\right\},\\
            C_\bX' &= \min\sth{1, \frac{\cM_{q,\alpha}^\bX (\log p)^{5/2}}{\bPhi_{q,\alpha}^\bX p^{4/q}}}, \notag
        \end{align}
        and \(F^*_{n,m} = n^{1-q/2}m^{q/2-1}\) (resp., \(n^{1-q/2}m^{q/2-1} + n^{-q/4-\alpha q/2}m^{q/4}\)) if \(\alpha > 1/2 - 2/q\) (resp., \(\alpha \le 1/2-2/q\)).
    \end{proposition}

    \begin{proof}
    Similar to (\ref{pfthm2-1}) we have \begin{equation*}
            P\left\{2\pi n \left\|E_0\left(\widehat{\bbf}_{\theta}\right)\right\|_{\cS,\max} > x\right\} \le \sum_{\omega = 1}^{2}P\left[\left\|E_0\left\{\bQ_\omega\left(\theta\right)\right\}\right\|_{\max} >x/2\right].
        \end{equation*}
        The above equation has an elimination of factor $m_0$ compared with \eqref{pfthm2-1} since we do not need to take maximum with respect to $\theta$ and a step of Bonferroni correction can be saved. Now we use Lemma \ref{nonGaussianConcen} and \ref{lm:elementwisecon1} to obtain \begin{equation*}
            \begin{aligned}
                P\left[\left\|E_0\left\{\bQ_\omega\left(\theta\right)\right\}\right\|_{\cS,\max} > x\right] \le & C_{q,\alpha}x^{-q/2}l^{5q/4}\|\|\bX_1\|_{\cH,\infty}\|_{q,\alpha}^qF_{n,m_0}' \\
                & + C_\alpha p^2\exp\left\{-\frac{x^2}{C_\alpha\left(\bPhi_{4,\alpha}^\bX\right)^2nm_0}\right\},\\
                P\left[\left\|E_0\left\{Q_{\omega jk}\left(\theta\right)\right\}\right\|_{\cS,\max} > x\right] \le & C_{q,\alpha}x^{-q/2}\|\|X_{1j}\|_{\cH,\infty}\|_{q,\alpha}^q\|\|X_{1k}\|_{\cH,\infty}\|_{q,\alpha}^qF_{n,m_0}' \\
                &+ C_\alpha \exp\left\{-\frac{x^2}{C_\alpha\left(\bPhi_{4,\alpha}^\bX\right)^2nm_0}\right\}.
            \end{aligned}
        \end{equation*}
        Here $F'_{n,m_0}=nm_0^{q/2-1}$ (resp., $nm_0^{q/2-1}+n^{q/4-\alpha q/2}m_0^{q/4}$) if $\alpha > 1/2-2/q$ (resp., $\alpha \le 1/2-2/q$). Equation (\ref{propS3eq1}) comes from $\bPhi_{4,\alpha}^\bX \le \cM_{q,\alpha}^\bX , \|\|\bX_1\|_{\cH,\infty}\|_{q,\alpha}^2 = \cM_{q,\alpha}^\bX $, some elementary calculation and Bonferroni inequality.
        \color{black}
        \end{proof}

     \section{Additional results of sparse spectral density operator estimation}\label{Sec.Add_Sparse}

     We define the class of approximately sparse spectral density operators (at a fixed frequency).

    \begin{definition}\label{DefS.3}
        Let $0\le q^* < 1$. We define the following class $\cC\{q^*,s_0(p),\theta\}$ as approximately sparse spectral density function at frequency $\theta$ if $$\cC\{q^*,s_0(p),\theta\} = \left\{\bbf_{\theta}:\bbf_\theta\succeq 0,\max\limits_{k\in[p]}\sum_{j = 1}^p\|f_{\theta, jk}\|^{q^*}_\mathcal{S} \le s_0(p)\right\}.$$
    \end{definition}

    In the following, we define the functional thresholding operators (at a fixed frequency).

    \begin{definition}
        We define that $s_\lambda^*:\eS\rightarrow\eS$ is a functional thresholding operator (at a fixed frequency) if it satisfies the following three conditions:

            (\romannumeral 1) $\|s_\lambda^*(Z)\|_{\cS}\le c\|Y\|_{\cS}$ for all $Z,Y\in\eS$ satisfying $\|Z-Y\|_{\cS}\le \lambda$ and some $c > 0$.

            (\romannumeral 2) $\|s_\lambda^*(Z)\|_{\cS} = 0$ for all $\|Z\|_{\cS} \le \lambda$.

            (\romannumeral 3) $\|s_\lambda^*(Z)-Z\|_{\cS}\le \lambda$ for all $Z\in\eS$.
            
    \end{definition}


    To estimate $\bbf_\theta$ in \Cref{DefS.3}, we use the following threshold estimator $$\widehat{\bbf}_{\theta}^{\cT*} = \left\{\hat{f}_{\theta, jk}^{\cT*}\right\}_{j,k\in[p]}\quad\mbox{with} \quad\hat{f}_{\theta, jk}^{\cT*} = s_\lambda^*\left(\hat{f}_{\theta, jk}\right). $$

    \begin{proposition}\label{propS.2}
    Suppose that all conditions in Theorem \ref{thm2} hold and $\lambda^{-1}\cH_6\rightarrow\infty$, where $\cH_6$ is defined in \Cref{RPropS.1}. Then uniformly over $\cC\{q^*,s_0(p),\theta\}$, \begin{align*}
                &\|\widehat{\bbf}_{\theta}^{\cT*}-\bbf_{\theta}\|_{\cS,1}  = \max\limits_{k\in[p]}\sum_{j = 1}^p \|\hat{f}_{\theta, jk}^{\cT*}-f_{\theta, jk}\|_{\cS}\nonumber = O_P\left\{s_0(p)\lambda^{1-q^*}\right\}.
            \end{align*}
    \end{proposition}

    \begin{proof}
        Let $\Omega_{n2} = \left\{\|\hat{f}_{\theta, jk}-f_{\theta, jk}\|_\cS\le \lambda\right\}$. Using similar argument as in the proof of Theorem \ref{thm5}, we have $\sum_{k = 1}^p\|\hat{f}_{\theta, jk}^{\cT*}-f_{\theta, jk}\|_\cS \lesssim s_0(p)\lambda^{1-q^*}$ under $\Omega_{n2}$. We have $1-P\{\Omega_{n2}\} = o(1)$ given $\lambda^{-1}\cH_6\rightarrow\infty$ and we finish the proof.
    \end{proof}

    We define the class of truly sparse spectral density operators at frequency $\theta$ to be \begin{align*}
        \cC\{s_0(p),\theta\} &= \left\{\bbf_{\theta}:\bbf_{\theta}\succeq 0, \max\limits_{k\in[p]}\sum_{j = 1}^p I(\|f_{\theta, jk}\|_\cS\ne 0)\le s_0(p)\right\}.
    \end{align*}
        The support at frequency $\theta$ is defined as $\mathrm{supp}^*(\bbf_\theta) = \{(j,k):\|f_{\theta, jk}\|_\cS > 0\}$. We present the following support recovery result.

        \begin{proposition}
            Suppose all conditions in Theorem \ref{propS.2} hold and $\|f_{\theta, jk}\|_\cS\ge \lambda$ for all $(j,k)\in\mathrm{supp}(\bbf_\theta)$. Then we have \begin{align*}
                \inf\limits_{\bbf_\theta\in\cC(s_0(p),\theta)}P\left\{\mathrm{supp}^*(\widehat{\bbf}_{\theta}^{\cT*})=\mathrm{supp}^*(\bbf_{\theta})\right\}\rightarrow 1\quad \mathrm{as}\quad n\rightarrow\infty.
            \end{align*}
        \end{proposition}

        The proof is similar to the proof of Theorem \ref{thm6} and thus omitted here.

\end{document}